\documentclass[a4paper, 12pt]{article}
\usepackage[colorlinks=true, linkcolor=blue, citecolor=blue, urlcolor=blue]{hyperref}
\usepackage[english]{babel}
\usepackage[utf8]{inputenc}
\usepackage[backend=bibtex,
style=trad-abbrv]{biblatex}
\usepackage{geometry}

\usepackage{amsmath, amssymb, amsthm}
\usepackage{mathtools}
\usepackage{titling}
\usepackage{float}
\usepackage{mathrsfs}
\usepackage{dsfont}
\usepackage{upgreek}
\usepackage{tcolorbox}
\usepackage{xcolor}
\usepackage{extarrows}
\usepackage{caption}
\usepackage{array}
\usepackage{subcaption}
\usepackage{stmaryrd}
\usepackage{tikz}
\usepackage{tikz-3dplot}
\usetikzlibrary{calc, cd, patterns}
\usepackage[ruled]{algorithm2e}
\usepackage{multicol}
\usepackage{csquotes}
\usepackage{empheq}
\usepackage{enumitem}
\usepackage{titlesec}
\usepackage{titletoc}
\usepackage{aliascnt}
\usepackage{cleveref}
\AtEveryBibitem{
  \clearname{editor}
  \clearfield{url}
  \clearfield{doi}
  \clearfield{issn}
  \clearfield{url}
  \clearfield{urlyear}
}
\makeatletter
\DeclareFontFamily{OMX}{MnSymbolE}{}
\DeclareSymbolFont{MnLargeSymbols}{OMX}{MnSymbolE}{m}{n}
\SetSymbolFont{MnLargeSymbols}{bold}{OMX}{MnSymbolE}{b}{n}
\DeclareFontShape{OMX}{MnSymbolE}{m}{n}{
    <-6>  MnSymbolE5
  <6-7>  MnSymbolE6
  <7-8>  MnSymbolE7
  <8-9>  MnSymbolE8
  <9-10> MnSymbolE9
  <10-12> MnSymbolE10
  <12->   MnSymbolE12
}{}
\DeclareFontShape{OMX}{MnSymbolE}{b}{n}{
    <-6>  MnSymbolE-Bold5
  <6-7>  MnSymbolE-Bold6
  <7-8>  MnSymbolE-Bold7
  <8-9>  MnSymbolE-Bold8
  <9-10> MnSymbolE-Bold9
  <10-12> MnSymbolE-Bold10
  <12->   MnSymbolE-Bold12
}{}

\let\llangle\@undefined
\let\rrangle\@undefined
\DeclareMathDelimiter{\llangle}{\mathopen}%
                    {MnLargeSymbols}{'164}{MnLargeSymbols}{'164}
\DeclareMathDelimiter{\rrangle}{\mathclose}%
                    {MnLargeSymbols}{'171}{MnLargeSymbols}{'171}
\makeatother

\DeclareMathAlphabet{\mcal}{OMS}{zplm}{m}{n}
\DeclareFontFamily{U}{stix2bb}{\skewchar\font127 }
\DeclareFontShape{U}{stix2bb}{m}{n} {<-> stix2-mathbb}{}
\DeclareMathAlphabet{\mathbbl}{U}{stix2bb}{m}{n}
\newcommand{\norm}[1]{\left\lVert#1\right\rVert}

\newcommand{\tnorm}[1]{{\left\vert\kern-0.25ex\left\vert\kern-0.25ex\left\vert #1 
    \right\vert\kern-0.25ex\right\vert\kern-0.25ex\right\vert}}
\newcommand{\abs}[1]{\left|#1\right|}
\newcommand{\brac}[1]{\left(#1\right)}
\newcommand{\cbrac}[1]{\left\{#1\right\}}

\newcommand{\sqbrac}[1]{\left[#1\right]}
\newcommand{\floor}[1]{\left\lfloor#1\right\rfloor}

\newcommand{\lrangle}[1]{\left\langle #1\right\rangle}
\newcommand{\llrrangle}[1]{\left\llangle #1\right\rrangle}
\newcommand{\KL}[2]{\mathrm{KL}\brac{#1 | #2}}
\newcommand{\OT}{\mathrm{OT}}

\newcommand{\grad}{\nabla}

\newcommand{\supp}[1]{\mathrm{supp}\brac{#1}}
\newcommand{\vspan}{\mathrm{span}}

\newcommand{\id}{\mathrm{Id}}
\newcommand{\diag}[1]{\mathrm{diag}\brac{#1}}
\newcommand{\dom}[1]{\mathrm{dom}\brac{#1}}

\newcommand{\interior}[1]{\mathrm{int}\brac{#1}}

\renewcommand{\div}{\mathrm{div}}
\renewcommand{\Bar}{\overline}

\newcommand{\inv}{^{-1}}
\newcommand{\pinv}{^\dagger}

\newcommand{\txtstar}{\textasteriskcentered{}}

\DeclareMathOperator*{\argmin}{arg\,min}
\makeatletter
\let\olddot\Dot
\makeatother
\renewcommand{\Dot}[1]{\smash{\olddot{#1}}}

\newcommand{\sbar}[1]{\smash{\Bar{#1}}}
\renewcommand{\tau}{\uptau}
\newcommand{\dd}{\mathrm d}

\newtheoremstyle{sl}
{3pt}
{3pt}
{\slshape}
{}
{\bfseries}
{.}
{.5em}
{}
\theoremstyle{sl}
\newtheorem{theorem}{Theorem}[section]

\newaliascnt{lemma}{theorem}
\newaliascnt{proposition}{theorem}
\newaliascnt{corollary}{theorem}
\newaliascnt{definition}{theorem}
\newaliascnt{remark}{theorem}
\newaliascnt{example}{theorem}

\newtheorem{lemma}[lemma]{Lemma}
\aliascntresetthe{lemma}

\newtheorem{proposition}[proposition]{Proposition}
\aliascntresetthe{proposition}

\newtheorem{corollary}[corollary]{Corollary}
\aliascntresetthe{corollary}

\newtheorem{definition}[definition]{Definition}
\aliascntresetthe{definition}

\newtheorem*{assumption}{Assumptions}

\theoremstyle{remark}
\newtheorem{remark}[remark]{Remark}
\aliascntresetthe{remark}

\aliascntresetthe{example}

\floatstyle{boxed}
\newfloat{listing}{H}{loi}
\floatname{listing}{Pseudocode}
\newlist{thmenum}{enumerate}{1}
\setlist[thmenum]{label=(\alph*), ref=(\alph*)}

\numberwithin{equation}{section}
\crefrangeformat{equation}{(#3#1#4--#5#2#6)}
\crefmultiformat{equation}{(#2#1#3)}%
{ and~(#2#1#3)}{, (#2#1#3)}{ and~(#2#1#3)}
\crefformat{equation}{(#2#1#3)}
\crefname{theorem}{Theorem}{Theorems}
\crefname{lemma}{Lemma}{Lemmas}
\crefname{proposition}{Proposition}{Propositions}
\crefname{corollary}{Corollary}{Corollaries}
\crefname{algorithm}{Algorithm}{Algorithms}
\crefname{definition}{Definition}{Definitions}
\crefname{assumption}{Assumption}{Assumptions}
\crefname{remark}{Remark}{Remarks}
\crefname{example}{Example}{Examples}
\crefname{pcode}{Pseudocode}{Pseudocodes}
\newcommand{\refthmitem}[2]{\cref{#1} \ref{#1:#2}}
\crefformat{figure}{Figure #2#1#3}
\crefformat{section}{Section #2#1#3}
\crefformat{subsection}{Section #2#1#3}
\crefformat{appendix}{Appendix #2#1#3}
\newcommand{\theaffiliations}{}
\newcommand{\affiliations}[1]{\renewcommand{\theaffiliations}{#1}}
\let\oldmaketitle\maketitle
\renewcommand{\maketitle}{
\renewcommand{\thefootnote}{\fnsymbol{footnote}}
\oldmaketitle \theaffiliations
\renewcommand{\thefootnote}{\arabic{footnote}}}

\title{Gradient Flows of Potential Energies \\ in the Geometry of Sinkhorn Divergences}
\author{Mathis Hardion\protect\footnotemark[1]
~and Hugo Lavenant\protect\footnotemark[2]}
\affiliations{\footnotetext[1]{Gustave Eiffel University, LIGM, Champs-sur-Marne, France.}\footnotetext[2]{Bocconi University, Department of Decision Sciences and BIDSA, Milan, Italy.}

}
\date{\today}

\newcommand{\X}{\mathcal X}
\newcommand{\PX}{\mathcal{P(X)}}
\newcommand{\CX}{\mathcal{C(X)}}
\newcommand{\MX}{\mathcal{M(X)}}
\newcommand{\MoX}{\mathcal M_0(\X)}
\newcommand{\MpX}{\mathcal M_+(\X)}
\newcommand{\tgtspace}[1]{\mathcal H_{#1}^{*, 0}}
\newcommand{\R}{\mathds R}

\begin{document}

\maketitle

\begin{abstract}
We analyze the gradient flow of a potential energy in the space of probability measures when we substitute the optimal transport geometry with a geometry based on Sinkhorn divergences, a debiased version of entropic optimal transport. This gradient flow appears formally as the limit of the minimizing movement scheme, a.k.a. JKO scheme, when the squared Wasserstein distance is substituted by the Sinkhorn divergence. We prove well-posedness and stability of the flow, and that, in the long term, the energy always converges to its minimal value. The analysis is based on a change of variable to study the flow in a Reproducing Kernel Hilbert Space, in which the evolution is no longer a gradient flow but described by a monotone operator. Under a restrictive assumption we prove the convergence of our modified JKO scheme towards this flow as the time step vanishes. We also provide numerical illustrations of the intriguing properties of this newly defined gradient flow.

\medskip


\medskip

\end{abstract}



\section{Introduction}\label{section:intro}

Gradient flows define evolution equations displaying an interplay between a functional $E$, defined on a space $\mcal Y$, and a distance $\mathsf{d}$ on $\mcal Y$. In Euclidean spaces they take the form
$\dot{y}_t = - \nabla E (y_t)$, where $-\nabla E$ gives the direction of steepest descent measured with respect to the (Euclidean) metric $\mathsf{d}$ on $\mcal Y$. 

A central application of this theory is when $\mcal Y = \PX$ is the space of probability distributions over a base space $\X$ endowed with the Monge-Kantorovich distance, a.k.a. Wasserstein distance, coming from optimal transport. Indeed, since the work of Jordan, Kinderlehrer and Otto \cite{JKO} it is understood that many meaningful Partial Differential Equations (PDEs), such as diffusion equations, can be interpreted as gradient flows in the geometry of optimal transport. Whereas the initial motivation was to build the natural geometry for equations coming from physics \cite{ottoporousmedium,adams2011large}, this theory has found numerous applications as a tool for the analysis of PDEs \cite{kellersegel,matthes2009family,aguehdegenerateparabolic}, to design numerical schemes to solve these PDEs \cite{peyre15,benamou2016augmented}, as a modeling tool \cite{crowdmotion}, and more recently to propose and study learning algorithms \cite{chizat2018global,lambert2022variational}. We will not attempt to provide an exhaustive list of references, and instead refer to the textbooks and surveys \cite{AGS,otam,santambrogio17,computationalot,chewi2024statistical} and references therein for an overview of breadth of this theory and its numerous applications. 

In this work we change the geometry, that is, the metric, by using entropic optimal transport \cite{cuturi13,computationalot,NutzIntroEOT} instead of plain optimal transport. Indeed the recent work \cite{RGSD} by the second author and co-authors explained how Sinkhorn divergences, a debiased version of entropic optimal transport, can be used to define a Riemannian-like geometry on the space of probability distributions. To motivate the introduction of this geometry, we first explain how one can use entropic optimal transport and Sinkhorn divergences in the canonical variational scheme used to define Wasserstein gradient flows, the so-called JKO scheme. This will lead us to the scheme \eqref{eq:SJKO} below, which is the starting point of our analysis.

\paragraph{Wasserstein space and gradient flows}

On a compact space \(\X\) with a continuous ground cost \(c: \X\times \X \to \R_+\), Optimal Transport (OT) gives a way to lift this cost to the space \(\PX\) of probability measures on \(\X\). The OT cost between \(\mu, \nu \in \PX\) is defined as
\begin{equation}\label{eq:MKdist}
\OT_0(\mu, \nu) \coloneqq \min_{\pi \in \Pi(\mu, \nu)}\iint_{\X\times \X} c(x, y)\dd\pi(x, y)
\end{equation}
where \(\Pi(\mu,\nu) \coloneqq \cbrac{\pi \in\mcal{P(X\times X)} \ : \ (\mathrm{proj}_1)_\sharp \pi = \mu, (\mathrm{proj}_2)_\sharp\pi = \nu}\) is the set of couplings with respective marginals \(\mu, \nu\), with \((\mathrm{proj}_i)_\sharp \pi\) the pushforward of \(\pi\) by the projection \(\mathrm{proj}_i\) onto the \(i\)th component of the product space \(\X\times \X\).
When \(\X\) is endowed with a distance \(\mathsf d\) and the ground cost is taken as \(c = \mathsf d^2\), the square root of the OT cost defines a distance on \(\PX\), called the Monge-Kantorovitch or Wasserstein distance and denoted \(W_2\coloneqq \sqrt{\OT_0}\). We refer to the textbooks \cite{villani2009,otam,AGS,computationalot} for an exposition of this theory.

In the Wasserstein space \(\brac{\PX, W_2}\), it is possible to define a notion of gradient flow thanks to the minimizing movement scheme, named JKO after the authors of \cite{JKO}.
For a time step \(\tau >0\) and a lower bounded and lower semi-continuous energy functional \(E:\PX \to \R\), it consists in starting from $\bar{\mu}_0 \in \PX$ and defining a sequence $(\mu^\tau_k)_{k \geq 0}$ recursively via $\mu^\tau_0 = \bar{\mu}_0$ and
\begin{equation}
\mu_{k+1}^\tau \in \argmin_{\mu\in\PX} \; E(\mu) + \frac1{2\tau}W^2_2\brac{\mu,\mu_k^\tau}.
\tag{JKO}
\label{eq:JKO}
\end{equation}
In the case where we replace \(\PX\) by a Euclidean space and \(W_2\) by the Euclidean distance, we recover the implicit Euler discretization of $\dot{x}_t = - \nabla E(x_t)$.
Under some conditions on the energy $E$, the discrete sequence $(\mu^\tau_k)_{k \geq 0}$, suitably interpolated in time (e.g. with a piecewise constant interpolation), converges to a continuous curve $(\mu_t)_{t \geq 0}$ which solves in a metric sense the gradient flow equation $\Dot \mu_t = - \nabla_{W_2}E(\mu_t)$.  
We refer to the surveys \cite{santambrogio17,usersguideOT} and the main reference book \cite{AGS} for accounts of the theory of gradient flows in metric spaces and particularly the Wasserstein space. 

With $E : \mu \mapsto \int V \dd \mu + \int \mu \log \mu$ for a smooth potential \(V:\X \to \R\), the corresponding gradient flow, that is, the limit when $\tau \to 0$ of the curve $\mu^\tau$ solving~\eqref{eq:JKO}, solves the Fokker Planck equation
\begin{equation}
\label{eq:FP}
\frac{\partial \mu_t}{\partial t} = \div\brac{\mu_t\grad V} + \Delta \mu_t.
\end{equation}

\paragraph{Entropic regularization of the JKO scheme}

For a regularization parameter \(\varepsilon > 0\), the Entropic Optimal Transport (EOT) problem, or Schrödinger problem, is the entropic regularization of the Monge-Kantorovitch problem, defining a cost
\begin{equation}\label{eq:OTeps}
\OT_\varepsilon(\mu, \nu) \coloneqq \min_{\pi \in \Pi(\mu, \nu)}\iint_{\mcal X\times \mcal X} c(x, y)\dd\pi(x, y) + \varepsilon\KL{\pi}{\mu\otimes\nu}
\end{equation}
where the KL divergence is given by
\begin{equation*}
\KL{\pi}{\gamma} \coloneqq \begin{cases}
\displaystyle{\iint_{\mcal X\times \mcal X}\log\brac{\frac{\dd\pi}{\dd\gamma}}\dd\pi} & \text{ if }\pi \ll \gamma,\\
+\infty &\text{ otherwise}.
\end{cases}
\end{equation*}
The loss \(\OT_\varepsilon\) converges to \(\OT_0\) as \(\varepsilon \to 0\) \cite{schrodingertomk}, but is also interesting \emph{per se} for $\varepsilon$ fixed.
On top of better numerical tractability \cite{computationalot,schmitzer2019stabilized}, it has smaller statistical complexity mitigating the curse of dimensionality \cite{genevay2019sample,menasampleclt} and is smoother with efficiently computed gradients \cite{genevay18a}. We refer to \cite{leonard14survey,NutzIntroEOT} for expositions of the theory of EOT, and to \cite[Chapter 4]{computationalot} for computational considerations.
 
Given these practical advantages, it is not surprising that it was proposed to replace $W_2^2 = \OT_0$ by \(\OT_\varepsilon\) in the scheme \eqref{eq:JKO}. Specifically the work \cite{peyre15} introduced the recursive scheme
\begin{equation}
\mu_{k+1}^\tau \in \argmin_{\mu\in\mcal{P(X)}} \; E(\mu) + \frac1{2\tau}\OT_\varepsilon(\mu, \mu_k^\tau),
\tag{EJKO}
\label{eq:EJKO}
\end{equation} 
we call it the ``Entropic JKO scheme''. In \cite{peyre15}, the author recasts \eqref{eq:EJKO} as a proximal stepping with respect to the KL divergence and takes advantage of this structure to devise an efficient computational scheme to solve it. 

The main drawback is that \(\OT_\varepsilon\) induces a bias, as generally speaking \(\OT_\varepsilon(\mu, \mu) > 0\) and more importantly \(\mu \notin \argmin \{ \OT_\varepsilon(\mu, \nu) : \nu\in\PX \}\). Thus global minima of $E$ are not stable by the iterates of~\eqref{eq:EJKO}, and, when looking at the limit \(\tau \to 0\) for a fixed \(\varepsilon>0\), one cannot recover a continuous curve as mentioned in the conclusion of \cite{peyre15}. 
A natural way to fix this issue is to take $\varepsilon \to 0$ together with $\tau \to 0$. With $E(\mu) = \int V \dd \mu + \int \mu \log \mu$ it was studied in \cite{carlier2017convergence}: the authors show that solutions to the scheme~\eqref{eq:EJKO} also converge to the solution of the Fokker-Planck equation~\eqref{eq:FP} when both \(\varepsilon\) and the time step \(\tau\) jointly vanish at a suitable rate with the necessary condition \(\varepsilon = o(\tau)\).
This result was recently generalized by \cite{baradat2025}, which shows that, when taking \( \varepsilon / \tau \to \alpha \in \R_+\), a linear diffusion term proportional to \(\alpha\) is added to the limiting equation\footnote{Note however that the results of \cite{baradat2025} apply to a slightly different version of \eqref{eq:EJKO}, as the reference measure in the $\mathrm{KL}$ penalization in~\eqref{eq:OTeps} is not $\mu \otimes \nu$ but the Lebesgue measure.}.

In this work, we keep $\varepsilon > 0$ fixed once and for all, it will not be vanishing with $\tau$: one motivation is numerical, as the computational complexity of solving~\eqref{eq:EJKO} degrades as $\varepsilon \to 0$. 
In order to still have a well-defined limit, we modify the distance: instead of naively substituting $\OT_0$ by $\OT_\varepsilon$, we will plug in its debiased version, a.k.a. the Sinkhorn divergence.

\paragraph{The Sinkhorn divergence}

The Sinkhorn divergence is the debiased version of the entropic optimal transport problem, defined in \cite{genevay18a} (see also \cite{Ramdas,salimans2018improving}) as
\begin{equation*}
S_\varepsilon(\mu, \nu) \coloneqq \OT_\varepsilon(\mu, \nu) - \frac12\OT_\varepsilon(\mu, \mu) - \frac12 \OT_\varepsilon(\nu, \nu).
\end{equation*}
By construction $S_\varepsilon(\mu,\mu) = 0$ for all $\mu$, and \cite{feydy19} proved that it can serve as a proper loss function.

\begin{theorem}[{{\cite[Theorem 1]{feydy19}}}]\label{thm:feydythm1}
Assume that \(c:\mcal X\times \mcal X \to \R_+\) is a (jointly) Lipschitz cost function such that \(k_c: (x, y) \mapsto \exp(-c(x, y) /\varepsilon )\) is a positive definite universal kernel. 

Then \(S_\varepsilon(\mu,\nu) \geq 0\) with equality if and only if $\mu = \nu$, and \(S_\varepsilon(\mu,\nu) \to 0\) if and only if $\mu \to \nu$ weakly-\txtstar{}. Moreover, $S_\varepsilon$ is convex in each of its input variables.
\end{theorem}

Reminders on positive definite kernels and Reproducing Kernel Hilbert Spaces (RKHS) are given in \cref{appendix:RKHS}, and we also refer to \cite{micchelliuniversalkernels, berlinet2011rkhs} for more details. The assumptions of the theorem are satisfied in the usual case where the cost \(c\) is the squared Euclidean distance on \(\mcal X\) a compact subset of $\R^d$. However $S_\varepsilon$ is not the square of a distance, contrary to $S_0 = \OT_0 = W_2^2$ \cite[Theorem 7.1]{RGSD}.

We debias the scheme \eqref{eq:EJKO} and obtain the Sinkhorn-JKO scheme defined as follows: starting from $\mu^\tau_0 = \bar{\mu}_0 \in \PX$, for $k \geq 0$, 
\begin{equation}
\tag{SJKO}
\label{eq:SJKO}
\mu_{k+1}^\tau \in \argmin_{\mu\in\mcal{P(X)}} \; E(\mu) + \frac1{2\tau}S_\varepsilon(\mu, \mu_k^\tau).
\end{equation}
By \cref{thm:feydythm1}, as long as $E$ is convex and lower semi-continuous, each iteration of~\eqref{eq:SJKO} is a well-posed convex problem, which is numerically tractable as we illustrate in this work.

\begin{center}
\fbox{\begin{minipage}{0.9\textwidth}
The aim of the present work is to understand the evolution that results from taking the limit $\tau \to 0$ in the scheme~\eqref{eq:SJKO}, focusing on the case of a potential energy $E : \mu \mapsto \int_\X V \dd \mu$,
for $V : \X \to \R$. We will connect it to a gradient flow in the geometry of Sinkhorn divergences, investigated in \cite{RGSD}.
The main results are the derivation of the limit equation (\cref{def:sinflow}), its analysis (\cref{thm:main}), together with the proof of the convergence of the outputs of~\eqref{eq:SJKO} to the limit equation in a restrictive setting (\cref{thm:tauto0}).
\end{minipage}}
\end{center}

\paragraph{Contributions and outline}

We start by deriving formally the evolution equation obtained when taking the limit $\tau \to 0$ in~\eqref{eq:SJKO}. We call \emph{Sinkhorn potential flow} the limit equation, see~\eqref{eq:Sinflow_formal}. At first glance, it seems ill-posed as it involves a deconvolution (\cref{rm:ill_posed}). We interpret, still at a formal level, the Sinkhorn potential flow as the gradient flow of $E$ in the geometry of Sinkhorn divergences. This is done in Section~\ref{sec:formal}. Then in Section~\ref{section:prelim} we recall the framework developed in \cite{RGSD} to understand the geometry of Sinkhorn divergences. In particular, \cite{RGSD} suggests the key change of variables to analyze the Sinkhorn potential flow: it consists of a non-linear embedding of $\PX$ into the Reproducing Kernel Hilbert Space (RKHS) $\mcal H_c$ built on the kernel $k_c = \exp(-c/\varepsilon)$. 

After all these preliminaries in Section~\ref{section:prelim} we formulate our main results in Section~\ref{sec:main_res}: that the equation~\eqref{eq:Sinflow_formal}, derived formally as the limit $\tau \to 0$ in~\eqref{eq:SJKO}, is well-posed, stable, and that in the long term the energy always converges to its minimal value (\cref{thm:main}). The next two sections are then focused on the proof of the main results. As anticipated, the key point is to rewrite the evolution in the RKHS $\mcal H_c$. Indeed, in this space the flow has a peculiar interpretation as a constrained rotation on the unit sphere of $\mcal H_c$ (\cref{rmk:maximal_monotone}). With the help of the understanding of the structure of the flow in Section~\ref{sec:structure}, we prove our main theorem in Section~\ref{sec:proofs}. 

We are able to recover the flow as the limit of the scheme~\eqref{eq:SJKO} when \(\X\) is a finite set of points (\cref{thm:tauto0}), 
while a lack of regularity in the Sinkhorn geometry makes the general case challenging (\cref{rem:obstructions}). The details are in Section~\ref{section:tauto0}.

In Section~\ref{sec:numerics}, we propose a simple computational scheme to solve~\eqref{eq:SJKO}, and provide illustrations of the flow's behavior for convex and non-convex potentials, in Eulerian and Lagrangian discretizations.

\subsection{Related work}

The scheme~\eqref{eq:SJKO} is new to the best of our knowledge. We are aware of the recent work \cite{aubin2025evolution}, which defines a general framework for schemes of the type~\eqref{eq:SJKO}, not necessarily in the space of probability measures and where $S_\varepsilon$ is substituted by any ``cost function''. Though~\eqref{eq:SJKO} is formally an instance of the framework \cite{aubin2025evolution}, the authors develop there a theory based on curvature of the cost (here $S_\varepsilon$) and convexity of the energy $E$. It is unclear if these assumptions are satisfied in our context, and, as the reader will see, our methods do not rely on curvature nor convexity, but they are very specific to the setting of the Sinkhorn divergence and potential energies.

We already mentioned \cite{peyre15,carlier2017convergence,baradat2025} which study~\eqref{eq:EJKO}, but need to send $\varepsilon \to 0$ together with $\tau$. In addition, we point out the works \cite{sander2022sinkformers,deb2023wasserstein,agarwal2025langevin} which study different flows in the space of probability measures built on EOT, and they also have a focus on a small $\varepsilon$ limit. 

Finally, and it was indeed the primary motivation of \cite{genevay18a} when introducing $S_\varepsilon$, one can look at a gradient flow in the classical optimal transport geometry, but when the energy $E$ is defined using $S_\varepsilon$, as is done for instance in \cite{carlier2024displacement} where the well-posedness of such a flow is proven, and in \cite{zhu2024neural} where it is used as a generative model by learning the velocity field of the flow on a dataset.

\subsection{Notations and setting}

We work on a compact metric space \((\X, \mathsf d)\). We write $\CX$ for the space of continuous functions over $\X$ valued in $\R$, endowed with the topology of uniform convergence. We denote by \(\MX\) the space of signed Radon measures on \(\X\), identified with the dual of \(\CX\) through the Riesz-Markov-Kakutani theorem, and endowed with the weak-\txtstar{} topology generated by this duality pairing. We also write \(\MpX\) for non-negative measures, \(\PX\) for Radon probability measures. The topological support of a non-negative measure $\sigma \in \MpX$ is denoted $\supp \sigma$.  We use $\CX / \R$ to denote $\CX$ quotiented by constant functions, it is canonically in duality with \(\MoX\) the space of balanced measures, i.e. measures \(\nu\in\MX\) verifying \(\int 1\dd\nu = 0\). 

The duality pairing between a space and its dual is denoted by \(\lrangle{\cdot, \cdot}\), and the dot product in a Hilbert space \(\mcal H\) is denoted \(\lrangle{\cdot, \cdot}_{\mcal H}\), with the corresponding norm denoted as \(\norm{\cdot}_{\mcal H}\).

If $(\mu_t)_{t \in I}$ is a curve defined on an interval $I$ and valued in $\PX$, we write $\dot{\mu}_t$ for $\frac{\partial \mu_t}{\partial t}$ whenever it exists. It is defined as belonging to the dual of a subspace of $\CX$ via the formula:
\begin{equation*}
\langle \dot{\mu}_t, \phi \rangle = \left. \frac{\dd}{\dd s} \langle \mu_s, \phi \rangle \right|_{s=t},
\end{equation*}
for $\phi$ belonging to a class of functions which will be clarified with the context.

\begin{assumption}[valid throughout the article]
\label{asmp:main}
The set $\X$ is a compact metric space.
We fix $\varepsilon > 0$ and a symmetric, non-negative cost \(c \in \mcal C(\X\times \X)\), assumed to induce a positive definite universal kernel \(k_c \coloneqq \exp(-c/\varepsilon)\). We consider \(V \in \CX\) and denote the potential energy \(E:\mu \mapsto \lrangle{\mu, V}\).
\end{assumption}

\noindent In particular our theory applies when $\X$ is a compact subset of $\R^d$ and $c(x,y) = \| x-y \|^2$ is the squared Euclidean distance because the Gaussian kernel is positive definite and universal \cite[Section 4]{micchelliuniversalkernels}.

\section{Derivation of the Sinkhorn potential flow}
\label{sec:formal}

The goal of this section is to derive, at least heuristically, the limit obtained when $\tau \to 0$ in the scheme~\eqref{eq:SJKO} and to make the link with the geometry of Sinkhorn divergences investigated in~\cite{RGSD}. As a preliminary step we need additional results on the entropic optimal transport problem.

\subsection{Dual EOT problem}

The EOT problem \eqref{eq:OTeps} has the following dual formulation \cite[Remark 4.24]{computationalot}:
\begin{equation}\label{eq:dualOTeps}
\OT_\varepsilon(\mu, \nu) = \sup_{f, g \in \CX} \lrangle{\mu, f} + \lrangle{\nu, g} - \varepsilon \lrangle{\mu\otimes \nu, \exp\brac{\frac1\varepsilon\brac{f\oplus g - c}}-1},
\end{equation}
with the notation \(f\oplus g: (x, y) \mapsto f(x) + g(y)\). This problem has maximizers \(f_{\mu, \nu}, g_{\mu, \nu}\) that we call Schrödinger potentials. They are characterized as solutions to the Euler-Lagrange equations, the so-called Schrödinger system:
\begin{empheq}[left=\empheqlbrace]{equation}\label{eq:schrödingersystem}
\begin{aligned}
f_{\mu, \nu} &= T_\varepsilon(g_{\mu, \nu}, \nu),\\
g_{\mu, \nu} &= T_\varepsilon(f_{\mu, \nu}, \mu).
\end{aligned}
\end{empheq}
Here the Sinkhorn mapping \(T_\varepsilon: \CX \times \PX \to \CX\) is defined, for $f \in \CX$ and $\mu \in \PX$, by 
\begin{equation}
\label{eq:TSinkhorn}
T_\varepsilon(f, \mu) : y \mapsto -\varepsilon \log\int_{\mcal X}\exp\brac{\frac1\varepsilon\brac{f(x)-c(x, y)}}\dd\mu(x).
\end{equation}
These potentials are unique up to constant shifts, i.e. if $(f_{\mu,\nu}, g_{\mu,\nu})$ solves~\eqref{eq:schrödingersystem} then the set \(\cbrac{(f_{\mu, \nu} + \lambda, g_{\mu, \nu} - \lambda), \lambda \in \R}\) describes all solutions of \eqref{eq:schrödingersystem}. We normalize the potentials such that $\langle \mu, f_{\mu,\nu} \rangle = \langle \nu, g_{\mu,\nu} \rangle$: with this convention \( g_{\mu, \nu} = f_{\nu, \mu}\) and in particular \(f_{\mu, \mu} = g_{\mu, \mu}\) which we shall denote \(f_\mu\) for short. 
The optimal transport plan in \eqref{eq:OTeps} is given by 
\(\pi = \exp\brac{\frac1\varepsilon\brac{f_{\mu, \nu} \oplus g_{\mu, \nu} - c}}(\mu\otimes\nu)\).


The Schrödinger potentials \( f_{\mu,\nu}, f_{\nu, \mu}\) correspond to the gradients of \(\OT_\varepsilon\) with respect to the input measures \( \mu, \nu\) \cite[Proposition 2]{feydy19}. In particular it entails the following result for the Sinkhorn divergence: fixing $\mu, \nu \in \PX$, take $\sigma \in \mcal M_0(\X)$ such that $t \mapsto \mu + t \sigma$ is a probability distribution for $t \geq 0$ small enough. Then, from \cite[Proposition 2]{feydy19} we obtain
\begin{equation}
\label{eq:gradientSeps}
\left. \frac{\dd  }{\dd t} S_\varepsilon(\mu + t\sigma, \nu)  \right|_{t=0} = \lrangle{\sigma, f_{\mu, \nu} - f_{\mu}}. 
\end{equation}  

\subsection{Formal limit in the scheme~\texorpdfstring{\eqref{eq:SJKO}}{(SJKO)}}

Take $\bar{\mu}_0 \in \PX$ and let $(\mu^\tau_k)_{k \geq 0}$ be defined iteratively by $\mu^\tau_0 = \bar{\mu}_0$ and~\eqref{eq:SJKO}. Note that the iterates always exist as $\PX$ is compact and $S_\varepsilon$ and $E$ are continuous with respect to weak-\txtstar{} convergence. We start by deriving rigorously the optimality condition in~\eqref{eq:SJKO}.

\begin{lemma}
\label{lm:optimality_SJKO}
If $\mu^\tau_{k+1}$ solves~\eqref{eq:SJKO}, then there exists $p^\tau_{k+1} \in \CX$ with $p^\tau_{k+1} \leq 0$, $p^\tau_{k+1} = 0$ on $\supp{\mu^\tau_{k+1}}$ and a constant $C^\tau_{k+1} \in \R$ such that 
\begin{equation}
\label{eq:optimality_SJKO}
\frac{f_{\mu^\tau_{k+1}, \mu^\tau_k} - f_{\mu^\tau_{k+1}}}{2 \tau} + V + p^\tau_{k+1} = C^\tau_{k+1}. 
\end{equation}
\end{lemma}

\begin{proof}
Fix $\nu \in \PX$ and let us use $(1-t) \mu^\tau_{k+1} + t \nu$ as a competitor in~\eqref{eq:SJKO}. Then from~\eqref{eq:gradientSeps}, taking the limit $t \to 0$ and with $\sigma = \nu - \mu^\tau_{k+1}$ we obtain 
\begin{equation*}
\frac{1}{2\tau} \langle \sigma, f_{\mu^\tau_{k+1}, \mu^\tau_k} - f_{\mu^\tau_{k+1}} \rangle + \lrangle{\sigma, V} \geq 0. 
\end{equation*}
With $g = (f_{\mu^\tau_{k+1}, \mu^\tau_k} - f_{\mu^\tau_{k+1}})/(2 \tau) + V$ which is a continuous function, it implies that $\lrangle{\mu^\tau_{k+1},g} \leq \lrangle{\nu,g}$ for any $\nu \in \PX$. It easily implies that $\mu^\tau_{k+1}$ is concentrated on the points of minimum of $g$: calling $C^\tau_{k+1}$ the minimal value of $g$, we obtain the result by defining $p^\tau_{k+1} = C^\tau_{k+1} - g$. 
\end{proof}

The second step is to notice that, if $\mu^\tau_{k}$ and $\mu^\tau_{k+1}$ are close, then $f_{\mu^\tau_{k+1}, \mu^\tau_k}$ and $f_{\mu^\tau_{k+1}}$ should be close too, and their difference should be approximatively linear in $\mu^\tau_{k+1} - \mu^\tau_{k}$. This can be obtained by linearizing the Schrödinger system~\eqref{eq:schrödingersystem}. The linearization was already performed in several in different contexts \cite{carlier2020differential,carlier2024displacement,Goldfeld2022,GonzalezSanz2022,RGSD}, we report here the result of \cite{RGSD}, see also \cref{prop:dfts/ds} below. The expression of the derivatives of the Schrödinger potentials involves the following operators \cite[Definition 3.3]{RGSD}: denoting the self transport kernel of \(\mu\in\mcal{P(X)}\) by \(k_\mu \coloneqq \exp\brac{\frac1\varepsilon\brac{f_\mu \oplus f_\mu - c}}\),
the operators \(H_\mu:\MX\to\CX\) and \(K_\mu:\CX \to \CX\) are defined as
\begin{align}
\forall \sigma \in \MX, \: H_\mu\sqbrac{\sigma}&:y\mapsto \ \int_{\mcal X}k_\mu(x, y)\dd\sigma(x), \label{eq:def:H_mu}\\
\forall \phi \in \CX, \: K_\mu[\phi] \coloneqq H_\mu\sqbrac{\phi\mu}&:y\mapsto \int_{\mcal X}k_\mu(x,y)\phi(x)\dd\mu(x). \label{eq:def:K_mu}
\end{align}
With these notations, \cite[Proposition 3.14]{RGSD} yields that, in $\CX / \R$,  
\begin{equation}
\label{eq:approx_dfts/ds}
\frac{f_{\mu^\tau_{k+1} \mu^\tau_k} - f_{\mu^\tau_{k+1}}}{\tau}  \simeq  \varepsilon (\id - K_{\mu^\tau_{k+1}}^2)\inv H_{\mu^\tau_{k+1}} \left[ \frac{\mu^\tau_{k+1} - \mu^\tau_k}{\tau} \right]. 
\end{equation}
To derive the equation of the limiting flow we assume that $(\mu^\tau_k)_{k \geq 0}$ converges to a differentiable limit curve $(\mu_t)_{t \geq 0}$ as $\tau \to 0$, in the sense that $\mu^\tau_k \simeq \mu_{k \tau}$ but also
$(\mu^\tau_{k+1} - \mu^\tau_k)/\tau \simeq  \left. \frac{\partial \mu}{\partial t} \right|_{t = k \tau} =: \dot{\mu}_{k \tau}$.
Passing to the limit $\tau \to 0$ in~\eqref{eq:optimality_SJKO} with the help of~\eqref{eq:approx_dfts/ds}, we obtain that the limit curve $(\mu_t)_{t \geq 0}$ should satisfy: for every $t \geq 0$, there exists $p_t \in \CX$ such that $p_t \leq 0$ and $p_t = 0$ on $\supp{\mu_t}$, and a constant $C_t \in \R$ such that 
\begin{equation}
\label{eq:Sinflow_formal}
\frac{\varepsilon}{2} (\id - K_{\mu_t}^2)\inv H_{\mu_t}[\dot{\mu}_t] + V + p_t = C_t.
\end{equation}
As \(K_\mu[1] = 1\), the operator \((\id-K_{\mu_t}^2)\inv\) is well defined as a bounded operator only on \(\CX/\R\) \cite[Theorem 3.9]{RGSD}. Thus~\eqref{eq:Sinflow_formal} should be understood as an equality in $\CX / \R$.
We call~\eqref{eq:Sinflow_formal} the equation of the \emph{Sinkhorn potential flow}, and in this work we prove the well-posedness of this equation (see \cref{def:sinflow} and \cref{thm:main} below). To give a rigorous meaning to this equation we need to specify the regularity of the curve $(\mu_t)_{t \geq 0}$, in particular the functional space to which $\dot{\mu}_t$ belongs. This will be only possible after a review of the tools of Section~\ref{section:prelim}.   

\begin{remark}[The flow looks unstable]
\label{rm:ill_posed}
Let us pause to take a look at~\eqref{eq:Sinflow_formal}. It is an evolution equation where, a priori, $\dot{\mu}_t$ can be inferred with the knowledge of $\mu_t$. Forgetting about constant functions as $(\id - K_{\mu_t}^2)$ cancels them, we have
\begin{equation*}
\dot{\mu}_t = - \frac{2}{\varepsilon} H_{\mu_t}^{-1}[  (\id - K_{\mu_t}^2)(V+p_t)]
\end{equation*}
The operator $H_{\mu_t}$ is a kernel integral operator, it is even the convolution with a Gaussian kernel in the case $f_{\mu_t} = 0$ and $c$ is the quadratic cost. Thus~\eqref{eq:Sinflow_formal} involves taking the inverse of an kernel integral operator! It is non-local, and, more frightening, a priori very unstable in the sense that $H_{\mu_t}^{-1}$ is a bounded operator only if we endow its domain with a very strong norm and its co-domain with a very weak one. 
It is therefore surprising that we still manage to deduce existence and stability in \cref{thm:main}. 
\end{remark}

\begin{remark}[Other energies $E$]
\label{rm:other_E}
At least formally, it is not difficult to see that if $E$ is not a potential energy, the equation of the flow is obtained by substituting $\frac{\delta E}{\delta \mu}$ (as defined in \cite[Chapter 7]{otam}) in place of $V$ in~\eqref{eq:Sinflow_formal}. Moreover, the pressure term $p$ vanishes if $E$ is for instance the differential entropy as automatically $\mu^\tau_{k+1}$ has full support \cite[Lemma 8.6]{otam}. However, proving existence of a solution to the limit flow for other energies than potential ones is a completely open question due to the analytical difficulties mentioned in \cref{rm:ill_posed}. As the reader will see, our proof method really utilizes that $E$ is a potential energy.  
\end{remark}

Though the optimality conditions~\eqref{eq:optimality_SJKO} were derived rigorously, passing to the limit $\tau \to 0$ to obtain~\eqref{eq:Sinflow_formal} is more challenging. One difficulty is to make the approximation~\eqref{eq:approx_dfts/ds} uniformly in $\tau$. We will only be able to rigorously justify the limit when $\X$ is a finite space, see Section~\ref{section:tauto0}.

\subsection{Link with the geometry of Sinkhorn divergences}

The work~\cite{RGSD} introduces a new geometry on the space of probability distributions. The starting point of the investigation is that, even though $S_\varepsilon$ is not a squared distance, it behaves asymptotically as a squared distance close to the diagonal $\mu = \nu$. Specifically, if one takes $t\mapsto \mu_t \coloneqq \mu + t\Dot \mu$ a vertical perturbation, where \(\mu \in \PX\), \(\Dot\mu \in \MoX\) and so that $\mu_t \in \PX$ at least for $t$ in a neighborhood of $t=0$, then \cite[Theorem 3.4]{RGSD} reads
\begin{equation}\label{eq:sinkhornhessian}
\frac1{t^2}S_\varepsilon(\mu, \mu_t) \xrightarrow[t\to0]{}\frac\varepsilon2\lrangle{\Dot \mu, (\id-K_\mu^2)\inv H_\mu\sqbrac{\Dot \mu}}.
\end{equation}

Following this computation, the metric tensor $\mathbf{g}_\mu(\cdot,\cdot)$ is defined \cite[Definition 4.1]{RGSD} as the right hand side of the expansion~\eqref{eq:sinkhornhessian}: for \(\mu\in\mcal{P(X)}\) and \(\Dot \mu_1, \Dot\mu_2 \in \MoX\):
\begin{equation}\label{eq:g_mu}
\mathbf g_\mu(\Dot \mu_1, \Dot \mu_2) \coloneqq \frac\varepsilon2 \lrangle{\Dot \mu_1, (\id - K_\mu^2)\inv H_\mu\sqbrac{\Dot \mu_2}}
\end{equation}
This is clearly a quadratic form on $\MoX$, actually a positive definite one~\cite[Theorem 4.6]{RGSD}: this can be seen as a consequence of~\eqref{eq:sinkhornhessian} and the non-negativity of $S_\varepsilon$ (\cref{thm:feydythm1}).

As in the limit $\tau \to 0$ we have $\mu^\tau_{k+1} \simeq \mu^\tau_k$,
we expect the scheme~\eqref{eq:SJKO} to look like the scheme where we substitute $S_\varepsilon$ by its approximation~\eqref{eq:sinkhornhessian}. That is, the formal equation~\eqref{eq:Sinflow_formal} should be the same as the gradient flow of the energy $E$, but in the geometry of $\PX$ given by the metric tensor $\mathbf{g}_\mu$. One way to formally check that we have the right limit in~\eqref{eq:Sinflow_formal} is to multiply this equation by $\sigma \in \mathcal{M}_0(\X)$ a measure with zero mass, specifically $\sigma = \nu - \mu_t$. We obtain:
\begin{equation}
\label{eq:gradient_flow_formal}
\forall \nu \in  \PX, \quad \mathbf{g}_{\mu_t}(\sigma,\dot{\mu}_t) + DE(\mu_t)(\sigma) \geq 0, \quad \text{with } \sigma = \nu - \mu_t
\end{equation} 
where $DE(\mu_t)(\sigma)$ is the directional derivative of $E$ in the direction $\sigma$, that is, along the curve $s \mapsto \mu_t + s \sigma$. Considering that $(\nu - \mu_t)$ should generate a dense subspace of the space of admissible directions at $\mu_t$, this is indeed the weak formulation of a (constrained) gradient flow in a Riemannian manifold as explained in the following remark. 

\begin{remark}[Weak formulation of a constrained gradient flow on a Riemannian manifold]
\label{rmk:GF_M}
In a smooth Riemannian manifold $M$ with a metric tensor $(g_y(\cdot,\cdot))_{y \in M}$, the equation of a gradient flow is $\dot{y}_t = - \nabla E (y_t)$. Here $\nabla E(y)$ is the gradient of $E$ at $y \in M$: it belongs to the tangent space $T_y M$, and is characterized by $DE(y)(z) = g_y(z,\nabla E(y))$ for any $z \in T_y M$, with $DE(y)(z)$ the differential of $E$ in the direction $z$. Thus taking the dot product with $z$ in the evolution equation yields the following weak characterization of a Riemannian gradient flow \cite[Equation (4)]{ottoporousmedium}:
\begin{equation*}
\forall z \in T_{y_t} M, \quad g_{y_t}(z,\dot{y}_t) + DE(y_t)(z) = 0.  
\end{equation*} 
Next we look at a constrained gradient flow, where for each $y \in M$ we can only move in the directions $N_y \subseteq T_y M$, with $N_y$ a convex cone in the tangent space. 
In this case the equation of the flow should be $\dot{y}_t = \Pi_{N_{y_t}}( - \nabla E (y_t))$, where $\Pi_{N_y}$ is the orthogonal projection on $N_y$ in the Hilbert space $(T_y M, g_y(\cdot,\cdot))$. As $N_y$ is a convex cone, the projection is characterized by: for any $a \in T_yM$ and $z \in N_y$, $g_y(z, a - \Pi_{N_y}(a)) \leq 0$. Using $a = - \nabla E(y)$, we obtain the following weak characterization:
\begin{equation}
\label{eq:GF_R_projected}
\forall z \in N_{y_t} , \quad g_{y_t}(\dot{y}_t,z) + DE(y_t)(z) \geq 0.  
\end{equation}  
We indeed recover~\eqref{eq:gradient_flow_formal} if the cone of admissible directions at $\mu_t$ is made by measures $\sigma = \nu - \mu_t$, with $\nu \in \PX$. It is important to restrict the space of admissible directions for an evolution on the ``manifold'' of probability distributions, as one must ensure not to leave the positivity constraint: the space $\PX$ can be thought as a manifold ``with corners''. The variable $p_t$ in~\eqref{eq:Sinflow_formal} is interpreted as the Lagrange multiplier enforcing the positivity constraint.  
\end{remark} 

\begin{remark}[Link with the metric theory]
In the study of gradient flows on metric spaces, the limiting equation is usually characterized by purely metric formulations such as the ``Energy Dissipation Equality'' or the ``Evolution Variational Inequality'' \cite{AGS}. Here we rather exploit the specific structure of the space $(\PX;\mathbf{g}_\mu(\cdot,\cdot))$ we have at hand, as the evolution is defined as a more concrete differential inclusion in $\CX / \R$. 

Moreover, the usual analysis of gradient flows in metric spaces relies on an assumption of geodesic convexity of the functionals involved \cite{AGS}. However in our setting, as proved in \cite[Corollary 5.18]{RGSD}, when $c(x,y) = \| x-y \|^2$ is the quadratic cost the curve $t \mapsto \delta_{(1-t) x_0+t x_1}$ is a constant-speed geodesic between $\delta_{x_0}$ and $\delta_{x_1}$. Thus, if $E$ were to be geodesically convex over $\PX$ for the geometry induced by $\mathbf{g}_{\mu}(\cdot,\cdot)$, it would imply that $V$ is convex over $\X$ in the usual sense. In contrast, our theory only requires $V$ to be continuous, not even semi-convex: our approach does not rely on the classical tools of the analysis of gradient flows in metric spaces but really utilizes the specific structure of Sinkhorn divergences, leading to stronger results with weaker assumptions.
\end{remark}

\section{Preliminaries on the geometry of Sinkhorn divergences}\label{section:prelim}

As hinted above, a key to analyze~\eqref{eq:Sinflow_formal} is the geometry of Sinkhorn divergences, characterized by the metric tensor $\mathbf{g}_\mu$ defined in~\eqref{eq:g_mu}, so far defined as a quadratic form on $\MoX$. In this section we highlight some key results of \cite{RGSD}, in particular the functional spaces to work with. It will allow us to select the right meaning to give for a solution of~\eqref{eq:Sinflow_formal}, as well as suggest a change of variables to help proving the well-posedness of the equation. 

Up to this point we have only introduced $\CX$ and $\MX$, respectively the spaces of functions and measures on $\X$. The operators $H_\mu$ (resp. $K_\mu$) defined in~\eqref{eq:def:H_mu} (resp.~\eqref{eq:def:K_mu}) are so far defined from $\MX$ (resp. $\CX$) into $\CX$. 

\begin{proposition}[{\cite[Proposition 3.6 and Theorem 3.9]{RGSD}}]
\label{prop:ipmK_bijection}
For any $\mu \in \PX$, the operators $H_\mu : \MX \to \CX$ and $K_\mu : \CX \to \CX$ are compact. Moreover, the operator $(\id + K_\mu)$ (resp. $(\id - K_\mu)$) is a continuous linear bijection of $\mcal C(\X)$ (resp. $\mcal C(\X) / \R$). 
\end{proposition}

However the pair $(\CX, \MX)$ is not enough to understand the properties of the geometry of Sinkhorn divergences. 
To push the analysis further we need the RKHSs built on the kernels $k_c = \exp(-c/\varepsilon)$ and $k_\mu = \exp((f_\mu \oplus f_\mu -c ) / \varepsilon)$.

\paragraph{Reproducing kernel Hilbert spaces}

Recall that a positive definite kernel $k : \X \times \X \to \R$ defines a unique RKHS $\mcal H_k$, as a subset of the space of functions from $\X$ to $\R$. It is the completion of $\mathrm{span} \{ k(x,\cdot) \ : x \in \X \}$ equipped with the norm induced by the dot product $\lrangle{k(x,\cdot), k(y,\cdot)}_{\mcal H_k} = k(x,y)$. If $k$ is continuous, which will always be our case, then $\mcal H_k \subseteq \CX$. We say that $k$ is universal if $\mcal H_k$ is dense in $\CX$. We refer to Appendix~\ref{appendix:RKHS} for more details.

If $\sigma$ is a linear form on $\mcal H_k$, denoted by $\sigma \in \mcal H^*_k$, we write $H_k[\sigma] \in \mcal H_k$ for its Riesz representative, that is, such that, for $\phi \in \mcal H_k$, 
\begin{equation*}
\lrangle{\sigma,\phi} = \lrangle{H_k[\sigma], \phi}_{\mcal H_k}.
\end{equation*}
We will only use the kernel $k_c = \exp(-c/\varepsilon)$ giving rise to the RKHS $\mcal H_c$; and for every  $\mu \in \PX$, the kernel $k_\mu = \exp( (f_\mu \oplus f_\mu - c) / \varepsilon)$, which gives rise to the RKHS $\mcal H_\mu$.
Even though the notation does not emphasize it, both these RKHSs depend on $\varepsilon$.
As $k_c$ and $k_\mu$ are continuous, $\mcal H_c, \mcal H_\mu \subseteq \CX$ and thus $\MX \subseteq \mcal H^*_c, \mcal H^*_\mu$. We use the shortcut $H_\mu$ instead of $H_{k_\mu}$, and this is consistent with~\eqref{eq:def:K_mu}: it can be checked that $H_\mu[\sigma]$, seen as the Riesz representation operator from $\mcal H_\mu^* \to \mcal H_\mu$ applied to $\sigma$, coincides with the expression in~\eqref{eq:def:K_mu} if $\sigma \in \MX$. On the other hand we use $H_c$ as a shortcut for $H_{k_c}$: at least if $\sigma \in \MX$
\begin{equation*}
H_c[\sigma](y) = \int \exp\left( - \frac{c(x,y)}{\varepsilon} \right) \dd \sigma(x). 
\end{equation*}

\begin{proposition}[{\cite[Proposition 4.2]{RGSD}}]
\label{prop:analytic_HK_mu}
For $\mu \in \PX$, the operators $H_c$ and $H_\mu$ are compact when restricted to $\MX$ and valued in $\mcal H_c$, $\mcal H_\mu$ respectively. 

For $\mu \in \PX$, the operator $K_\mu$ is non-negative, self-adjoint and compact, both as an operator from $L^2(\X, \mu)$ to $L^2(\X, \mu)$ and an operator from $\mcal H_\mu$ to $\mcal H_\mu$. The eigenvalue $1$ has multiplicity $1$, with eigenvectors given by constant functions, and all the other eigenvalues are in $[0,1-\exp(-4\| c \|_\infty/\varepsilon)]$. 
\end{proposition}

\paragraph{The tangent space as a dual of a RKHS}

Considering that the computations yielding \eqref{eq:sinkhornhessian} were only carried out in the case of vertical perturbations for simplicity, the space \(\MoX\) is expected to be only a subset of the tangent space. To identify a natural candidate, \cite{RGSD} considers the completion of \(\MoX\) with respect to the metric tensor \(\mathbf g_\mu\). The result happens to be a subset of the dual of a RKHS.  

From the identity $\mathbf g_\mu(\Dot \mu_1, \Dot \mu_2) = \frac{\varepsilon}{2} \lrangle{H_\mu[\Dot \mu_1], (\id - K_\mu^2)\inv H_\mu[\Dot \mu_2]}_{\mcal H_\mu}$ and Proposition~\ref{prop:analytic_HK_mu}, the work \cite{RGSD} proves that the completion of $\MoX$ with respect to $\mathbf{g}_\mu$ is the subset 
\begin{equation*}
\tgtspace{\mu} \coloneqq \cbrac{\sigma \in \mcal H_\mu^* \ : \ \lrangle{\sigma, 1} = 0},
\end{equation*}
and that $\mathbf{g}_\mu(\Dot \mu, \Dot \mu)$ is equivalent to the dual norm of $\Dot \mu$ in $\mcal H_\mu$, equivalently to $\| H_\mu[ \Dot \mu] \|^2_{\mcal H_\mu}$. The constraint \(\lrangle{\sigma, 1} = 0\) in the definition of \(\tgtspace{\mu}\) comes from the conservation of mass.

\begin{remark}
\label{rm:vertical_horizontal}
In particular $\mcal M_0(\X) \subseteq \tgtspace{\mu}$, which means all ``vertical perturbations'', that is, all curves $t \mapsto (1-t) \mu + t \nu$ with $\mu, \nu \in \PX$, have finite speed when measured with $\mathbf{g}_\mu(\cdot,\cdot)$. If in addition $\X \subset \R^d$ and $c$ is the quadratic cost, then $\tgtspace{\mu}$ contains not only vertical but also ``horizontal perturbations'' \(\Dot \mu = \div(\mu v)\) for a vector field \(v\in L^2(\X, \mu; \R^d)\). The latter horizontal perturbations are the only ones admissible in the geometry of optimal transport \cite[Chapter 8]{AGS}. Thus, the geometry of Sinkhorn divergences allows for more ways to move on $\PX$. 
\end{remark}

\begin{remark}[Regularity of the metric tensor and of the gradient flow]
The metric tensor $\mathbf{g}_{\mu}(\cdot,\cdot)$ is very regularizing, in the sense that it is finite for a large class of perturbations: all the ones such that $\dot{\mu} \in \tgtspace{\mu}$. However, the gradient flow involves the inverse of the metric tensor, which has in this case a de-regularizing effect, thus the limiting equation appears to be unstable (see \cref{rm:ill_posed}). A similar phenomenon has been discussed in \cite[Section 7]{park2025geometry} regarding gradient flows in the geometry of sliced optimal transport.
\end{remark}

\paragraph{A key change of variables}

The space of distributions \(\mcal H_{\mu}^*\) containing the tangent space is unfortunately dependent on \(\mu\), motivating \cite{RGSD}, inspired by \cite{feydy19,séjournéunbalanced}, to perform a change of variables in order to identify $\PX$ as a submanifold of the fixed Hilbert space $\mcal H_c$. The authors consider the mapping $ B : \PX \to \mcal H_c$ defined as 
\begin{equation}\label{eq:Bdef}
B(\mu) \coloneqq \exp\brac{-\frac{f_\mu}{\varepsilon}} =  H_c \left[ \exp\brac{\frac{f_\mu}{\varepsilon}} \mu \right],
\end{equation} 
where the last identity comes from the Schrödinger system \eqref{eq:schrödingersystem}, and shows that $B$ is valued in $\mcal H_c$.

\begin{theorem}[{\cite[Theorem 4.8]{RGSD}}]
\label{thm:embeddingB}
The map $B$ is an homeomorphism between $\PX$ and the set
\begin{equation*}
\mcal B \coloneqq H_c\sqbrac{\MpX} \cap \cbrac{b \in \mcal H_c \ : \ \norm{b}_{\mcal H_c} = 1},
\end{equation*} 
that is, the intersection between the convex cone $H_c[\MpX]$ and the unit sphere of $\mcal H_c$. Moreover, the set $\mcal B \subset \mcal H_c$ is compact, and strong and weak convergence coincide on $\mcal B$.
\end{theorem}

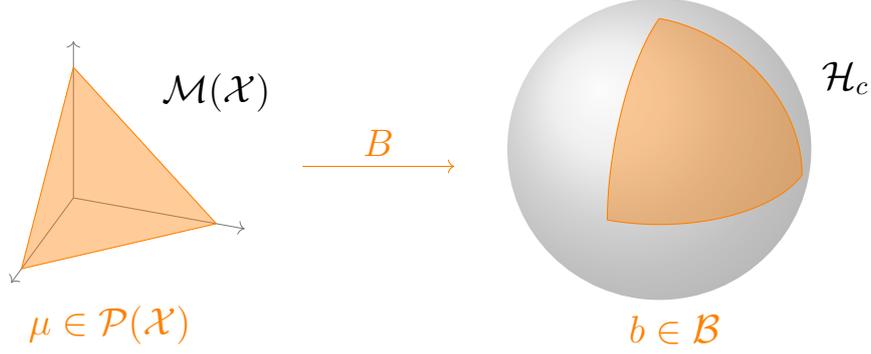
\begin{figure}[h]
\centering
\tdplotsetmaincoords{60}{110}
\begin{tikzpicture}[tdplot_main_coords, scale=2]
\draw[->, gray] (0, 0, 0) -- (1.2, 0, 0);
\draw[->, gray] (0, 0, 0) -- (0, 1.2, 0);
\draw[->, gray] (0, 0, 0) -- (0, 0, 1.2);
\fill[color=orange, opacity=.4] plot[variable=\x,domain=0:1] (\x, {1-\x}, 0)
-- plot[variable=\x,domain=0:1] (\x, 0, {1-\x})
-- plot[variable=\x,domain=0:1] (0, \x, {1-\x});
\draw[orange] plot[variable=\x,domain=0:1] (\x, {1-\x}, 0);
\draw[orange] plot[variable=\x,domain=0:1] (\x, 0, {1-\x});
\draw[orange] plot[variable=\x,domain=0:1] (0, \x, {1-\x});
\draw[orange] (1.5, .8, 0) node{\large \(\mu\in\PX\)};
\draw (0, 1, 1) node{\large \(\MX\)};
\end{tikzpicture}
\raisebox{24mm}{
\begin{tikzpicture}
\draw[orange, ->] (0, 0) -- (2, 0) node[midway, above, orange]{\large \(B\)};
\end{tikzpicture}
}
\tdplotsetmaincoords{60}{110}
\begin{tikzpicture}[tdplot_main_coords, scale=2]
\shade[tdplot_screen_coords, ball color = gray!40, opacity = 0.4] (0,0) circle (1);
\fill[color=orange, opacity=.4] plot[variable=\x,domain=0:90] (xyz spherical cs:radius=1,latitude=\x,longitude=0)
-- plot[variable=\x,domain=90:0] (xyz spherical cs:radius=1,latitude=\x,longitude=90)
-- plot[variable=\x,domain=90:0] (xyz spherical cs:radius=1,latitude=0,longitude=\x);
\draw[orange] plot[variable=\x,domain=0:90] (xyz spherical cs:radius=1,latitude=\x,longitude=0);
\draw[orange] plot[variable=\x,domain=90:0] (xyz spherical cs:radius=1,latitude=\x,longitude=90);
\draw[orange] plot[variable=\x,domain=90:0] (xyz spherical cs:radius=1,latitude=0,longitude=\x);
\draw[orange] (1.9, .8, 0) node[anchor=north]{\large\(b\in\mcal B\)};
\draw (0, 1.3, .8) node{\large \(\mcal H_c\)};
\end{tikzpicture}
\caption{Illustration of the change of variables.}
\end{figure}

\noindent The cone $H_c[\MpX]$ encodes the non-negativity constraint of the variable \(\mu\), and the unit sphere in \(\mcal H_c\) corresponds to the constraint of unit mass on \(\mu\).

With $b = B(\mu)$, in the sequel $b^{-1}$ denotes the function $x \mapsto 1/b(x)$ which is a continuous function as $b \in \CX$ and $b > 0$.
As~\eqref{eq:Bdef} reads $b = H_c[b^{-1} \mu]$ from~\eqref{eq:Bdef}, we obtain that the inverse of $B$, defined on $\mcal B$, is 
\begin{equation}
\label{eq:expression_Binv}
B \inv(b) = bH_c^{-1}[b]. 
\end{equation} 
In this expression $H_c^{-1}[b]$ is a measure while $b$ is a function, so that $B \inv (b)$ is indeed a measure.
In particular if $b = B(\mu)$ then we have 
\begin{align}
\label{eq:Hmu_in_b}
H_\mu[\sigma] & = e^{f_\mu/\varepsilon} H_c[ e^{f_\mu/\varepsilon}\sigma] = b^{-1} H_c[b^{-1} \sigma]. \\
\label{eq:Kmu_in_b}
K_\mu [\phi] & = e^{f_\mu/\varepsilon} H_c[ e^{f_\mu/\varepsilon} \phi  \mu] = b^{-1} H_c[\phi H_c^{-1}[b]] 
\end{align}
We write for the record the following easy result, that we will use in the computations below.

\begin{proposition}[{\cite[Proposition 4.7]{RGSD}}]
\label{prop:isometry_RKHS}
If $\mu \in \PX$ and $b = B(\mu)$, then $\phi \mapsto b^{-1} \phi$ is an isometry between $\mcal H_c$ and $\mcal H_{\mu}$.
\end{proposition}

The work \cite{RGSD} then computes the pullback of the metric tensor $\mathbf{g}_\mu$ by $B^{-1}$ onto $\mcal H_c$. The first step is to compute the differential of $B$. The following proposition ultimately boils down to computing the derivative of a Schrödinger potential along a curve, thus is related to the linearization of the Schrödinger system~\eqref{eq:schrödingersystem}. If a curve $(x_t)_t$ is valued in a linear space, we say that it is weakly differentiable if it is differentiable for the associated weak topology on this linear space.   

\begin{proposition}[{\cite[Lemma 4.13]{RGSD}}]
\label{prop:change_mudot_bdot}
Let \((\mu_t)_t\) be a curve valued in \(\PX\) and \((b_t)_t\coloneqq (B(\mu_t))_t\). If \((b_t)_t\) is weakly differentiable at \(t\) in \(\mcal H_c\), then \((\mu_t)_t\) is weakly-\txtstar{} differentiable at \(t\) in \(\mcal H_{\mu_t}^*\) and  
\begin{equation*}
\dot{b}_t = b_t (\id + K_{\mu_t})\inv H_{\mu_t}[\dot{\mu}_t].  
\end{equation*}
\end{proposition}

\noindent The differentiability of \((\mu_t)_t\) in \(\mcal H_{\mu_t}^*\) means that $s \mapsto \langle \mu_s,\phi \rangle$ is differentiable at $s=t$ for any $\phi \in \mcal H_{\mu_t}$, and that the derivative $\langle \dot{\mu}_t, \phi \rangle$ defines a linear form on $\mcal H_{\mu_t}^*$. 

Under the assumption of this proposition one can then check \cite[Theorem 4.12]{RGSD} that 
\begin{equation*}
\mathbf g_{\mu_t}(\Dot\mu_t, \Dot\mu_t)=\Tilde{\mathbf g}_{b_t}(\Dot b_t, \Dot b_t),
\end{equation*}
where for \(\mu \in \PX\), \(b \coloneqq B(\mu)\) and \(\Dot b \in b^\perp = \{ \phi \in \mcal H_c \ : \ \langle \phi, b \rangle_{\mcal H_c} = 0 \}\), 
\begin{equation}\label{eq:def:~g}
\Tilde{\mathbf g}_b\brac{\Dot b, \Dot b} \coloneqq \frac\varepsilon2\brac{\lrangle{\Dot b, \Dot b}_{\mcal H_c} + 2\lrangle{b^{-1}\Dot b, \brac{\id - K_\mu}\inv\sqbrac{b^{-1}\Dot b}}_{L^2(\X, \mu)}}.
\end{equation}  
The constraint $\Dot b \in b^\perp$ is natural: as $(b_t)_t$ is valued in $\mcal B$ which is included in the unit sphere of $\mcal H_c$, we have $\Dot b_t \in b_t^\perp$ as soon as the derivative exists. 

A key computation building on Proposition~\ref{prop:analytic_HK_mu} which gives the spectrum of $K_\mu$ yields the following estimate, independent on $\mu$ \cite[Proposition 4.15]{RGSD}:
\begin{equation}
\label{eq:equivalent_metric_tensor_flat}
\frac{\varepsilon}{2} \| \Dot b \|^2_{\mcal H_c} \leq \Tilde{\mathbf g}_b\brac{\Dot b, \Dot b} \leq \frac{\varepsilon}{2}  \left( 1+ 2 \exp\left( \frac{11}{2} \frac{\| c \|_\infty}{\varepsilon} \right) \right) \| \Dot b \|^2_{\mcal H_c}.
\end{equation}
It shows that the metric tensor $\Tilde{\mathbf g}_b$ compares uniformly to flat norm on $\mcal H_c$. Thus in the $b$-variable we have a fixed Hilbert space $\mcal H_c$ in which we can work, and the metric tensor is uniformly comparable to the flat norm, and can even be shown to be jointly continuous in $(b,\dot b)$ \cite[Proposition 4.14]{RGSD}.  

\paragraph{The Riemannian distance}

The article \cite{RGSD} finally defines a Riemannian-like distance \(\mathsf d_S\) on $\PX$ as
\begin{equation*}
\mathsf{d}_S(\mu,\nu)^2 \coloneqq \inf_{(\mu_t)_{t \in [0,1]}} \; \int_0^1 \mathbf{g}_{\mu_t}(\Dot \mu_t, \Dot \mu_t) \dd t,
\end{equation*}
where the infimum is taken among all paths $(\mu_t)_{t \in [0,1]}$ such that $(b_t)_t = (B(\mu_t))_t$ belongs to the Sobolev space \(\mathscr H^1([0, 1]; \mcal H_c)\) (see \cref{appendix:H1}) and $\mu_0 = \mu$, $\mu_1 = \nu$. This latter regularity condition is natural given that, in the $b$-variable, the metric tensor is equivalent to the flat norm~\eqref{eq:equivalent_metric_tensor_flat}. When \( (b_t)_t \in \mathscr H^1([0, 1]; \mcal H_c)\), the curve is differentiable for a.e. $t$ thus by construction of $\tilde{\mathbf{g}}_b$   
\begin{equation*}
\int_0^1 \mathbf{g}_{\mu_t}(\Dot \mu_t, \Dot \mu_t) \dd t = \int_0^1 \tilde{\mathbf{g}}_{b_t}(\Dot b_t ,\Dot b_t) \dd t. 
\end{equation*}
The minimum in the definition in $\mathsf{d}_S$ is attained \cite[Theorem 5.4]{RGSD}: geodesics exist for \(\mathsf d_S\).

By integrating the estimate~\eqref{eq:equivalent_metric_tensor_flat} regarding the metric tensor, it is possible to prove the following comparison result between the Riemannian distances.

\begin{proposition}[{\cite[Theorem 5.2]{RGSD}}]
\label{prop:equivalence_distance_curves}
The distance $\mathsf{d}_S$ metrizes weak-\txtstar{} convergence over $\PX$ and
there exists a constant $C$ which depends only on $\| c \|_\infty$ and $\varepsilon$ such that for all $\mu, \nu \in \PX$, 
\begin{equation*}
\frac{1}{C} \| B(\mu) - B(\nu) \|_{\mcal H_c} \leq \mathsf{d}_S(\mu,\nu) \leq C \| B(\mu) - B(\nu) \|_{\mcal H_c}.
\end{equation*}
In particular for $I$ an interval of $\R$, if $(\mu_t)_{t \in I}$ is a curve valued in $\PX$ and $(b_t)_{t \in I} = (B(\mu_t)_{t \in I})$ is its embedding in $\mcal B$, then $(\mu_t)_{t \in I} \in \mathrm{AC}^2(I;(\PX, \mathsf{d}_S))$ if and only if \((b_t)_{t \in I} \in \mathscr H^1(I; \mcal H_c)\).
\end{proposition}

\begin{remark}
\label{rem:finiteX}
If $\X$ is a finite space, that is, it contains a finite number of points, then some analytical issues disappear. Indeed in this case $\mcal H_c = \CX = \MX$ as all are identified with $\R^\X$, and all the topologies on these spaces coincide, being finite-dimensional. In this case $H_c[\mcal M_+(\X)]$ is a closed convex cone with non-empty interior.
\end{remark}

\section{Main results on the Sinkhorn potential flow}
\label{sec:main_res}

We are now able to give a rigorous meaning to~\eqref{eq:Sinflow_formal} the equation of the Sinkhorn potential flow, under the assumptions recalled Page~\pageref{asmp:main}. 
While the definition is written in the variable $\mu$, most of the interesting analysis will be performed in the variable $b$. For a measure $\mu$, we write 
\begin{equation*}
P\mu \coloneqq \left\{ p \in \CX \ : \ p \leq 0 \text{ and } p = 0 \text{ on } \supp \mu \right\}
\end{equation*}
the set of admissible pressures, as a subset of $\CX$. The notion of absolutely continuous curve is recalled in \cref{appendix:H1}.

\begin{definition}[Definition of the Sinkhorn potential flow]
\label{def:sinflow}
Let $V \in \CX$ and $\bar{\mu}_0 \in \PX$. We say that $(\mu_t)_{t \geq 0}$ is a \emph{Sinkhorn potential flow} starting from $\bar \mu_0$ if it belongs to $\mathrm{AC}_\text{loc}^2([0, + \infty); (\PX,\mathsf{d}_S))$ with $\mu_0 = \bar \mu_0$ and for a.e. $t$, its derivative $\dot{\mu}_t \in \tgtspace{\mu_t}$ satisfies as an inclusion in $\CX / \R$
\begin{equation}
\label{eq:sinflow_mu}
\frac{\varepsilon}{2}(\id - K_{\mu_t}^2)\inv H_{\mu_t}[\dot \mu_t] + V + P \mu_t  \ni 0. 
\end{equation} 
By a slight abuse of notations, we say that a curve $(b_t)_{t \geq 0}$ valued in $\mcal B$ is a \emph{Sinkhorn potential flow} starting from $\bar b_0 \in \mcal B$ if $(B^{-1}(b_t))_{t \geq 0}$ is a Sinkhorn potential flow starting from $B^{-1}(\bar b_0)$.  
\end{definition}

\noindent This definition is correct since a curve $ (\mu_t)_{t \geq 0} \in \mathrm{AC}_\text{loc}^2([0, + \infty); (\PX,\mathsf{d}_S))$ is differentiable in $\tgtspace{\mu_t}$ for a.e. $t$, see \cref{prop:change_mudot_bdot} and \cref{prop:equivalence_distance_curves}. 

\begin{theorem}[Main results on the Sinkhorn potential flow]
\label{thm:main}
For any \(V\in \CX\) and \(\bar{\mu}_0 \in \PX\), there exists a unique Sinkhorn potential flow \((\mu_t)_{t \geq 0}\) of \(V\) starting at \(\bar{\mu}_0\). In addition the flow has the following properties.
\begin{thmenum}
\item \emph{Energy dissipation}. The energy \(t\mapsto E(\mu_t) \) is absolutely continuous, non-increasing, and for a.e. \(t\) there holds 
\begin{equation}
\label{eq:energy_dissipation}
\frac{\dd}{\dd t} E(\mu_t) = - \mathbf g_{\mu_t}(\Dot \mu_t, \Dot \mu_t).
\end{equation}
\vspace{-15pt}
\label{thm:main:energy}
\item \emph{Long term behavior}. There holds
\begin{equation}\label{eq:limE=minE}
\lim_{t \to + \infty} E(\mu_t) = \min_{\PX}E.
\end{equation}
Moreover, if \(V\) has a unique minimizer \(x^\star \in \X\) then, weakly-\txtstar{} in $\PX$,
\begin{equation*}
\lim_{t \to + \infty} \mu_t = \delta_{x^\star}.  
\end{equation*}
\vspace{-15pt}
\label{thm:main:asymptotic} 
\item \emph{Non-expansiveness}. The flow is non-expansive in $\mcal H_c$ i.e. if $(\mu^1_t)_{t \geq 0}$, $(\mu^2_t)_{t \geq 0}$ are two Sinkhorn potential flows of $V$, then \(t\mapsto \norm{B(\mu^1_t)- B(\mu^2_t)}_{\mcal H_c}\) is non-increasing.\label{thm:main:contractivity}
\end{thmenum}
\end{theorem}

\noindent The proof will be provided in the next two sections. We first state two corollaries, on the integrability and monotonicity of the flow.

\begin{corollary}
\label{crl:bdot_L2}
If $(b_t)_{t \geq 0}$ is a Sinkhorn potential flow then $(\dot{b}_t)_{t \geq 0} \in L^2([0, + \infty); \mcal H_c)$: with $\mathrm{osc}(V) = \max V - \min V$ the oscillation of $V$,
\begin{equation}
\label{eq:born_bdot_L2}
\int_0^{+ \infty} \| \dot{b}_t \|^2_{\mcal H_c} \dd t \leq \frac{2}{\varepsilon} \mathrm{osc}(V).
\end{equation} 
\end{corollary}

\begin{proof}
We write $\mu_t = B^{-1}(b_t)$. From the bound~\eqref{eq:equivalent_metric_tensor_flat} and the energy dissipation equality~\eqref{eq:energy_dissipation}, for any $T > 0$: 
\begin{equation*}
\int_0^{T} \| \dot{b}_t \|^2_{\mcal H_c} \dd t \leq \frac{2}{\varepsilon} \int_{0}^{T} \tilde{\mathbf{g}}_{b_t}(\dot{b}_t,\dot{b}_t) \dd t = \frac{2}{\varepsilon} \int_{0}^{T} \mathbf{g}_{\mu_t} (\dot{\mu}_t,\dot{\mu}_t) \dd t = \frac{2}{\varepsilon} (E(\mu_0) - E(\mu_T)).
\end{equation*} 
The conclusion follows as $\mathrm{osc}(V) = \max E - \min E$ and $T$ is arbitrary.
\end{proof}

\begin{corollary}\label{coro:normdotb}
If $(b_t)_{t \geq 0}$ is a Sinkhorn potential flow then \(t \mapsto \norm{\Dot b_t}_{\mcal H_c}\) is non-increasing. 
\end{corollary}

\begin{proof}
If \(b\) is a Sinkhorn potential flow, \((b_{t+u})_{t \geq 0}\) also is one for any \(u>0\), and as a result of \refthmitem{thm:main}{contractivity}, for \(s\geq t\) holds \(\norm{b_{s+u} - b_s}_{\mcal H_c} \leq \norm{b_{t+u} - b_t}_{\mcal H_c}\). Dividing by \(u\) and making \(u\to 0\) gives \(\norm{\Dot b_s}_{\mcal H_c} \leq \norm{\Dot b_t}_{\mcal H_c}\). 
\end{proof}

\begin{remark}[Comparison with the transport equation]
Before we dive into the technical details, the reader can compare to what can be said for the gradient flow of $E$ in the Wasserstein geometry.
On $\X \subseteq \R^d$ the gradient flow of a potential energy $E : \mu \mapsto \int_\X V \dd \mu$ with respect to $\sqrt{S_0} = \sqrt{\OT_0} = W_2$ corresponds to the transport equation
\begin{equation*}
\frac{\partial \mu_t}{\partial t} = \div(\mu_t \nabla V)
\end{equation*} 
completed with no-flux boundary conditions.
Assuming $V$ is smooth and ignoring boundary issues, denoting by $\Phi_t : \X \to \X$ the flow of the ODE $\dot{x}_t = - \nabla V (x_t)$, we recall that $\mu_t = (\Phi_t)_\sharp \mu_0$.

Thus global existence of the flow is obtained if $V$ is semi-convex, and the flow is non-expansive in $W_2$ if the flow of $-\nabla V$ is non-expansive on $\R^d$, equivalently if $V$ is convex. On the other hand a Sinkhorn potential flow requires less regularity to get existence: $\X$ is a compact space (without a differentiable structure) and $V$ is any continuous function (not necessarily semi-convex nor differentiable). These assumptions also always yield non-expansiveness, though for a different metric. This could not be taken for granted given how unstable the equation looks at first glance (\cref{rm:ill_posed}).

Maybe more surprisingly, the energy always converges to $\min E$, whereas in the Wasserstein geometry the flow can get stuck in local minima.  
Indeed if $\bar{\mu}_0$ is concentrated on a local minimum of $V$, then the Wasserstein gradient flow of $E$ does not move, whereas Theorem~\ref{thm:main} shows that the Sinkhorn potential flow should converge to the point of global minimum. Though the proof below seems to rely on obscure algebraic computations, some intuition can be found as follows. For any $\mu,\nu \in  \PX$, the curve $t \mapsto (1-t) \mu + t \nu$ which interpolates linearly between $\mu$ and $\nu$ is always Lipschitz for $\mathsf{d}_S$ (\cref{rm:vertical_horizontal}), whereas it is not for $W_2$ if $\mu,\nu$ have disjoint support. Thus in the geometry of Sinkhorn divergences it is possible to ``teleport'' mass (that is, move linearly in the flat geometry of measures) and thus escape local minima, something which is not possible with classical optimal transport (\cref{rm:vertical_horizontal}). This intuition is confirmed by our empirical investigation in Section~\ref{sec:numerics}. 
\end{remark}

\section{The structure of the flow in the variable \texorpdfstring{$b$}{b}}
\label{sec:structure}

To prove \cref{thm:main} we write the equation in the variable $b = B(\mu)$ rather than in $\mu$. Indeed, it has a very peculiar structure when analyzed in the flat geometry of $\mcal H_c$.  

\subsection{The equation of the flow}

The first step is to obtain the equation satisfied by $(b_t)_{t \geq 0} = (B(\mu_t))_{t \geq 0}$ which is valued in the space $\mcal B \subset \mcal H_c$. Deriving the equation is not so difficult in itself, it relies only on \cref{prop:change_mudot_bdot} and algebraic manipulations.

We denote $V : \CX \to \CX$ for the linear operator corresponding to multiplication by $V$, meaning $Vb$ is the multiplication of the functions $V$ and $b$. We also write \(V^*\) the linear operator \(H_c[\MX] \to \mcal H_c\) defined by $H_cVH_c\inv$, that is,
\begin{equation*}
V^*b \coloneqq  H_c[V H_c\inv[b]].
\end{equation*}
We will use intensively the linear operator $W : H_c[\MX] \to \CX$ defined as 
\begin{equation*}
W \coloneqq \frac{2}{\varepsilon}(V - V^*).
\end{equation*} 
Lastly, we write $\tilde{P}$ for the multivalued map from $H_c[\MpX]$ to $\CX$ defined as
\begin{equation*}
\tilde{P} b \coloneqq \{ p \in \CX \ : \ p \leq 0 \text{ and } p = 0 \text{ on } \supp{H_c^{-1}[b]} \}.
\end{equation*}
With $b = B(\mu)$, given that $H_c^{-1}[b] = b^{-1} \mu$ and $b > 0$ on $\X$, we have $\tilde{P} b = P \mu$.

\begin{theorem}
\label{thm:change_flow_b}
A curve $(b_t)_{t \geq 0}$ valued in $\mcal B$ is a Sinkhorn potential flow starting from $\bar{b}_0$ if and only if it is an element of $\mathscr H^1_\text{loc}([0, + \infty); \mcal H_c)$ with $b_0 = \bar{b}_0$ which satisfies for a.e. $t$ as an inclusion in $\CX$
\begin{equation}
\label{eq:b_flat}
\dot{b}_t + W b_t + \tilde{P} b_t  \ni 0, \\
\end{equation}
\end{theorem}

\begin{proof}
Given \cref{prop:equivalence_distance_curves}, it is clear that $(b_t)_{t \geq 0} = (B(\mu_t))_{t \geq 0}$ for some curve $(\mu_t)_{t \geq 0}$ in $\mathrm{AC}_\text{loc}^2([0, + \infty); (\PX,\mathsf{d}_S))$ if and only if $b \in \mathscr H^1_\text{loc}([0, + \infty); \mcal H_c)$. The only delicate point is to prove that~\eqref{eq:sinflow_mu} and~\eqref{eq:b_flat} are equivalent.    

To that end we rely on \cref{prop:change_mudot_bdot} which relates $\dot{\mu}_t$ and $\dot{b}_t$. Indeed, using the algebraic identity $(\id - K_{\mu_t}^2)^{-1} = (\id - K_{\mu_t})^{-1} (\id + K_{\mu_t})^{-1}$ and \cref{prop:change_mudot_bdot},
\begin{equation}
\label{eq:zz_aux_metric_tensor} 
\frac{\varepsilon}{2} (\id - K_{\mu_t}^2)\inv H_{\mu_t}[\dot \mu_t] = \frac{\varepsilon}{2} (\id - K_{\mu_t})\inv [b_t^{-1} \dot{b}_t].
\end{equation}
Thus we start from~\eqref{eq:sinflow_mu} and compose it with $\frac{2}{\varepsilon} (\id - K_{\mu_t})$: as $(\id - K_{\mu_t})$ vanishes on constant functions, it implies
\begin{equation}
\label{eq:zz_aux_derivation_equation} 
b_t^{-1} \dot{b}_t + \frac{2}{\varepsilon}  (\id - K_{\mu_t})[V+P \mu_t] \ni 0. 
\end{equation} 
Conversely, if this inclusion~\eqref{eq:zz_aux_derivation_equation} holds, it also holds as an inclusion in $\CX/\R$, and thus as $(\id - K_{\mu_t})$ is injective on $\CX / \R$ (\cref{prop:ipmK_bijection}) we find that~\eqref{eq:zz_aux_derivation_equation} implies~\eqref{eq:sinflow_mu}. Multiplying the inclusion~\eqref{eq:zz_aux_derivation_equation} by $b_t$, we find that~\eqref{eq:sinflow_mu} is equivalent to the following inclusion on $\CX$: 
\begin{equation*}
\dot{b}_t + \frac{2}{\varepsilon} b_t (\id - K_{\mu_t})[V+P \mu_t] \ni 0. 
\end{equation*}  
Using the expression of $K_\mu$ given in~\eqref{eq:Kmu_in_b}, we have $b_t K_{\mu_t}[V] = 	V^* b_t$ so that $\frac{2}{\varepsilon} b_t (\id - K_{\mu_t})[V] = W b_t$.   
For the pressure we use $K_{\mu_t}[p]=0$ for any $p \in P \mu_t$ and $P \mu_t = \tilde{P} b_t$. That~\eqref{eq:sinflow_mu} is equivalent to~\eqref{eq:b_flat} follows. 
\end{proof}

\subsection{The monotone structure of the flow}

We highlight and explain the structure of the differential inclusion~\eqref{eq:b_flat}. Even though it no longer has a gradient flow structure, it is a monotone evolution in the flat geometry of $\mcal H_c$. 	

\begin{definition}
The bilinear form \(\llrrangle{\cdot,\cdot}_{\mcal H_c}:H_c[\MX]\times \CX\to \R\) is defined by
\begin{equation*}
\llrrangle{\phi, \psi}_{\mcal H_c} \coloneqq \lrangle{H_c\inv [\phi], \psi}
\end{equation*}
as a duality pairing between \(\MX\) and \(\CX\).
\end{definition}

\noindent This definition may seem slightly convoluted but by definition of $H_c$ as the Riesz representation operator,
\begin{equation*}
\forall \phi, \psi\in H_c[\MX], \quad \llrrangle{\phi,\psi}_{\mcal H_c} = \lrangle{\phi,\psi}_{\mcal H_c}.
\end{equation*}
That is, the restriction of \(\llrrangle{\cdot, \cdot}_{\mcal H_c}\) to \(H_c[\MX]\times H_c[\MX]\) is symmetric and coincides with $\lrangle{\cdot, \cdot}_{\mcal H_c}$. However, the key point is that the second argument of $\llrrangle{\cdot,\cdot}_{\mcal H_c}$ may be much less regular than $\mcal H_c$, as it can be any continuous function.  

With \(E:\mu \mapsto \lrangle{\mu, V}\), we write $\tilde{E} = E \circ B^{-1}$ for the energy defined over $\mcal B$. Then given the expression of $B\inv$ in~\eqref{eq:expression_Binv}, with $b = B(\mu)$,
\begin{equation}\label{eq:~E}
\Tilde E(b) = E(\mu) = \llrrangle{b, Vb}_{\mcal H_c}.
\end{equation}
In particular the energy is linear in $\mu$ but quadratic in $b$. It explains why the equation~\eqref{eq:b_flat} is \emph{linear} in the variable $b$ once we forget about the pressure. The expression in~\eqref{eq:~E} does not reduce to a dot product in \(\mcal H_c\) since for a generic \(b\in \mcal B\), the product \(Vb\) is in \(\CX\) and not necessarily in \(\mcal H_c\).

\begin{proposition}
\label{prop:adjoint_VP}
The following holds.
\begin{thmenum}
\item
For any $\phi, \psi \in H_c[\MX]$, we have 
\begin{equation}
\label{eq:V*}
\llrrangle{\phi, V\psi}_{\mcal H_c} = \llrrangle{\psi, V^* \phi}_{\mcal H_c},
\end{equation}
and in particular $\llrrangle{\phi, W\psi}_{\mcal H_c} = - \llrrangle{\psi, W \phi}_{\mcal H_c}$.
\label{prop:adjoint_VP:V}
\item
For any $b \in H_c[\mcal M_+(\X)]$ there holds
\begin{equation}
\label{eq:defTildeP}
\tilde{P}b = \cbrac{p \in \CX  \ : \  \forall \phi \in H_c[\mcal M_+(\X)], \, \llrrangle{\phi - b, p}_{\mcal H_c} \leq 0}, 
\end{equation}
and in particular $\llrrangle{b, p}_{\mcal H_c} = 0 $ if $p \in \tilde{P}b$.
\label{prop:adjoint_VP:P}
\end{thmenum}

\end{proposition}

\begin{proof}
For the first part we simply follow the definitions:
\begin{equation*}
\llrrangle{\phi, V\psi}_{\mcal H_c} = \lrangle{H_c^{-1}[\phi], V \psi} = \lrangle{ V H_c^{-1}[\phi],  \psi} = \llrrangle{H_c[V H_c^{-1}[\phi]], \psi}_{\mcal H_c} = \llrrangle{V^*\phi, \psi}_{\mcal H_c},
\end{equation*}  
and the conclusion follows as $\llrrangle{\cdot,\cdot}_{\mcal H_c}$ is symmetric on $H_c[\MX] \times H_c[\MX]$. The identity for $W$ follows.

For the second part, with $b = H_c[\nu]$, $\nu \in \mcal M_+(\X)$, recall that $p \in \tilde{P} b$ if and only if $p \leq 0$ and $p = 0$ on $\supp \nu$, which is equivalent to $\lrangle{\nu, p} \geq \lrangle{\sigma, p}$ for any $\sigma \in \mcal M_+(\X)$. The result follows from the definition of $\llrrangle{\cdot, \cdot}_{\mcal H_c}$. The equation $\llrrangle{b, p}_{\mcal H_c} = 0 $ follows by testing $\llrrangle{\phi - b, p}_{\mcal H_c}$ with $\phi= 0$ and $\phi = 2b$.
\end{proof}

\begin{corollary}
\label{crl_monotone}
The operator $b \mapsto W b + \tilde{P} b$ is a multivalued operator from $H_c[\mcal M_+(\X)]$ to $\CX$ which is monotone in the following sense: if $b^i \in H_c[\mcal M_+(\X)]$ and $\phi^i \in W b^i + \tilde{P} b^i$ for $i=1,2$ then 
\begin{equation*}
\llrrangle{ b^1 - b^2, \phi^1 - \phi^2 }_{\mcal H_c} \geq 0. 
\end{equation*}
\end{corollary}

\begin{proof}
From \refthmitem{prop:adjoint_VP}{V} we obtain that $\llrrangle{\phi, W \phi}_{\mcal H_c} = 0$ for any $\phi \in H_c[\MX]$. Applying to $\phi=b^1 - b^2$ we have $\llrrangle{ b^1 - b^2, W b^1 - W b^2 }_{\mcal H_c} = 0$.
On the other hand with $p^i \in \tilde{P} b^i$ for $i=1,2$, from~\eqref{eq:defTildeP} we have $\llrrangle{ b^1 - b^2, p^1}_{\mcal H_c} \geq 0$ and $\llrrangle{ b^2 - b^1, p^2}_{\mcal H_c} \geq 0$. Summing these three estimates yields the result. 
\end{proof}

\begin{proposition}
\label{prop:continuous_WP}
When restricted to $\mcal B$ the map $W : \mcal B \to \CX$ is continuous. Moreover, the graph of $\tilde{P}$, that is, $\{ (b,p) \ : \ b \in \mcal B \text{ and } p \in \tilde{P} b \}$, as a subset of $\mcal H_c \times \CX$, is closed.   
\end{proposition}

\begin{proof}
Continuity of $V$ is immediate since convergence in $\mcal H_c$ implies convergence in $\CX$ (\cref{lemma:infnorm<Hcnorm}) and the product is jointly continuous for continuous functions. Regarding $V^* = H_c V H_c^{-1}$, \cref{lemma:Hkc0} yields that $H_c^{-1}$ is norm-to-weak-\txtstar{} continuous when restricted to $\mcal B$ as $H_c^{-1}[\mcal B]$ is a bounded subset of $\MX$, multiplication by $V$ is weak-\txtstar{}-to-weak-\txtstar{} continuous, and $H_c$ is weak-\txtstar{}-to-norm continuous, see again \cref{lemma:Hkc0}. We conclude that $W$ is continuous.

Regarding $\tilde{P}$, assume that $b_n \to b$ in $\mcal H_c$ and $p_n \to p$ in $\CX$ with $p_n \in \tilde{P} b_n$ for any $n$. Clearly $p \leq 0$ as $p_n \leq 0$ for any $n$. Moreover, using \cref{lemma:Hkc0} again we have that $H_c^{-1}[b_n]$ converges weakly-\txtstar{} to $H_c^{-1}[b]$. The condition $p_n = 0$ on $\mathrm{supp}(H_c^{-1}[b_n])$, namely $\lrangle{H_c^{-1}[b_n],p_n} = 0$, can be passed to the limit $n \to + \infty$, giving $p = 0$ on $\mathrm{supp}(H_c^{-1}[b])$. Thus $p \in \tilde{P}b$. 
\end{proof}

\begin{remark}[A maximal monotone evolution?]
\label{rmk:maximal_monotone}
Up to the problem that $W$ is only defined on $H_c[\MX] \subseteq \mcal H_c$ and valued in $\CX \supseteq \mcal H_c$, we see that formally $W$ is skew-symmetric with respect to $\lrangle{\cdot, \cdot}_{\mcal H_c}$. By the Stone theorem \cite[Theorem 1.10.8]{Pazy1983} the unconstrained ODE described by \(\Dot b_t + Wb_t = 0\) has the form of a rotational motion. Regarding the second term $\tilde{P} b$, \refthmitem{prop:adjoint_VP}{P} actually characterizes it as the closure in $\CX$ of the subdifferential at $b$ of the $0/+\infty$ indicator of $H_c[\mcal M_+ (\X)]$. The structure ``linear skew-symmetric $+$ subdifferential of a convex function'' of the inclusion~\eqref{eq:b_flat} enables to interpret the flow as a ``constrained rotation'', as illustrated \cref{fig:sphere} below. It also yields that the operator $b \mapsto Wb + \tilde{P} b$ is monotone, but not the subdifferential of a convex function, as $b \mapsto Wb$ is not a gradient. As pointed out to us by Yann Brenier, the structure ``linear skew-symmetric $+$ subdifferential of a convex function'' of the evolution~\eqref{eq:b_flat}, though atypical, is not new and has been used in the analysis of conservation laws \cite{Brenier2009,Perepelitsa2014} and vibrating strings \cite{Brenier2004}. 

If $b \mapsto Wb + \tilde{P} b$ were maximal monotone, then existence of a solution to the Sinkhorn potential flow would follow from the well established theory of maximal monotone evolutions \cite{brezis1973}. However monotonicity is only proved with respect to $\llrrangle{\cdot,\cdot}_{\mcal H_c}$ and not $\lrangle{\cdot,\cdot}_{\mcal H_c}$, and $W + \tilde{P}$ is valued in $\CX$ which is larger than $\mcal H_c$. In Remark~\ref{rmk:needCX}, we show on an explicit example that it is not possible to prove existence of a solution only by working in the Hilbert space $\mcal H_c$, indicating that we need more ingredients than just the theory of maximal monotone operators.
\end{remark}

\begin{figure}[h]
\centering
\includegraphics[width=.6\textwidth]{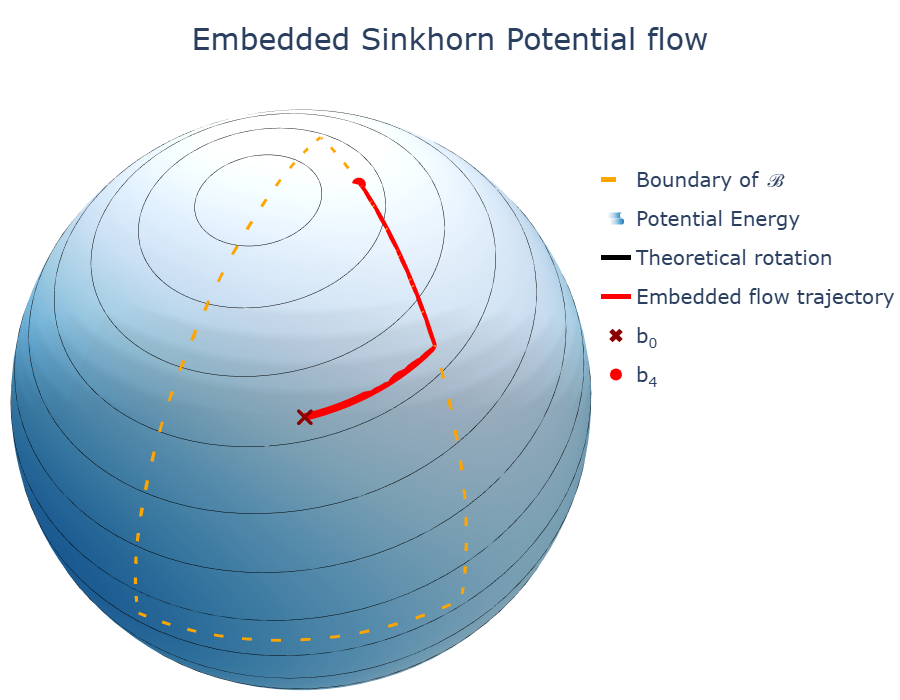}
\caption{Numerical simulation (see \cref{sec:numerics} for details of the implementation) of the Sinkhorn potential flow on a 3-point space embedded on the RKHS sphere, illustrating the structure of constrained rotation. The white to blue heat map corresponds to the potential energy \(\Tilde E\).}
\label{fig:sphere}
\end{figure}

\subsection{The gradient flow structure of the flow}

We started from a gradient flow but the equation~\eqref{eq:b_flat} seems to no longer have a gradient flow structure (\cref{rmk:maximal_monotone}). This is because we also changed the geometry: when looking at~\eqref{eq:b_flat} we use the flat geometry of $\mcal H_c$, whereas the Sinkhorn geometry is described in $\mcal H_c$ by the metric tensor $\tilde{\mathbf{g}}_b(\cdot, \cdot)$ as defined in~\eqref{eq:def:~g}. Though it is not directly needed for the sequel, we explain how we can also recover, at least formally, a gradient flow structure in the $b$-variable. 

The first observation is that~\eqref{eq:~E}, together with the identity~\eqref{eq:V*}, yields $\Tilde E(b) = \frac{1}{2} \llrrangle{b, (V+V^*)b}_{\mcal H_c}$, so that at least formally the gradient of $E$ in the flat geometry of $\mcal H_c$ should read
\begin{equation*}
\nabla_{\mcal H_c} \Tilde E(b) = (V+V^*)b. 
\end{equation*}
Note that $(V + V^*) b$ has no reason to be an element of $\mcal H_c$, it is a mere continuous function. Thus a priori $\tilde{E}$ cannot be differentiable in $\mcal H_c$.

The second observation is a rewriting of the metric tensor $\tilde{\mathbf{g}}_b(\cdot,\cdot)$. With $b = B(\mu)$, if we define \(\Tilde K_\mu: \mcal H_c \to \mcal H_c\) by \(\Tilde K_\mu [\phi] \coloneqq bK_\mu[b^{-1}\phi]\) then we claim that
\begin{equation*}
\Tilde{\mathbf g}_b\brac{\Dot b, \Dot b} = \langle\Dot b, \tilde{G}_b[\Dot b] \rangle_{\mcal H_c},
\end{equation*}
where \(\Tilde G_b: b^\perp \to b^\perp\) is given as
\begin{equation*}
\forall \Dot b \in b^\perp, \quad \Tilde G_b\sqbrac{\Dot b} \coloneqq \frac\varepsilon2(\id +\Tilde K_\mu) (\id - \Tilde K_\mu)\inv\sqbrac{\Dot b}. 
\end{equation*}
This comes from~\eqref{eq:def:~g}, using the identity 
\begin{equation*}
\langle b^{-1}  \dot{b}, (\id - K_\mu)^{-1} [b^{-1} \dot{b}] \rangle_{L^2(\X,\mu)} = \langle b^{-1} (\id - \tilde{K}_\mu)^{-1} [\dot{b}]\mu, b^{-1}  \dot{b} \rangle = \langle \dot{b}, \tilde{K}_\mu (\id - \tilde{K}_\mu)^{-1} [ \dot{b}] \rangle_{\mcal H_c}, 
\end{equation*}
(recall Proposition~\ref{prop:isometry_RKHS} for the last equality) as well as $\id + 2\tilde{K}_\mu (\id - \tilde{K}_\mu)^{-1} = (\id +\Tilde K_\mu) (\id - \Tilde K_\mu)\inv$. If we multiply the equation~\eqref{eq:b_flat} by $\tilde{G}_{b_t}$, we obtain after further algebraic simplifications that~\eqref{eq:b_flat}, and thus~\eqref{eq:sinflow_mu}, are equivalent to the evolution equation
\begin{equation*}
\tilde{G}_{b_t}[\dot{b}_t] + (V+V^*) b_t + \tilde{P} b_t + \R b_t  \ni 0.
\end{equation*}
This equation is the equation of a constrained gradient flow on the sub-Riemannian manifold $\mcal B$ endowed with the metric tensor $\tilde{\mathbf{g}}_b(\cdot,\cdot)$. To see this, we refer to Remark~\ref{rmk:GF_M}: if we multiply this equation by a tangent vector $z \in \mcal H_c$ which preserves the constraint $\mcal B$, that is, $z \in b_t^\perp$ and $z \in H_c[\MX]$ with $H_c^{-1}[z] \geq 0$ outside of $\mathrm{supp}(H_c^{-1}[b_t])$, we obtain for any such $z$ 
\begin{equation*}
\tilde{\mathbf{g}}_{b_t}(z,\dot{b}_t) + D \tilde{E}(b_t)(z) \geq 0.
\end{equation*}
This is exactly the weak form of the equation of a constrained gradient flow~\eqref{eq:GF_R_projected}.

\subsection{Example: the flow of a single Dirac mass}

To conclude this section, let us compute the flow of a single particle, i.e., when $\mu_t$ is concentrated on a single Dirac mass. This part is not needed to understand the proof of \cref{thm:main} and can be skipped at first reading. For a function $V$, the notation $\partial V$ stands for the subdifferential in the sense of convex analysis, which may be empty if $V$ is not convex. 

\begin{proposition}[Flow of a Dirac mass]
\label{prop:flow_dirac}
Let \(\mcal X\) be a compact convex subset of \(\R^d\), and \(c: (x, y) \mapsto \norm{x-y}^2\) be the square Euclidean distance. Take an absolutely continuous curve \((x_t)_{t \geq 0} \subset \mcal X\), let \(\mu_t\coloneqq \delta_{x_t}\) and \( (b_t)_{t \geq 0} \coloneqq (B(\mu_t))_{t \geq 0}\) the corresponding curve on \(\mcal B\). 

Then $(b_t)_{t \geq 0}$ is absolutely continuous in $\mcal H_c$ and for any \(t\) when $\dot{x}_t$ and $\dot{b}_t$ exist,
\begin{equation}\label{eq:diracflow}
\Dot b_t + Wb_t + \tilde{P} b_t \ni 0 \quad \Leftrightarrow \quad \Dot x_t \in -\partial V(x_t).
\end{equation}
In particular, for \(V\) convex and \(\bar{\mu}_0 = \delta_{\bar x_0}\), the Sinkhorn potential flow  \((\mu_t)_{t \geq 0}\) of \(V \) starting at \(\bar{\mu}_0\) can be written \(\mu_t = \delta_{x_t}\) with \((x_t)_{t \geq 0}\) the unique gradient flow of \(V\) starting at \(\bar x_0\).
\end{proposition}

Thus for a convex \(V\) and a Dirac mass as initial condition, the Sinkhorn and Wasserstein gradient flows of the potential energy coïncide. However, when \(V\) is non-convex, as soon as the particle following the Sinkhorn flow reaches a point where the subdifferential of $V$ in the sense of convex analysis is empty, its mass will split since it cannot keep being a Dirac mass. It is confirmed by our numerical experiments, see \cref{sec:numerics}. 

\begin{proof}
Recall that $k_c = \exp(-c/\varepsilon)$.
We compute easily $b_t = k_c(x_t,\cdot)$ by checking that it solves the Schrödinger system~\eqref{eq:schrödingersystem}. In particular from the reproducing property $\| b_t - b_s \|^2_{\mcal H_c} = 2 - 2 \exp(-\| x_t - x_s \|^2/\varepsilon) \leq 2 \| x_t - x_s \|^2 / \varepsilon$. Thus $(b_t)_{t \geq 0}$ is an absolutely continuous curve in $\mcal H_c$ if $(x_t)_{t \geq 0}$ is one in $\R^d$.

Consequently, one can compute, at any point where $\dot{x}_t$ and $\dot{b}_t$ exist, that \(\Dot b_t(y) = \frac2\varepsilon\lrangle{\Dot x_t, y-x_t}_{\R^d}b_t(y)\). Additionally, since \(H_c\inv [b_t]\) is supported on \(\cbrac{x_t}\), we have $V^*b_t = V(x_t)b_t$
and thus
\begin{equation*}
\Dot b_t(y) + \frac2\varepsilon [(V-V^*)b_t](y) = \frac2\varepsilon\brac{V(y) - V(x_t) + \lrangle{\Dot x_t, y-x_t}_{\R^d}}b_t(y).
\end{equation*}
The right hand side is evidently 0 for \(y=x_t\), and it is non-negative for all \(y\) if and only if \(-\Dot x_t\) is in the subdifferential of \(V\) at $x_t$, giving \eqref{eq:diracflow}. The rest of the statement is an immediate consequence.
\end{proof}

\begin{remark}
\label{rmk:needCX}
As seen in the proof above, along the evolution and for a convex $V$ we do not have $V b_t \in \mcal H_c$, only $V b_t \in \CX$. Thus, if we were to define $W$ only on $\mcal H_c$ as an unbounded operator with domain given by the set of $b$'s such that $Wb \in \mcal H_c$, we would not have $b_t \in \mathrm{dom}(W)$. It shows that the evolution~\eqref{eq:b_flat} cannot be analyzed only in $\mcal H_c$, and that we must indeed formulate it in $\CX$.    
\end{remark}

\section{Proof of the main theorem}
\label{sec:proofs}

In this section we prove \cref{thm:main} using the structure we explored in the previous section. In particular, we prefer to work with the equation~\eqref{eq:b_flat} (in the variable $b$) which is equivalent to the definition~\eqref{eq:sinflow_mu} (in the variable $\mu$), see \cref{thm:change_flow_b}. 

The main difficulty is the existence of a solution: to obtain it we start with the case of finite space, where we can utilize the theory of maximal monotone operators. Then we reason by approximating the space $\X$ with a finite space. The proof of non-expansiveness on the other hand, which we do before the one of existence, is immediate from the monotonicity of the evolution in the $b$-variable. In particular it implies uniqueness for a fixed initial condition. The long term behavior is obtained by a qualitative argument, by showing that the only critical points are global minima.

\subsection{Non-expansiveness}

\begin{proof}[\textbf{Proof of \refthmitem{thm:main}{contractivity}}]
For the non-expansiveness, with $b^i_t = B(\mu_t^i)$ and $p^i_t \in \tilde{P} b^i$ the pressure curve corresponding to each Sinkhorn potential flow for $i=1,2$,
the chain rule gives
\begin{equation*}
\frac{\dd}{\dd t}\norm{b^1_t-b^2_t}_{\mcal H_c}^2 = 2 \lrangle{b^1_t - b^2_t, \Dot b^1_t - \Dot b^2_t}_{\mcal H_c} = -2 \llrrangle{ b^1_t - b^2_t, W(b^1_t - b^2_t) + p_t^1 - p_t^2}_{\mcal H_c}.
\end{equation*}
Corollary~\ref{crl_monotone} guarantees that the right hand side is non-positive.
\end{proof}

\subsection{Existence on a finite space}

In this subsection, except for \cref{lm:energy_dissipation}, we consider that $\X$ is a finite space, namely
\begin{equation}\label{eq:discreteX}
\X = \{x_1, \dots, x_n\}.
\end{equation}
We refer to Remark~\ref{rem:finiteX} for comments on the importance of this assumption.

\begin{theorem}\label{thm:existenceuniqueness_finite}
Assuming \eqref{eq:discreteX}, for any \(\Bar b_0\in \mcal B\) there exists a Sinkhorn potential flow starting at \(\Bar b_0\). 
\end{theorem}

\begin{proof}
Recalling that the evolution is $\dot{b}_t + (W+\tilde{P})b_t \ni 0$, we prove that the multivalued map \(W+\tilde{P}\) is maximal monotone. Thanks to the finiteness of $\X$, we have $\CX = H_c[\MX] = \mcal H_c = \R^\X$, so that $W$ is linear and defined everywhere. By \eqref{eq:V*} we see that $W$ is a skew-symmetric continuous operator on the finite dimensional Hilbert space $\mcal H_c$. Moreover, by \eqref{eq:defTildeP} $\tilde{P} b$ is the subdifferential of a proper convex function, the one equal to $0$ on $H_c[\MpX]$ and $+ \infty$ elsewhere. Thus \(W+\tilde{P}\) is the sum of two maximal monotone operators. Since \(\interior{\dom{W}}\cap \mathrm{dom}(\tilde{P}) =  \mathrm{dom}(\tilde{P})\), \cite[Theorem 1 (a)]{rockafellarsumofmm} ensures that the sum \(W + \tilde{P}\) is maximal monotone.  

Applying the Hille-Yosida theorem (\cite{brezis1973}, see e.g. \cite[Theorem 2.7]{peypouquet2010} for a version in English) gives an absolutely continuous curve \((b_t)_{t \geq 0}\) verifying \(\Dot b_t + Wb_t + \tilde{P}b_t \ni 0\) for a.e. \(t\), valued in \(\mathrm{dom}(\tilde{P}) = H_c[\mcal M_+(\X)]\), and with bounded derivative (in particular $\dot{b} \in L^2_\text{loc}([0, +\infty); \mcal H_c)$).
It is also valued in the sphere of \(\mcal H_c\): indeed \(\lrangle{\Dot b_t, b_t}_{\mcal H_c} = \lrangle{Wb_t + p_t, b_t}_{\mcal H_c} = 0\) (see \cref{prop:adjoint_VP}), giving \(\norm{b_t}_{\mcal H_c}= \norm{b_0}_{\mcal H_c} = 1\) for all $t$. Thus $b_t \in \mcal B$ for all $t$. 
\end{proof}

Next we deduce in this case the dissipation of the energy, which will be useful to obtain a priori estimates. We first state a lemma which is valid even without the finiteness assumption on $\X$.

\begin{lemma}
\label{lm:energy_dissipation}
Let $(\mu_t)_{t \geq 0}$ be a Sinkhorn potential flow of $V$, then for a.e. $t$ such that $(E(\mu_t))_{t \geq 0}$ is differentiable at $t$,
\begin{equation*}
\frac{\dd}{\dd t} E(\mu_t) = - \mathbf{g}_{\mu_t}(\dot{\mu}_t,\dot{\mu}_t).
\end{equation*}
\end{lemma}

\begin{proof}
For a.e. $t$ by definition $(\mu_t)_{t \geq 0}$ is differentiable in $\tgtspace{\mu_t}$ at $t$ and satisfies~\eqref{eq:sinflow_mu}. Fix such an instant $t$ where in addition $(E(\mu_t))_{t \geq 0}$ is differentiable. For $s \neq t$ we can pair~\eqref{eq:sinflow_mu} with $(\mu_s - \mu_t) / (s-t)$. Calling $p_t \in P \mu_t$ the pressure at time $t$, as $E$ is linear and $\lrangle{\mu_t,p_t} = 0$, we obtain  
\begin{equation*}
\mathbf{g}_{\mu_t} \left( \frac{\mu_s - \mu_t}{s-t}, \dot{\mu}_t \right) + \frac{E(\mu_s) - E(\mu_t)}{s-t} = -\frac{1}{s-t} \lrangle{\mu_s,p_t}. 
\end{equation*}
As $s \to t$, the left hand side converges to $\mathbf{g}_{\mu_t}(\dot{\mu}_t,\dot{\mu}_t)+\frac{\dd}{\dd t} E(\mu_t)$ by the differentiability assumptions as $\mathbf{g}_{\mu_t}(\cdot,\cdot)$ is a continuous quadratic form on $\tgtspace{\mu_t}$. On the other hand the right hand side is non-negative (resp. non-positive) for $s > t$ (resp. $s < t$) since $p_t \leq 0$. As the limit when $s \to t$ exists, its only value can be $0$.  
\end{proof}

\noindent Leveraging this result, it is easy to prove the dissipation of the energy.

\begin{proposition}\label{prop:dE/dt}
Let \((b_t)_{t \geq 0}\) follow the Sinkhorn potential flow given by \cref{thm:existenceuniqueness_finite} and we write $\mu_t = B^{-1}(b_t)$. Then $(E(\mu_t))_{t \geq 0}$ belongs to $\mathscr H^1_\text{loc}([0, + \infty); \R)$ and satisfies the energy dissipation identity~\eqref{eq:energy_dissipation} for a.e. $t$. 
\end{proposition}

\begin{proof}
From the expression of $E$ given in~\eqref{eq:~E}, and as the curve $(b_t)_{t \geq 0}$ belongs to $\mathscr H^1_\text{loc}([0, +\infty); \mcal H_c)$ even with bounded derivative, we get $(\tilde{E}(b_t))_{t \geq 0} \in \mathscr H^1_\text{loc}([0, + \infty); \R)$. The latter coincides with $(E(\mu_t))_{t \geq 0}$, and is thus differentiable for a.e. $t$. We deduce the energy dissipation~\eqref{eq:energy_dissipation} from \cref{lm:energy_dissipation}.   
\end{proof}

\subsection{Existence in the general case}

Now that we have established results in the case of a finite space, we aim to generalize them to any compact metric space and prove the existence part of \cref{thm:main}. To obtain this result, we will discretize the space \(\mcal X\) and approximate it with some \(\mcal X^n\) containing \(n\) points, but doing so will yield solutions that are functions on \(\mcal X^n\), which we need to extend to \(\mcal X\). This is done using the following lemma.

\begin{lemma}\label{lemma:interp}
Let \(\mcal X^n = \{ x_1, \ldots, x_n \} \subset \mcal X\). For a positive definite kernel \(k\) on \(\mcal X\), denote \(\mcal H_k\) the corresponding RKHS and \(\mcal H_k^n\) the RKHS for the restriction \(k^n\) of \(k\) to \(\mcal X^n \times \X^n\). There is a linear isometric embedding \(\mcal I^n: \mcal H_k^n \to \mcal H_k\) verifying for any \(\phi\in\mcal H_k^n\) and any $i=1,\ldots,n$
\begin{equation}\label{eq:interp}
\mcal I^n[\phi](x_i) =  \phi(x_i).
\end{equation}
\end{lemma}

\begin{proof}
Since we have by construction \(\mcal H_k^n = \vspan(k^n(x_1,\cdot), \dots, k^n(x_n,\cdot))\), define for \(\phi = \sum_{i=1}^n\phi_i k^n(x_i, \cdot)\) its embedding $\mcal I^n[\phi] \coloneqq \sum_{i=1}^n\phi_i k(x_i, \cdot)$ which clearly satisfies~\eqref{eq:interp}.
It is evidently linear, and we can check \(\mcal I^n\) is an isometry by the reproducing kernel property:
\begin{equation*}
\norm{\mcal I^n[\phi]}_{\mcal H_k}^2 = \sum_{i,j=1}^n\phi_i \phi_j k(x_i, x_j) = \sum_{i,j=1}^n\phi_i \phi_j k^n(x_i, x_j)= \norm{\phi}_{\mcal H_k^n}^2. \qedhere
\end{equation*}
\end{proof}

\noindent In the following, we denote \(\cbrac{x_i}_{i=1}^\infty\) a countable set dense in \(\mcal X\), and use the notations of \cref{lemma:interp} for \(k=k_c\). We also write $\mcal B^n$ for the intersection of the unit sphere of $\mcal H^n_c$ and $H_c^n[\mcal M_+(\X^n)]$.

\begin{proposition}\label{prop:flow_ntoflow}
Let \(\Bar b_0 \in \mcal B\) an initial value, \((\sbar b^n_0)_n\) a sequence of elements with $\sbar b^n_0 \in \mcal B^n$ and such that \(\mcal I^n[\sbar b^n_0]\) converges in \(\mcal H_c\) to \(\Bar b_0 \in \mcal B\). For a given time horizon $T$, we write \(\brac{\underline b_t^n}_{t\in[0, T]}\) for the Sinkhorn potential flow on \([0, T]\) starting at \(\sbar b^n_0\) with underlying space \(\mcal X^n\) given by \cref{thm:existenceuniqueness_finite}, and denote \(b_t^n \coloneqq \mcal I^n\sqbrac{\underline b_t^n}\). 

Then the sequence \((b^n)_n\) has a subsequence converging uniformly in \(\mcal C([0, T]; \mcal H_c)\) to a Sinkhorn potential flow \((b_t)_{t\in[0, T]}\) on \([0, T]\) starting at \(\Bar b_0\). Moreover, with $\mu_t = B^{-1}(b_t)$, the energy $t \mapsto E(\mu_t)$ is non-increasing, absolutely continuous, and satisfies the energy dissipation equality~\eqref{eq:energy_dissipation}.
\end{proposition}

\begin{proof}
We proceed in two steps: first we show the compactness, then prove that the limit is a Sinkhorn flow satisfying~\eqref{eq:energy_dissipation}.

\medskip

\noindent \emph{Compactness.} Observe that \(\Dot b_t^n = \mcal I^n\sqbrac{\Dot{\underline b}^n_t}\) for a.e. $t$, from the linearity and continuity of \(\mcal I^n\). As $\mcal I^n$ is an isometry, with \cref{crl:bdot_L2} we obtain 
\begin{equation*}
\int_0^T\norm{\Dot{b}_t^n}_{\mcal H_c}^2\dd t = \int_0^T\norm{\Dot{\underline b}_t^n}_{\mcal H_c^n}^2\dd t \leq \frac2\varepsilon \mathrm{osc}(V).
\end{equation*}
Thus the sequence \((\Dot b^n)_n\) is bounded in \(L^2([0, T]; \mcal H_c)\). Moreover, $b^n_t \in \mcal B$ for all $n$ and all $t$, and the set $\mcal B$ is compact. By standard arguments (see \cref{thm:compactness_H1} in the appendix), up to extraction, the sequence $(b^n)_n$ converges uniformly in $\mcal C ([0,T],\mcal B)$ to a limit \(b \in \mathscr H^1([0, T];\mcal H_c) \), and the sequence $(\dot{b}^n)_n$ converges weakly in \(L^2([0, T]; \mcal H_c)\) to $\dot{b}$ the distributional derivative of the limit $b$.

\medskip

\noindent \emph{The limit is a Sinkhorn potential flow.}
Note that automatically \(b_0 = \Bar b_0\) by pointwise limit. Define
\(p_t^n \coloneqq -\brac{\Dot b_t^n + Wb_t^n}\) and \(p_t\coloneqq  -\brac{\Dot b_t + W b_t}\).
Showing \(p_t \in \tilde{P} b_t\) for a.e. \(t\) will yield that \((b_t)_{t\in[0, T]}\) is a Sinkhorn potential flow on \([0, T]\), see \cref{thm:change_flow_b}. 

From the interpolation property \eqref{eq:interp} of \(\mcal I^n\), we have for \(i=1, \ldots, n\) that \(\Dot b_t^n(x_i) = \Dot{\underline b}_t^n(x_i)\). Denoting \(W^n\) the operator corresponding to \(W\) on \(\mcal H_c^n\), we also have 
\([Wb_t^n] (x_i) = [W^n\underline b_t^n](x_i)\): this can be checked easily since $H_c^{-1}[\underline b_t^n]$ and $H_c^{-1}[ b_t^n]$ coincide, in the sense that they are both discrete measures supported on $\X^n$.

Therefore, it holds that \( p^n_t(x_i) \leq 0 \) for a.e. \(t\) and all $i = 1, \ldots n$. Thus for any $\sigma^n \in \mcal M_+(\X^n)$ and any $t_0 < t_1$,
\begin{equation}\label{eq:p_t^n<=0}
\int_{t_0}^{t_1} \langle \sigma^n, p^n_t \rangle \dd t \leq 0.
\end{equation}
Take now an arbitrary $\sigma \in \mcal M_+(\X)$. As $\X^n$ becomes dense in $\X$ as $n \to + \infty$, we can find a sequence $(\sigma^n)_n$ of measures with $\sigma^n \in \mcal M_+(\X^n)$ such that $\sigma^n$ converges weakly-\txtstar{} to $\sigma$ as $n \to + \infty$. Then it holds  
\begin{align*}
\int_{t_0}^{t_1} \langle \sigma^n, p^n_t \rangle \dd t & = - \int_{t_0}^{t_1} \lrangle{\sigma^n, Wb^n_t} \dd t - \int_{t_0}^{t_1} \langle \sigma^n, \dot{b}^n_t \rangle \dd t \\
&= - \int_{t_0}^{t_1} \lrangle{\sigma^n, Wb^n_t} \dd t - \int_{t_0}^{t_1} \langle H_c[\sigma^n], \dot{b}^n_t \rangle_{\mcal H_c} \dd t.
\end{align*}
The first term can be passed to the limit, since, for all $t$, $b^n_t \to b_t$ strongly in $\mcal H_c$ and thus \(Wb_t^n \to Wb_t\) strongly in $\CX$ (\cref{prop:continuous_WP}). Combined with the weak-\txtstar{} convergence of $\sigma^n$ to $\sigma$, the duality pairing converges, 
and the dominated convergence theorem gives convergence of the integral. For the second term, the convergence of \(\Dot b^n\) is weak in \(L^2([0, T]; \mcal H_c)\) and that of \( H_c[\sigma^n]\) to $H_c[\sigma]$ is strong in $\mcal H_c$ by \cref{lemma:Hkc0} and thus strong in the same \(L^2\) space, allowing one to pass the duality pairing to the limit.
We therefore get
\begin{equation*}
\lim_{n \to + \infty} \int_{t_0}^{t_1} \langle \sigma^n, p^n_t \rangle \dd t = \int_{t_0}^{t_1} \langle \sigma, Wb_t  + \dot{b}_t \rangle  \dd t = \int_{t_0}^{t_1} \langle \sigma, p_t \rangle \dd t.
\end{equation*}
Combined with \eqref{eq:p_t^n<=0} it gives that the right hand side is non-positive. By arbitrariness of $t_0 < t_1$ and $\sigma$, and as $p_t \in \CX$ for all $t$, we deduce that for a.e. $t$, we have $p_t(x) \leq 0$ for every \(x\). We also have for a.e. \(t\) that
\(
\llrrangle{ b_t, p_t}_{\mcal H_c} = -\llrrangle{ b_t, Wb_t}_{\mcal H_c} - \lrangle{b_t, \Dot b_t}_{\mcal H_c} = 0\)
since \(W\) is skew (\cref{prop:adjoint_VP}) and the norm of \(b_t\) is constant, thus \(p_t \in \tilde{P}b_t\). 

\medskip

\noindent \emph{Energy dissipation}. 
To prove the energy dissipation we only need to prove that $t \mapsto E(\mu_t)$ is absolutely continuous as the identity~\eqref{eq:energy_dissipation} follows from Lemma~\ref{lm:energy_dissipation}, and the monotonicity comes from~\eqref{eq:energy_dissipation}. Denoting by $C$ a constant independent of $n$ and whose exact value may be found in~\eqref{eq:equivalent_metric_tensor_flat}, with the help of \cref{prop:dE/dt}, for $t_0 < t_1$ 
\begin{equation*}
E(\mu^n_{t_0}) - E(\mu^n_{t_1}) = \int_{t_0}^{t_1} \tilde{\mathbf{g}}_{b^n_t}(\dot{b}^n_t, \dot{b}^n_t) \dd t \leq C \int_{t_0}^{t_1} \| \dot{b}^n_t \|_{\mcal H_c}^2 \dd t = C \int_{t_0}^{t_1} \| \underline{\Dot{b}}^n_t \|_{\mcal H^n_c}^2 \dd t. 
\end{equation*}
Next observe as $t \mapsto \| \underline{\Dot{b}}^n_t \|_{\mcal H_c}$ is non-increasing (\cref{coro:normdotb}) and by the integrability given in \cref{crl:bdot_L2} that 
\begin{equation*}
\| \underline{\Dot{b}}^n_t \|_{\mcal H_c} \leq \frac{1}{t} \int_0^t \| \underline{\Dot{b}}^n_s \|_{\mcal H^n_c} \dd s \leq \frac{1}{\sqrt{t}} \sqrt{\int_0^t \| \underline{\Dot{b}}^n_s \|_{\mcal H^n_c}^2 \dd s} \leq \sqrt{\frac{2 \mathrm{osc(V)}}{\varepsilon t}}.  
\end{equation*} 
Plugging these two estimates together, up to a newly defined constant $C$, we have $E(\mu^n_{t_0}) - E(\mu^n_{t_1}) \leq C \ln(t_1/t_0)$ if $t_0 < t_1$, and $t \mapsto E(\mu^n_t)$ is non-increasing and continuous. Passing to the limit we deduce $E(\mu_{t_0}) - E(\mu_{t_1}) \leq C \ln(t_1/t_0)$ and $t \mapsto E(\mu_t)$ non-increasing. It shows that $(E(\mu_t))_t$ is (locally) absolutely continuous. The conclusion about energy dissipation follows from Lemma~\ref{lm:energy_dissipation}. 
\end{proof}

\begin{proof}[\textbf{Proof of \cref{thm:main}: existence and \ref{thm:main:energy}}]
As $\X$ is metric and compact, we can find a sequence $(x_n)_{n}$ dense in $\X$. Define $\X^n = \{ x_1, \ldots, x_n \}$. Every $\bar{b}_0 \in \mcal B$ can be written $B(\bar{\mu}_0)$, and any $\bar{\mu}_0$ can be approximated by a sequence $\bar{\mu}^n_0 \in \mcal P(\mcal X^n)$. We write $b^n_t = \mcal I^n[\underline{b}^n_t]$ with $(\underline{b}^n_t)_{t \geq 0}$ the Sinkhorn potential flow on $\mcal X^n$ starting from $B(\bar{\mu}^n_0)$. Then \cref{prop:flow_ntoflow} and a diagonal extraction guarantee that $(b^n_t)_{t \geq 0}$ converges uniformly on compact sets to a Sinkhorn potential flow starting from $B(\bar{\mu}_0)$ which satisfies in addition \refthmitem{thm:main}{energy}. 
\end{proof}

\subsection{Long term behavior}

We now turn to the proof of the long term behavior, that is, the convergence of the energy to the minimal energy. Note that $E(\mu) = \int V \dd \mu$ has minimal value $\min V$, and that the minimal value is reached if $\mu$ is supported on \(\argmin_{\mcal X}V\). 

\refthmitem{thm:main}{asymptotic} is the consequence of \cref{lemma:cvtocritpoint} and \cref{lemma:critpointsaremin} below: we prove that all accumulation points of the flow are critical points, and that the only critical points in the Sinkhorn geometry are the global minimizers. 

\begin{lemma}\label{lemma:cvtocritpoint}
Let \((\mu_t)_{t \geq 0}\) follow a Sinkhorn potential flow of \(V\) and $(b_t)_{t \geq 0} = (B(\mu_t))_{t \geq 0}$. Then the curve $(b_t)_{t \geq 0}$ has at least one accumulation point in $\mcal B$, and any accumulation point $\bar b$ is critical in the sense that the following holds in $\CX$:
\begin{equation}\label{eq:critpoint}
 W \Bar b + \tilde{P} \Bar b \ni 0.
\end{equation}
\end{lemma}
\begin{proof}
The set \(\mcal B\) is compact (\cref{thm:embeddingB}) thus there exists at least one accumulation point for $(b_t)_{t \geq 0}$ as $t \to + \infty$. Let $\bar{b}$ be any accumulation point, that is, there exists $t_n \to + \infty$ such that \(b_{t_n} \to \bar b \in \mcal B\). We have \( -\Dot b_{t_n} \in (Wb_{t_n} + \tilde{P} b_{t_n}) \). Moreover, as $\| \dot{b}_t \|_{\mcal H_c}$ decreases with $t$ (\cref{coro:normdotb}), and $(\dot{b}_t)_{t \geq 0}$ is in $L^2([0, + \infty), \mcal H_c)$ (\cref{crl:bdot_L2}), we see that $\Dot b_{t_n}$ must converge to $0$ in $\mcal H_c$, thus in $\CX$.  
As the graph of $b \mapsto Wb + \tilde{P}b$ is closed (\cref{prop:continuous_WP}) we obtain~\eqref{eq:critpoint}.
\end{proof}

\begin{lemma}\label{lemma:critpointsaremin}
An element \( \bar{b} \in \mcal B\) is a critical point (i.e. verifies \eqref{eq:critpoint}) if and only if $B^{-1}(\bar{b})$ is supported on \(\argmin_{\mcal X}V\).
\end{lemma}

\begin{proof}
In this proof we write \(V_{\min} \coloneqq \min_{\mcal X}V\). We take $\bar{b} \in \mcal B$, and $\bar{\mu} = B^{-1}(\bar{b})$.

First we assume that $\bar{b}$ is a critical point. As $H_c^{-1}[\bar{b}]$ is a non-negative measure, it implies that $V H_c^{-1}[\bar{b}] \geq V_{\min} H_c^{-1}[\bar{b}]$, with equality if and only if \(\bar{\mu}\), equivalently $H_c^{-1}[\bar{b}]$, is supported on \(\argmin_{\mcal X}V\). Applying $H_c$, which is a kernel integral operator with a strictly positive kernel $\exp(-c/\varepsilon)$, we obtain for any $x$ 
\begin{equation*}
(V^* \bar{b})(x) =  H_c[V H_c^{-1}[\bar{b}]](x) \geq  H_c[V_{\min} H_c^{-1}[\bar{b}]](x) = V_{\min} \bar{b}(x),
\end{equation*}
with equality if and only if $V H_c^{-1}[\bar{b}] = V_{\min} H_c^{-1}[\bar{b}]$, that is, $\bar{\mu}$ is supported on \(\argmin_{\mcal X}V\). Recalling $W = \frac{2}{\varepsilon}(V - V^*)$, evaluating the inequality for $x^\star \in \argmin_{\mcal X}V$, we deduce that $W\bar{b}(x^\star) \leq 0$, and the inequality is strict except if $\bar{\mu}$ is supported on \(\argmin_{\mcal X}V\). As a critical point $\bar{b}$ satisfies $- W\bar{b} \in \tilde{P} \bar{b}$ so $-W\bar{b} \leq 0$, we see that there must be equality $W\bar{b}(x^\star) = 0$, so that a critical point $\bar{b}$ necessarily satisfies $\supp \mu \subseteq \argmin_{\mcal X}V$.  

Conversely, if $ \mathrm{supp}(\bar{\mu}) \subseteq \argmin_{\mcal X}V$ then $ V H_c^{-1}[\bar{b}] = V_{\min} H_c^{-1}[\bar{b}]$ and thus $V^* \bar{b} = V_{\min} \bar{b} $, leading to \(-W\bar{b} = \frac{2}{\varepsilon} (V_{\min} - V)\bar{b}\). The latter quantity is non-positive and vanishes on $\mathrm{supp}(\bar{\mu}) = \supp{H_c^{-1}[\bar{b}]}$ giving \(-W\bar{b} \in \tilde{P}\bar{b}\), that is, $\bar{b}$ is a critical point. 
\end{proof}

\begin{proof}[\textbf{Proof of \refthmitem{thm:main}{asymptotic}}]
As the energy $E(\mu_t) = \tilde{E}(b_t)$ is non-increasing, for any sequence $(t_n)_n$ going to $+ \infty$ we know that $\tilde{E}(b_{t_n})$ converges to $\lim_t \tilde{E}(b_t) = \lim_t E(\mu_t)$. Choosing $(t_n)$ a subsequence such that $(b_{t_n})_n$ converges to $\bar{b}$ an accumulation point of $(b_t)_{t \geq 0}$, we have that $\bar{b}$ is critical (\cref{lemma:cvtocritpoint}), and thus $\tilde{E}(\bar{b}) = \min E$ by \cref{lemma:critpointsaremin}. As $\tilde{E}$ is continuous over $\mcal B$ we obtain \eqref{eq:limE=minE}.

If furthermore \(V\) has a unique minimizer \(x^\star\), then there exists a unique critical point by \cref{lemma:critpointsaremin}, namely $B(\delta_{x^\star})$. Thus the function $(b_t)_{t \geq 0}$, valued in the compact $\mcal B$, has $B(\delta_{x^\star})$ as unique accumulation point. The conclusion follows by continuity of $B^{-1}$.
\end{proof}

\section{Convergence of the SJKO scheme in the case of a finite space}\label{section:tauto0}

Despite the fact that the equation~\eqref{eq:sinflow_mu} for the Sinkhorn potential flow was initially derived as the formal limit of the scheme~\eqref{eq:SJKO}, we have proven the existence of a solution to it through the means of monotone operator theory and spatial discretization, rather than proving directly that the scheme~\eqref{eq:SJKO} converges to a solution as the time step $\tau$ vanishes. The latter approach appears to be far more involved than the former in the general case, due to a lack of regularity which will appear further below (\cref{rem:obstructions}).
Nevertheless, in this section we prove the convergence of the scheme when \(\mcal X\) is a finite space, as a lot of difficulties of functional analysis disappear in this case, see \cref{rem:finiteX}.

\begin{theorem}\label{thm:tauto0}
Let \(\mcal X = \cbrac{x_1, \dots, x_n}\) a set with a finite number of points, and for \(\tau >0\) define \(\brac{\mu_k^\tau}_k\) the sequence given by the scheme \eqref{eq:SJKO} with initialization \( \bar{\mu}_0 \in \mcal{P(X)}\). Define \(b_k^\tau \coloneqq B(\mu_k^\tau)\), let \(\brac{\sbar b_t^\tau}_{t \geq 0}\) be its piecewise constant interpolation (i.e. \(\sbar b_t^\tau \coloneqq b_k^\tau\text{ for }t\in[k\tau, (k+1)\tau)\)), and \( (\hat{b}_t^\tau)_{t \geq 0}\) its piecewise geodesic interpolation (i.e. the restriction of this curve to \([k\tau, (k+1)\tau]\) is a constant-speed geodesic between \(b_k^\tau\) and \(b_{k+1}^\tau\) for \(\mathsf d_S\)). 

Then for any \(T>0\), as \(\tau \to 0\), \(\sbar b^\tau\) and \(\hat b^\tau\) both converge uniformly on $[0,T]$ for $\| \cdot \|_{\mcal H_c}$ to the Sinkhorn potential flow of \(V\) starting at \(B(\bar{\mu}_0)\).
\end{theorem}

In order to prove this result we follow the classical strategy to derive limiting equations for minimizing movement schemes \cite{santambrogio17}: we first use the usual coercivity estimates to obtain existence of a converging subsequence, and then take the optimality conditions~\eqref{eq:optimality_SJKO} to the limit. This second step relies on the quantification of the error in~\eqref{eq:approx_dfts/ds}, and it involves generalizations of the operators \(K_\mu\) and \(H_\mu\) which we must first introduce.

\subsection{The derivative of a Schrödinger potential}\label{section:dfts/ds}

The definition and results of this section would also hold if $\mcal X$ is not finite.

\begin{definition}\label{def:kmunu}
If $\mu, \nu \in \PX$ define the kernel \(k_{\mu, \nu}\coloneqq \exp\brac{(f_{\mu, \nu} \oplus f_{\nu, \mu} - c) /\varepsilon}.
\)
We define the operators \(H_{\mu,\nu}:\MX \to \CX\) and \(K_{\mu,\nu}:\CX \to \CX\) by
\begin{align*}
\forall \sigma \in \MX,\: H_{\mu,\nu}[\sigma]&:x\mapsto \int_\X k_{\mu, \nu}(x,y) \dd \sigma(y) \\
\forall \phi \in \CX,\: K_{\mu, \nu}[\phi] &\coloneqq H_{\mu,\nu}[\phi\nu].
\end{align*}
\end{definition}

\noindent If $\mu = \nu$ then $H_{\mu,\mu} = H_\mu$ and $K_{\mu,\mu} = K_\mu$, see~\eqref{eq:def:H_mu} and~\eqref{eq:def:K_mu}.
The notation is chosen with this convention because \(H_{\mu, \nu}\) is actually valued in the space \(\mcal H_{\mu, \nu} \coloneqq \exp( f_{\mu,\nu}/\varepsilon)\mcal H_c\), and the latter coincides with $\mcal H_\mu$ when $\mu = \nu$, see Proposition~\ref{prop:isometry_RKHS}. As when \(\mu=\nu\), the Schrödinger system \eqref{eq:schrödingersystem} gives \(K_{\mu, \nu}[1] = 1\). We now state some results extending those of \cite[Section 3]{RGSD}.

\begin{proposition}\label{prop:compactoperators}
For any $\mu, \nu \in \PX$, the operators
\(H_{\mu,\nu}:\MX \to \CX\) and \(K_{\mu,\nu}: \CX \to \CX\) are compact.
\end{proposition}

\begin{proof}
\cite[Proposition 3.6]{RGSD} proves the result when \(\mu=\nu\), and the proof still works when \(\mu \neq \nu\): we only need $k_{\mu,\nu}$ to be a continuous function on the compact space $\X \times \X$ for the integral operators $H_{\mu,\nu}$ and $K_{\mu,\nu}$ to be compact. 
\end{proof}

\noindent We also need some results on the operator \(\id-K_{\mu, \nu}K_{\nu, \mu}\), which should be thought as the generalization of $\id - K_\mu^2$.

\begin{proposition}\label{prop:K_munu}
For any $\mu, \nu \in \PX$, the operator \(\id-K_{\mu, \nu}K_{\nu, \mu}\) is invertible on \(\CX/\R\).
\end{proposition}

\begin{proof}
The quotient norm in \(\CX/\R\) is defined by \(\norm{\phi}_{\CX/\R} \coloneqq \inf_{\lambda \in \R}\norm{\phi - \lambda}_{\infty}\). The same proof as \cite[Proposition 3.8]{RGSD} yields \( \norm{K_{\mu, \nu}[\phi]}_{\CX/\R} \leq q \norm{\phi}_{\CX/\R}\) with $q = 1 - \exp(- 4 \| c \|_\infty / \varepsilon) \in (0,1)$. The result follows from the Fredholm alternative, see the proof of \cite[Theorem 3.9]{RGSD}.
\end{proof}

We can compute the derivative of the Schrödinger potentials and generalize \cite[Proposition 3.14]{RGSD} in order to quantify the error in the approximation~\eqref{eq:approx_dfts/ds}. 

\begin{proposition}\label{prop:dfts/ds}
Assume \((\mu_t)_{t \in I}\) is a curve valued in \(\mcal{P(X)}\) and weakly-\txtstar{} differentiable in \(\MX\), and with $t \mapsto \dot{\mu}_t \in \MX$ weak-\txtstar{} continuous. Denote \(f_{t,s} \coloneqq f_{\mu_t, \mu_s}\), \(K_{t,s} \coloneqq K_{\mu_t, \mu_s}\) and \(H_{t, s} \coloneqq H_{\mu_t, \mu_s}\). Then $s \mapsto f_{t,s} \in \mcal C^1(I;\CX)$ and in $\CX / \R$,
\begin{equation}\label{eq:dfts/ds}
\frac{\partial f_{t,s}}{\partial s} = -\varepsilon\brac{\id - K_{t, s}K_{s,t}}\inv H_{t, s}[\Dot \mu_s].
\end{equation}
\end{proposition}

\begin{proof}
We will follow the same method as the proof of \cite[Lemma 3.13]{RGSD}. The existence of derivatives is obtained by applying the implicit function theorem \cite[Theorem 10.2.1]{dieudonneanalysis} to the function $\tilde{T} : \CX / \R \times I \to \CX / \R$, where, recalling that $T_\varepsilon$ is defined in~\eqref{eq:TSinkhorn}, 
\begin{equation}
\Tilde T(f,s) =  f -  T_\varepsilon(T_\varepsilon(f, \mu_t), \mu_s)
\end{equation}
and \(t\) is fixed. The optimal conditions given by the Schrödinger system \eqref{eq:schrödingersystem} characterize \(f_{t,s}\) as the unique solution in \(\CX/\R\) of \(\Tilde T(f, s) = 0\). 

A direct computation yields that 
\begin{align}
D_1T_\varepsilon(f, \mu)[\phi](x) &= -\frac{\lrangle{\mu, \phi\exp(f/\varepsilon)k_c(x, \cdot)}}{\lrangle{\mu, \exp(f/\varepsilon)k_c(x, \cdot)}}\label{eq:D1Teps}\\
D_2T_\varepsilon(f, \mu)[\Dot \mu](x) &= -\varepsilon\frac{\lrangle{\Dot \mu,\exp(f/\varepsilon)k_c(x, \cdot)}}{\lrangle{\mu,\exp(f/\varepsilon)k_c(x, \cdot)}}.\label{eq:D2Teps}
\end{align}
In particular given the assumption of regularity on $(\mu_t)_t$ and $(\dot{\mu_t})_t$, the map $(f,t) \mapsto T_\varepsilon(f,\mu_t)$ is of class $\mcal C^1$ on $\CX \times I$, so that the map $\Tilde T$ is too.

Using that in \eqref{eq:D1Teps} the denominator is $\exp(-f_{t,s}/\varepsilon)$ when evaluated at $(f_{s,t},\mu_s)$, we find $D_1T_\varepsilon(f_{s,t}, \mu_s)[\phi] = - K_{t,s}[\phi]$ and \(D_1T_\varepsilon(f_{t,s}, \mu_t)[\phi] = - K_{s,t}[\phi]\). Using the chain rule, we deduce that 
\begin{equation}
\label{eq:D1T~}
D_1\Tilde T(f_{t,s}, s) = \id - D_1T_\varepsilon(f_{s,t}, \mu_s)D_1T_\varepsilon(f_{t,s}, \mu_t)  = \id - K_{t,s}K_{s,t}
\end{equation}
On the other hand using~\eqref{eq:D2Teps} and again the Schrödinger system, 
\begin{equation}
\label{eq:D2T~}
\frac{\partial \Tilde T(f_{t,s}, s)}{\partial s} = \varepsilon H_{t,s}[\dot{\mu}_s]. 
\end{equation}
We can apply the implicit function theorem since $D_1\Tilde T(f_{t,s}, s)$ is invertible on \(\CX/\R\) by \cref{prop:K_munu}. It yields that $s \mapsto f_{t,s}$ is of class $\mathcal{C}^1$ and 
\begin{equation*}
\frac{\partial f_{t,s}}{\partial s} = - \brac{D_1\Tilde T(f_{t,s}, s)}\inv\brac{\frac{\partial \Tilde T(f_{t,s}, s)}{\partial s}}.
\end{equation*}
Substituting the expressions obtained in~\eqref{eq:D1T~} and ~\eqref{eq:D2T~} we obtain the conclusion. 
\end{proof}

Notice that \cref{prop:dfts/ds} applies only when the derivative \(\Dot \mu\) is valued and continuous in \(\MX\), which is part of the reason why \cref{thm:tauto0} is proven only for a finite space. We will comment further on this difficulty in \cref{rem:obstructions}.

\subsection{Limit \texorpdfstring{\(\tau \to 0\)}{τ→0} and proof of \texorpdfstring{\cref{thm:tauto0}}{Theorem~\ref{thm:tauto0}}}

We can now move towards the proof of \cref{thm:tauto0}. As mentioned above, the proof is in two steps: first compactness, and then limit of the optimality conditions. Since the first one is provable for any compact space and may be used in the future to prove a generalization of \cref{thm:tauto0}, we state it as a separate lemma. We start by quickly recalling the following coercivity estimate from \cite{séjournéunbalanced} necessary to relate the Sinkhorn divergence and the norm distance after the change of variables.

\begin{lemma}\label{lemma:HcdistleqSinkhorn}
If \(\mu, \nu \in \mcal{P(X)}\) then it holds
\begin{equation*}
\| B(\mu) - B(\nu)\|_{\mcal H_c}^2 \leq \frac2\varepsilon S_\varepsilon\brac{\mu, \nu}.
\end{equation*}
\end{lemma}
\begin{proof}
We have \(\| e^{\frac{f_\mu}{\varepsilon}}\mu-e^{\frac{f_\nu}{\varepsilon}}\nu\|_{\mcal H_c^*}^2 \leq \frac2\varepsilon S_\varepsilon\brac{\mu, \nu}\) from
\cite[Proposition 16]{séjournéunbalanced}, and given the definition of $B$ the result follows as \(H_c : \mcal H_c^* \to \mcal H_c\) is an isometry.
\end{proof}

\noindent We deduce the compactness of the sequence given the scheme~\eqref{eq:SJKO}. 
In this part, to avoid using the cumbersome notation $\dot{\hat{b}}^\tau_t$ for the time derivative of $\hat{b}^\tau$, we rather see $b^\tau$ as a function defined on $[0,+\infty) \times \X$ and write $\partial_t \hat{b}^\tau_t$ for such a temporal derivative.

\begin{lemma}\label{lemma:tauto0compactness}
For \(\tau >0\), let \(\brac{\mu_k^\tau}_k\) be the sequence given by the scheme \eqref{eq:SJKO} with initialization \(\bar{\mu}_0 \in \mcal{P(X)}\), \(b_k^\tau \coloneqq B(\mu_k^\tau)\) its embedding, \(\brac{\sbar b_t^\tau}_{t \geq 0}\) its piecewise constant interpolation, and \( (\hat b_t^\tau)_{t \geq 0}\) its piecewise geodesic interpolation for \(\mathsf d_S\). 

Then for any \(T>0\), as \(\tau \to 0\) and up to a subsequence, \(\sbar b^\tau\) and \( \hat b^\tau\) both converge uniformly on \([0, T]\) for \(\norm{\cdot}_{\mcal H_c}\) to a curve \((b_t)_{t \geq 0}\in\mathscr H^1([0, T]; \mcal H_c)\), and the derivative of \(\hat b^\tau\) converges weakly in \(L^2([0, T]; \mcal H_c)\) to the derivative of \(b\).
\end{lemma}

\begin{proof}
Using the suboptimality of \(\mu_k^\tau\) in the scheme~\eqref{eq:SJKO} gives
\begin{equation*}
E(\mu_{k+1}^\tau) + \frac{S_\varepsilon(\mu_k^\tau, \mu_{k+1}^\tau)}{2\tau} \leq E(\mu_k^\tau).
\end{equation*}
For any \(\ell\), we sum over \(k\leq \ell\) to get from the boundedness of \(E\)
\begin{equation*}
\sum_{k=0}^\ell \frac{S_\varepsilon(\mu_k^\tau, \mu_{k+1}^\tau)}{2\tau} \leq E(\mu_0^\tau) - E(\mu_{\ell+1}^\tau) \leq C
\end{equation*}
where \(C>0\) is a constant independent on $\tau$. We will abusively denote \(C\) various such constants and not bother with their exact expressions for simplicity.
Since we just proved in \cref{lemma:HcdistleqSinkhorn} that \({\norm{b_{k+1}^\tau-b_k^\tau}_{\mcal H_c}^2 \leq \frac2\varepsilon S_\varepsilon(\mu_k^\tau, \mu_{k+1}^\tau)}\) and moreover as \(\mathsf d_S(\mu_{k+1}^\tau, \mu_k^\tau) \leq C\norm{b_{k+1}^\tau-b_k^\tau}_{\mcal H_c}\) (\cref{prop:equivalence_distance_curves}), we deduce 
\begin{equation*}
\sum_{k=0}^\ell \frac{\mathsf d_S(\mu_{k+1}^\tau, \mu_k^\tau)^2}{2\tau} \leq C.
\end{equation*}
Recall that $\hat b^\tau$ is the piecewise geodesic interpolation, so that for $t \in (k \tau, (k+1) \tau)$, the expression 
$\Tilde{\mathbf g}_{\hat b^\tau_t}( \partial_t \hat{b}^\tau_t, \partial_t \hat{b}^\tau_t )$ 
is constant and coincides with $\mathsf d_S(\mu_{k+1}^\tau, \mu_k^\tau)^2 / \tau$. Thus the previous bound implies $\int_0^T\Tilde{\mathbf g}_{\hat b^\tau_t}(\partial_t \hat{b}^\tau_t,\partial_t \hat{b}^\tau_t)\dd t\leq C$. From the bound~\eqref{eq:equivalent_metric_tensor_flat}, we deduce 
\begin{equation*}
\int_0^T \| \partial_t \hat{b}^\tau_t \|_{\mcal H_c}^2\dd t \leq C \int_0^T\Tilde{\mathbf g}_{\hat b^\tau_t} ( \partial_t \hat{b}^\tau_t,\partial_t \hat{b}^\tau_t ) \dd t\leq C.
\end{equation*}
As moreover $\hat b^\tau$ takes its values in the compact set $\mcal B$, standard arguments recalled in \cref{thm:compactness_H1} yield that, up to a subsequence, $(\hat b^\tau)_\tau$ converges uniformly to a limit $b$, and $(t\mapsto \partial_t \hat{b}^\tau_t )_\tau$ converges weakly in $L^2([0, T];\mcal H_c)$ to $\dot{b}$. Morrey's inequality \eqref{eq:morrey} yields $\| \hat{b}^\tau_t - \sbar b^\tau_t \|_{\mcal H_c} \leq C  \sqrt{\tau}$ uniformly in $t$, thus the piecewise constant interpolation \(\sbar b^\tau\) converges uniformly to the same limit $b$.
\end{proof}

We can now conclude the proof of \cref{thm:tauto0} by showing that the limit given by \cref{lemma:tauto0compactness} is the Sinkhorn potential flow in the case of a finite space. We now assume \(\X = \cbrac{x_1, \dots, x_n}\), identify \(\MX = \CX = \mcal H_c = \R^n\), and operators on these spaces are identified with \(n\times n\) matrices. We will omit the topology considered when dealing with limits in these spaces since all weak and norm topologies become equivalent.

\begin{proof}[\textbf{Proof of \cref{thm:tauto0}}]
With the same notations as in \cref{thm:tauto0}, and by \cref{lemma:tauto0compactness} we already have convergence of $\hat b^\tau$, $\partial_t \hat{b}^\tau_t$ and $\bar{b}^\tau$ up to subsequences to a limit. We write $\mu_t = B^{-1}(b_t)$ and $\hat\mu^\tau_t = B^{-1}(\hat{b}^\tau_t)$ for any $t \geq 0$.  

The optimality conditions of the scheme~\eqref{eq:SJKO} were derived in \cref{lm:optimality_SJKO}, they read as an equality in $\CX / \R$:
\begin{equation*}
\frac{1}{2\tau}\brac{f_{\mu_{k+1}^\tau, \mu_k^\tau} - f_{\mu_{k+1}^\tau}} + V + p_{k+1}^\tau = 0
\end{equation*}
for some \(p_{k+1}^\tau \in P \mu_{k+1}^\tau\). We write the piecewise interpolations
\begin{equation*}
\bar q_t^\tau \coloneqq \frac{1}{2\tau}\brac{f_{\mu_{k+1}^\tau, \mu_k^\tau} - f_{\mu_{k+1}^\tau}} \qquad \text{for } t\in[k\tau, (k+1)\tau),
\end{equation*}
as well as
\begin{equation*}
\bar p_t^\tau \coloneqq p_{k+1}^\tau\ \qquad \text{for } t\in[k\tau, (k+1)\tau),
\end{equation*}
so that we have $\bar q_t^\tau + V + \bar p^\tau_t = 0$ in $\CX/ \R$, at least for $t \in [k\tau, (k+1)\tau)$. We now pass this equation to the limit $\tau \to 0$.

\medskip 

\noindent \emph{Limit of $\bar q^\tau$.}
Writing \(f_{t, s}^\tau \coloneqq f_{\hat\mu_t^\tau, \hat\mu_s^\tau}\) as well as $H_{t,s}^\tau = H_{\hat\mu^\tau_t, \hat\mu^\tau_s}$ and $K_{t,s}^\tau = K_{\hat\mu^\tau_t,\hat\mu^\tau_s}$, we write the difference of Schrödinger potentials as the integral of its derivative i.e.
\begin{equation*}
\bar q_t^\tau = \frac{1}{2\tau}\brac{f_{(k+1)\tau, k\tau}^\tau - f_{(k+1)\tau, (k+1)\tau}^\tau} = -\frac{1}{2\tau}\int_{k\tau}^{(k+1)\tau}\frac{\partial f^\tau_{(k+1)\tau, s}}{\partial s}\dd s.
\end{equation*}
Since there is no ambiguity between the functional spaces at hand (\cref{rem:finiteX}), \cref{prop:dfts/ds} applies which combined with \(H^\tau_s[\partial_s \hat\mu_s^\tau] = (\id + K^\tau_s)[(b_s^\tau)^{-1}\partial_s \hat{b}^\tau_s]\) (\cref{prop:change_mudot_bdot}) yields
\begin{equation*}
-\frac 12\frac{\partial f^\tau_{(k+1)\tau, s}}{\partial s} = J^\tau_{(k+1)\tau, s}[\partial_s \hat b^\tau_s] 
\end{equation*}
where the $n \times n$ matrix $J^\tau_{t,s}$ is defined as:
\begin{equation}
\label{eq:defJtau}
J_{t, s}^\tau \coloneqq \frac\varepsilon2\brac{\id - K^\tau_{t, s}K^\tau_{s,t}}\pinv H^\tau_{t, s}(H^\tau_s)^{-1}(\id+K^\tau_s)\mathrm{diag} (\hat b^\tau_s)^{-1},
\end{equation}
Here \(\mathrm{diag} (\hat b^\tau_s)\) is the diagonal matrix with diagonal coefficients given by \(\hat b^\tau_s\), and \(M\pinv\) denotes the Moore-Penrose pseudoinverse of a matrix $M$, which appears due to the fact that \(\id - K^\tau_{t, s}K^\tau_{s,t}\) is only invertible in \(\CX/\R\). For some \(t_\tau\) converging to \(s\) as \(\tau \to 0\), we have \(H^\tau_{t_\tau, s} \to H_{s,s} \coloneqq H_{\mu_s, \mu_s}\): indeed \(\hat \mu^\tau\) converges uniformly to \(\mu\) (by applying \(B^{-1}\), which is a continuous function between compact sets, to the convergence of \(\hat b^\tau\)) and the Schrödinger potentials vary continuously with respect to the measures \cite[Proposition 13 (Appendix B)]{feydy19}. The same reasoning gives \(K^\tau_{t_\tau, s}\to K_{s,s}\coloneqq K_{\mu_s, \mu_s}\), \(K^\tau_{s, t_\tau}\to K_{s,s}\), \(K_s^\tau \to K_s\coloneqq K_{\mu_s}\), \(H_s^\tau \to H_s\coloneqq H_{\mu_s}\), and we also directly get \(\mathrm{diag}(\hat b^\tau_s) \to \mathrm{diag}(b_s)\) by the convergence of \(\hat b^\tau\). With the continuity of the inverse and that of the pseudoinverse for sequences of constant rank in finite dimension (the constant rank assumption being a consequence of \cref{prop:K_munu})
we have that the curve defined by
\(
\widehat J_s^\tau \coloneqq J^\tau_{\brac{\floor{s/\tau} + 1}\tau, s}
\)
converges pointwise as \(\tau\to0\) to 
\begin{equation*}
J_s = \frac\varepsilon2\brac{\id - K_{s,s}K_{s,s}}\pinv H_{s,s}H_{s,s}\inv(\id+K_{s})\diag{b_s}^{-1} = \frac\varepsilon2\brac{\id - K_s}\pinv\diag{b_s}^{-1}
\end{equation*}
We call $q_t$ the resulting expression when applied to $\dot{b}_t$, that is, given \cref{prop:change_mudot_bdot}, 
\begin{equation*}
q_t \coloneqq J_t[\Dot b_t] = \frac{\varepsilon}{2} (\id - K_{\mu_t}^2)^{-1} H_{\mu_t} [\Dot \mu_t].
\end{equation*}
We claim that the curve $(\bar q^\tau_t)_t$ converges weakly in \(L^2([0, T]; \mcal H_c)\) to \( (q_t)_t\). Take a continuous curve \((\phi_t)_t \in \mcal C([0, T]; \mcal H_c)\)  and compute
\begin{equation*}
\int_0^T \lrangle{\phi_t, \sbar{q}_t^\tau}_{\mcal H_c}\dd t = \sum_{0\leq k \tau \leq T} \frac1\tau\int_{k\tau}^{(k+1)\tau}\int_{k\tau}^{(k+1)\tau}\langle \phi_t, \widehat J_s^\tau[\partial_s \hat{b}^\tau_s] \rangle_{\mcal H_c}\dd t\dd s =\int_0^T\langle \psi_s^\tau, \partial_s \hat{b}^\tau_s \rangle_{\mcal H_c}\dd s
\end{equation*} 
where
\begin{equation*}
\psi^\tau_s \coloneqq \frac1\tau\brac{\widehat J_s^\tau}^*\sqbrac{ \int_{k\tau}^{(k+1)\tau}\phi_t \dd t} \quad \text{for }  s\in [k\tau, (k+1)\tau)
\end{equation*}
and \(M^*\) is the transpose of a matrix \(M\). The sequence \((\psi^\tau)_\tau\) is uniformly bounded and converges pointwise to \(s\mapsto J_s^*[\phi_s]\) (by the convergence of \(s\mapsto \widehat J_s^\tau\) and the fundamental theorem of calculus). The Lebesgue dominated convergence theorem gives strong convergence of the sequence in \(L^2([0, T]; \mcal H_c)\) and thus the duality pairing with \(s\mapsto \partial_s \hat{b}_s^\tau\) converges. 
Therefore, we have convergence of \(\lrangle{\phi, \bar q^\tau}_{L^2([0, T]; \mcal H_c)}\) to \(\lrangle{\phi, q}_{L^2([0, T]; \mcal H_c)}\) for all continuous curves \(\phi\), and using the density of such curves in \(L^2([0, T]; \mcal H_c)\) with the boundedness of the sequence $\bar q^\tau$  gives the weak limit.

\medskip

\noindent \emph{Limit of $\bar p^\tau$.}
Regarding the piecewise constant interpolation of the pressure $(\bar p_t^\tau)_t$, that can be written \(\bar p_t^\tau = -V-\bar q_t^\tau\), from the paragraph above we have convergence of the curves \((\bar p^\tau_t)_t\) to \(p = (-V - q_t)_t\) weakly in \(L^2([0,T]; \mcal H_c)\). Thus, showing that \(p_t\leq 0\) and \(\lrangle{\mu_t, p_t} = 0\) for a.e. \(t\) will give the desired result. That $p_t \leq 0$ is easy: indeed here the non-positivity constraint corresponds to $p_t \in (\R_-)^\X$, and the set \(L^2([0,T];  (\R_-)^\X)\) is closed and convex. As closed convex sets are stable by weak convergence, we deduce that $p \in L^2([0,T]; (\R_-)^\X)$. Next, since $\langle \bar{\mu}^\tau_t, \bar p^\tau_t \rangle = 0$ for all $t$ and $\bar{\mu}^\tau$ converges strongly in $L^2([0,T]; \mcal \PX)$, we can pass to the limit and deduce $\langle \mu_t, p_t \rangle = 0$ for a.e. $t$. Thus $p_t \in P \mu_t$ for a.e. $t$, which is sufficient to conclude that the limit $(\mu_t)_t$ satisfies the equation~\eqref{eq:sinflow_mu}, hence is a Sinkhorn potential flow. 

\medskip

Ultimately, as the Sinkhorn potential flow is unique (\cref{thm:main}), the limit not only holds along a subsequence but for the whole sequence as $\tau \to 0$.
\end{proof}

\begin{remark}[Obstructions in the general case]\label{rem:obstructions}
Taking the limit of the optimality conditions in the general setting is much more difficult due to the lack of homogeneity between the tangent spaces. Indeed, we have proven \cref{prop:dfts/ds} when the derivative \((\Dot \mu_t)_t\) is always valued in \(\MX\) which is always the case when all spaces are simply \(\R^\X\). However, generally speaking, we would have \(\Dot \mu_t \in \tgtspace{\mu_t}\) when building the curve \((\mu_t)_t\) from the geodesic interpolation as we have done, and the operator \(H_{\mu, \nu}\) from \cref{def:kmunu} is at most extended to \(\mcal H_{\nu, \mu}^*\) which will not usually contain \(\tgtspace{\mu}\). The expression on the right hand side of \eqref{eq:dfts/ds} is thus ill-defined for the more general class of curves. This is also apparent in the definition of $J^\tau$ in~\eqref{eq:defJtau}: it features the term $H^\tau_{t, s}(H^\tau_s)^{-1}$ which at the limit $\tau \to 0$ (hence $t \to s$) behaves like the identity, but may be very much unbounded for $\tau > 0$ for non-finite spaces.

Still, the expected limit is known to be well defined from \cref{thm:main}, and moreover it is retrieved as the limit \(n\to +\infty\) of the discretized space case. The limit results that we have are summarized by the diagram in \cref{fig:limitdiagram} below. Proving that \cref{thm:tauto0} holds for a general compact space $\X$ is a challenging and open question. 
\end{remark}

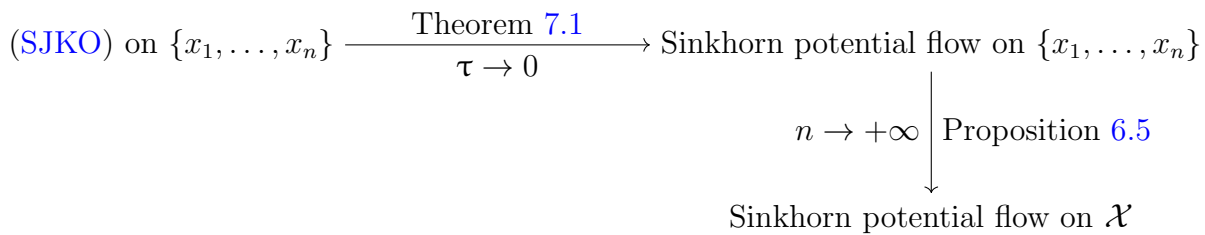
\begin{figure}[H]
\centering
\begin{tikzpicture}
\node (SJKOn) at (0, 0) {\eqref{eq:SJKO} on \(\cbrac{x_1, \dots, x_n}\)};
\node (SPFn) at (10, 0) {Sinkhorn potential flow on \(\cbrac{x_1, \dots, x_n}\)};
\node (SPF) at (10, -2.3) {Sinkhorn potential flow on \(\mcal X\)};
\draw[->] (SJKOn.east) -- (SPFn.west) node[midway, below]{\(\tau \to 0\)} node[midway, above]{\cref{thm:tauto0}};
\draw[->] (SPFn.south) -- (SPF.north) node[midway, left]{\(n\to + \infty\)} node[midway, right]{\cref{prop:flow_ntoflow}};
\end{tikzpicture}
\caption{Summary of limit results obtained the present work.}
\label{fig:limitdiagram}
\end{figure}

\section{Numerics}\label{sec:numerics}

In this section we provide numerical illustrations of the peculiar behavior of the Sinkhorn potential flow on simple examples. The full code, parameters used and animations are available online\footnote{\url{https://github.com/mhardion/sinkhorn_potential_flows}}. Some code was adapted from the Geomloss package \cite{feydy19}, and we use the library PyTorch \cite{torch} for its automatic differentiation capabilities. We first detail the algorithm that was used before moving to the results.

\subsection{Numerical scheme}
We explain how to compute a single \cref{eq:SJKO} step starting from \(\mu_k^\tau\in\PX\), i.e. minimize the objective
\begin{equation}\label{eq:SJKOfunctional}
\mu \mapsto 2\tau\lrangle{\mu, V} + S_\varepsilon(\mu, \mu_k^\tau)
\end{equation}
over \(\PX\), which has the same minimizer as in \eqref{eq:SJKO} but avoids dividing by a small \(\tau\) which could be numerically unstable. In the numerics we must work with discrete measures \(\mu = \sum_{i=1}^n a_i\delta_{x_i}\), where we have two main schemes: the Eulerian scheme, where the positions \((x_i)_i\) are fixed and the weights \((a_i)_i\) are the variables (which amounts to setting \(\X = \{x_1, \dots, x_n\}\)), and conversely the Lagrangian scheme, where the weights $(a_i)_i$ are set to \(1/n\) and the positions $(x_i)_i$ are free. To compute \(\OT_\varepsilon\) appearing in the Sinkhorn divergence, the basic method is the Sinkhorn algorithm given by iterating alternatively the two conditions in the Schrödinger system \eqref{eq:schrödingersystem}, which converges to the solution of the dual EOT problem \eqref{eq:dualOTeps} \cite[Theorem 4.2]{computationalot}. For the self transport of a measure, a more simple fixed-point algorithm can be utilized \cite[Section 3.1]{feydy19}. We use as initialization the values of the Schrödinger potentials at the previous step since they should be relatively close to the optimal ones. The gradients of the Sinkhorn divergence are easily deduced from the dual potentials computed by these algorithms \cite[Section 3.2]{feydy19}, which will allow us to perform each SJKO step through a first order method such as gradient descent. When considering a Eulerian discretization, the objective \eqref{eq:SJKOfunctional} is convex in the weights (see \cref{thm:feydythm1}) with Lipschitz gradient \cite{carlier2024displacement}, which classically yields convergence guarantees of such methods to the optimum if the gradients are computed accurately enough. The drawback is the need to enforce the simplex constraint $\sum_i a_i = 1$ and $a_i \geq 0$, which slows down convergence. When considering a Lagrangian discretization, the convexity is lost but the need to enforce constraints disappears. We summarize in \cref{alg:sjko} below the resulting algorithm.
\begin{algorithm}[h]

\SetKwInOut{Input}{Input}\SetKwInOut{Output}{Output}
\Input{measure \(\mu_k^\tau\), potential \(V\)}

Initialize \(\mu_{k+1}^\tau \leftarrow \mu_k^\tau\) \;
\Repeat{a convergence criterion is met}{
Compute \(f_{\mu_{k+1}^\tau, \mu_k^\tau}\) and \(f_{\mu_{k+1}^\tau}\) using Sinkhorn's algorithm \cite[Section 3.1]{feydy19}\;
\If{the discretization is Eulerian}{
  Compute the gradient of \eqref{eq:SJKOfunctional} w.r.t. the masses as \(2\tau V + f_{\mu_{k+1}^\tau, \mu_k^\tau} - f_{\mu_{k+1}^\tau}\)\;
  Perform a step of a gradient-based optimization algorithm, constrained to the probability simplex, to update the weights of \(\mu_{k+1}^\tau\)\;
}
\If{the discretization is Lagrangian}{
  Compute the gradient of \eqref{eq:SJKOfunctional} w.r.t. the positions using automatic differentiation \cite[Section 3.2]{feydy19}\;
  Perform a step of a gradient-based optimization algorithm to update the positions of the atoms of \(\mu_{k+1}^\tau\)\;
}
}
\Output{\(\mu_{k+1}^\tau\) approximating the minimizer of \eqref{eq:SJKOfunctional}}
\caption{Compute a single \cref{eq:SJKO} step}\label{alg:sjko}
\end{algorithm}
For a Eulerian discretization, the convergence criterion that we used is based on the optimality conditions from \cref{lm:optimality_SJKO} (note that they can be shown to be sufficient since the objective is convex).
Observe that the constant \(C_k^\tau\) in \eqref{eq:optimality_SJKO} can be computed by integrating against \(\mu_{k+1}^\tau\), giving \(C_k^\tau = \langle \mu_{k+1}^\tau, g_k^\tau\rangle\) where \(g_k^\tau \coloneqq V + (f_{\mu_{k+1}^\tau, \mu_k^\tau} - f_{\mu_{k+1}^\tau})/(2\tau)\). Thus, rewriting \eqref{eq:optimality_SJKO} as \(-p_k^\tau = g_k^\tau - \langle \mu_{k+1}^\tau, g_k^\tau\rangle\), the optimality conditions may be written \(g_k^\tau - \langle \mu_{k+1}^\tau, g_k^\tau\rangle \geq 0\). We can thus stop iterating when the inequality is valid up to a small tolerance.

For a Lagrangian discretization, the optimal criterion used must be different since one cannot expect to get to the global minimum of the objective functional \eqref{eq:SJKOfunctional} with a first order algorithm in general. Indeed when \(V\) is not convex, the functional is not convex in the positions \((x_i)_i\). Instead of the global optimality criterion that we have seen above, we may simply stop the algorithm when the gradient of the objective with respect to the positions is smaller than a tolerance or after a maximum number of iterations.

For the gradient based algorithm, we settled on accelerated projected gradient descent (a special case of FISTA \cite{beck2009fast}) in the Eulerian case, and on Nesterov accelerated gradient descent \cite{nesterov1983method} in the Lagrangian case. It should be noted that for the former, convergence is quite slow in our experiments, most likely due to the high dimensionality of the problem when taking a decently fine discretization of the space. The development of more efficient solvers of \cref{eq:SJKO} is left for future works.

In our experiments below we always take $\X \subset \R^d$ and $c$ to be the squared Euclidean distance.

\subsection{Visualizing the constrained rotation}\label{sec:numerics:sphere}

When the base space is made of three points i.e. \(\X=\{x_1, x_2, x_3\}\), the RKHS is 3-dimensional and it is thus possible to visualize the unit sphere and the embeddings \(b_k^\tau = \exp(-f_{\mu_k^\tau}/\varepsilon)\) of the scheme. The operator \(H_c\) is a \(3\times 3\) matrix, which is symmetric and positive definite since \(k_c\) is symmetric and strictly positive definite. The change of basis provided by \(H_c^{-\frac12}\) is an isometry between \((\mcal H_c, \langle \cdot, \cdot \rangle_{\mcal H_c})\) and \((\R^3, \langle\cdot,\cdot\rangle)\), we can thus use it to embed the RKHS sphere into the the Euclidean sphere to better visualize it. The axis of the rotation corresponding to the operator \(W\) can also be computed easily to compare the simulation with the theoretical evolution. The results displayed in \cref{fig:sphere} above in Page~\pageref{fig:sphere} qualitatively show good correspondance with the theory.

\subsection{Convex potential: collapsing behavior}

Consider a measure that has mass distributed within a ball of small diameter compared to \(\sqrt{\varepsilon}\). Intuitively, this is the regime where \(S_\varepsilon\) behaves similarly to a Maximum Mean Discrepancy \cite[Theorem 1]{genevay18a}, which is the square Euclidean distance between the barycenters of the input measures (with our choice of \(c\)), and the metric tensor behaves like the squared norm of the speed of the barycenter \cite[Theorem 4.17]{RGSD}. The cost of collapsing the mass to its barycenter is therefore small, whereas it makes the potential energy decrease by convexity of the potential. Therefore, one can expect the Sinkhorn potential flow to collapse mass to its barycenter when it accumulates in a ball of diameter comparable to \(\sqrt{\varepsilon}\). This behavior is observed in the numerics in \cref{fig:mass_collapse} (in 1D) and in \cref{fig:particle_collapse} (in 2D) below.

\begin{figure}[ht]
\centering
\begin{tabular}{@{\hspace{-.5cm}}c@{\hspace{.02cm}}cccc}

\begin{tikzpicture}
\draw (0, 0) node[rotate=90]{\footnotesize\(\varepsilon=0.04\) (Eulerian)};
\end{tikzpicture} &\includegraphics[width=.2\textwidth]{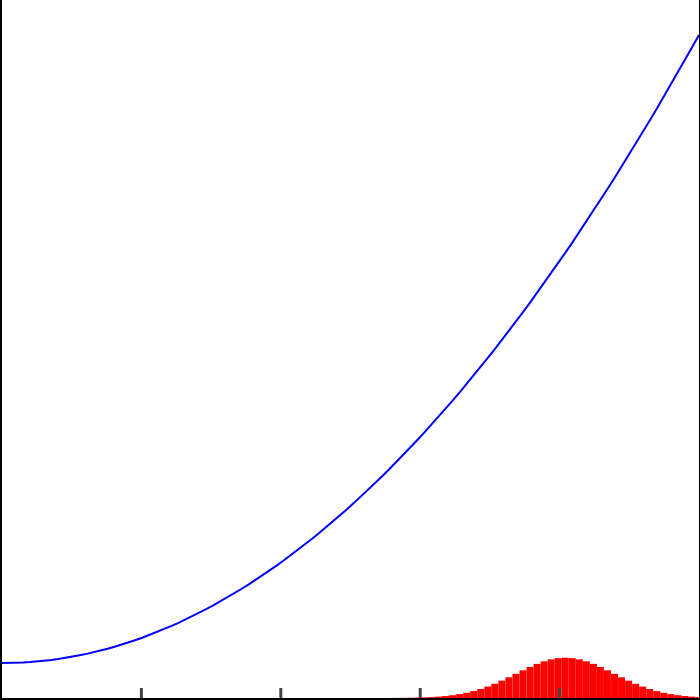} & \includegraphics[width=.2\textwidth]{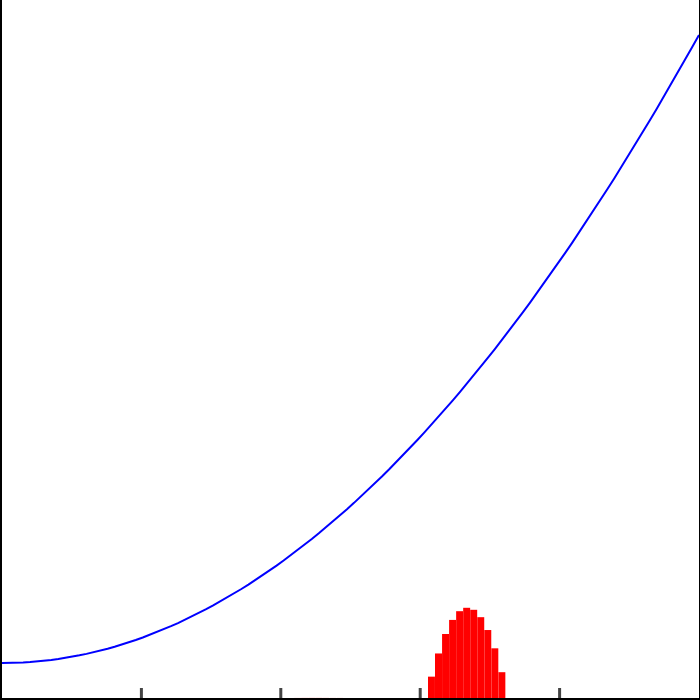} & \includegraphics[width=.2\textwidth]{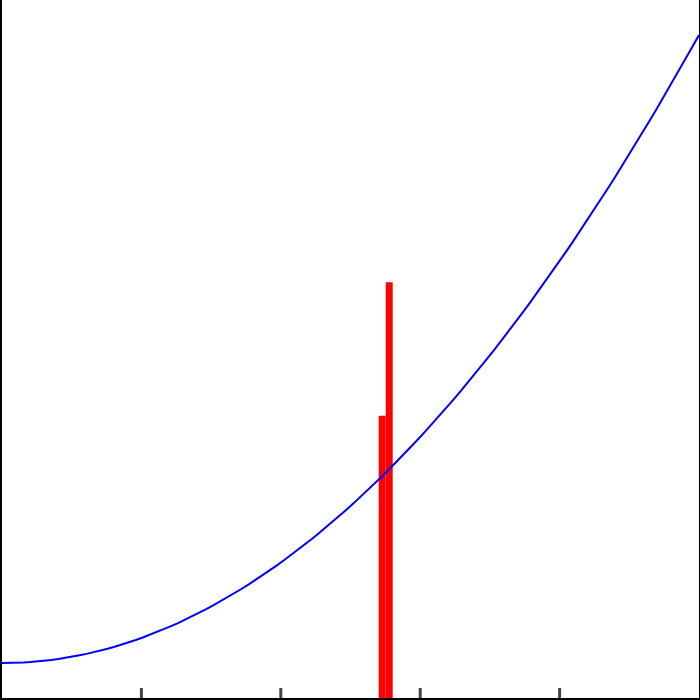} & \includegraphics[width=.2\textwidth]{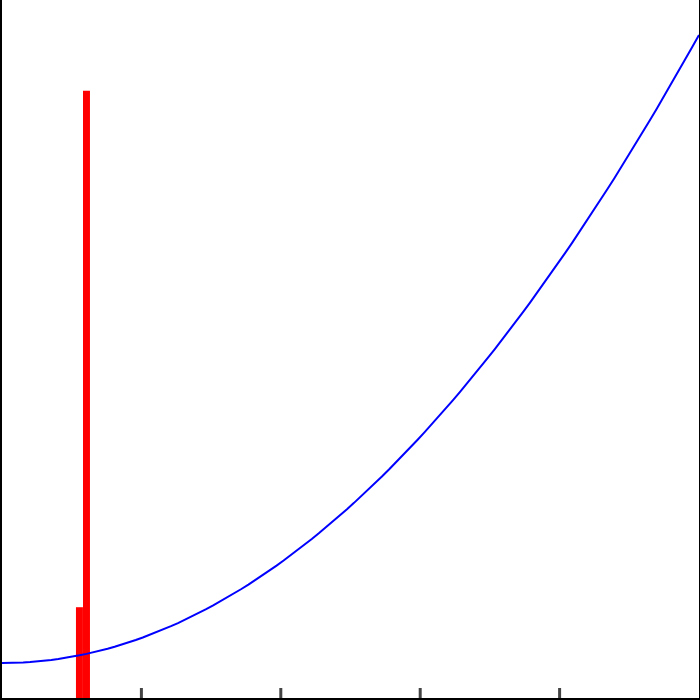}\\
\raisebox{-.1cm}{
\begin{tikzpicture}
\draw (0, 0) node[rotate=90]{\footnotesize\(\varepsilon=0.04\) (Lagrangian)};
\end{tikzpicture}}&\includegraphics[width=.2\textwidth]{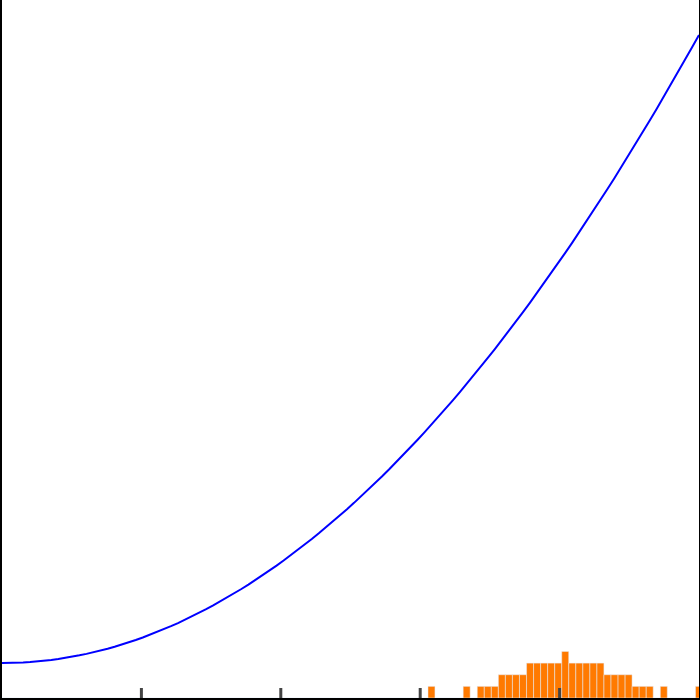} &  \includegraphics[width=.2\textwidth]{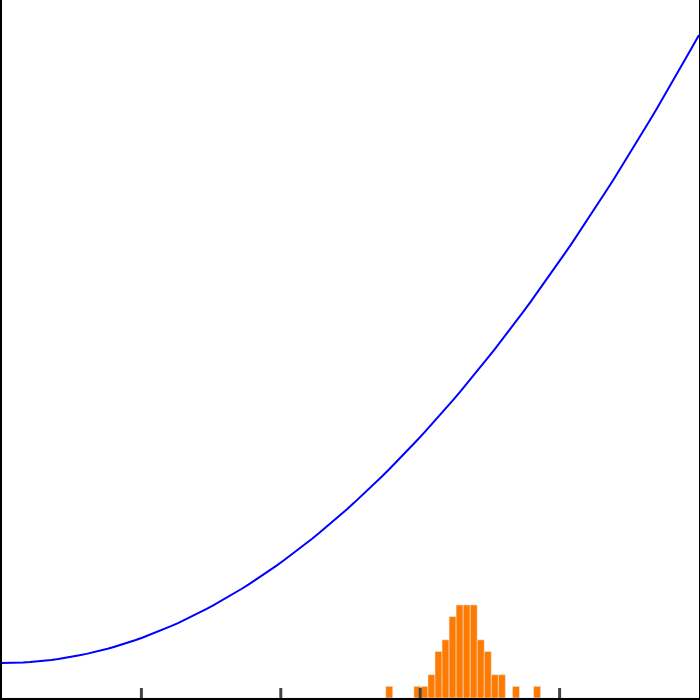} & \includegraphics[width=.2\textwidth]{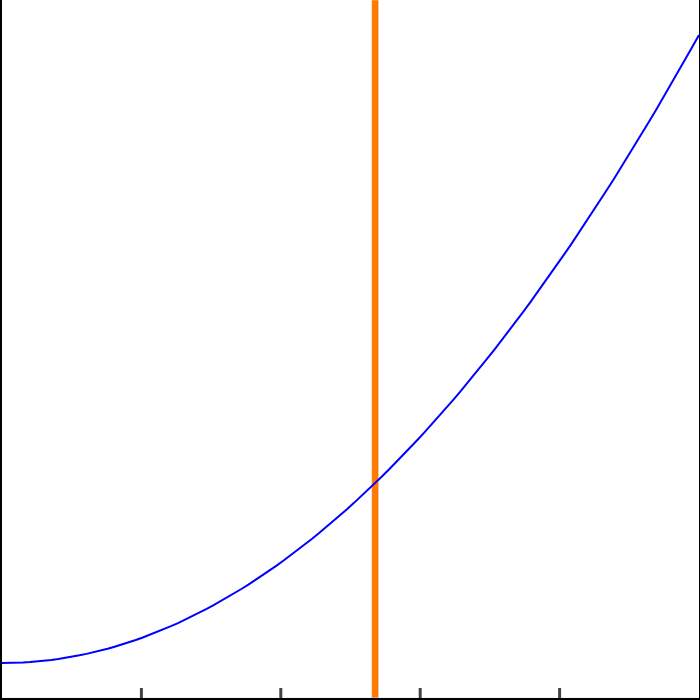} & \includegraphics[width=.2\textwidth]{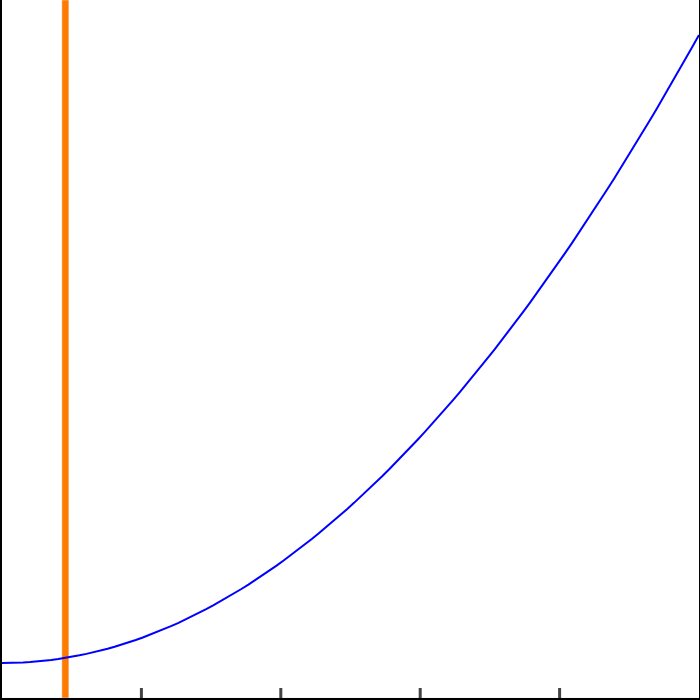}\\

\begin{tikzpicture}
\draw (0, 0) node[rotate=90]{\footnotesize\(\varepsilon=0\) (Wasserstein)};
\end{tikzpicture}&\includegraphics[width=.2\textwidth]{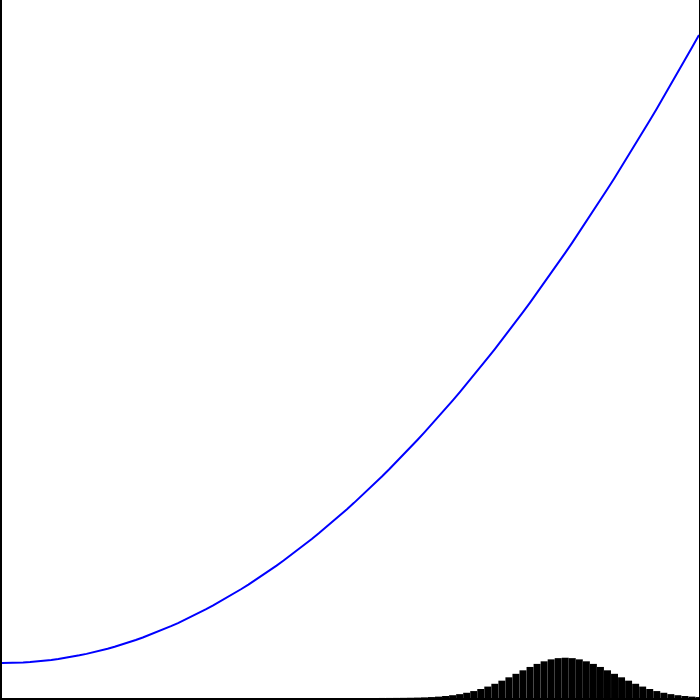} &  \includegraphics[width=.2\textwidth]{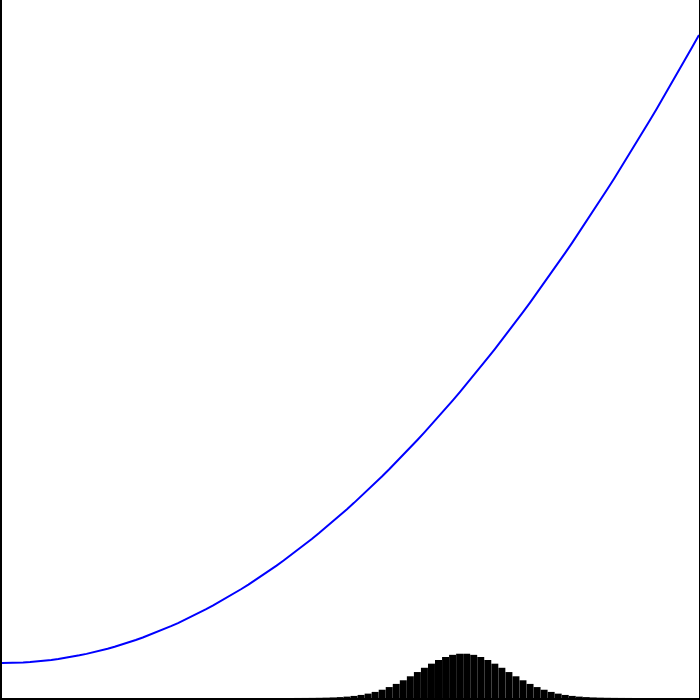} & \includegraphics[width=.2\textwidth]{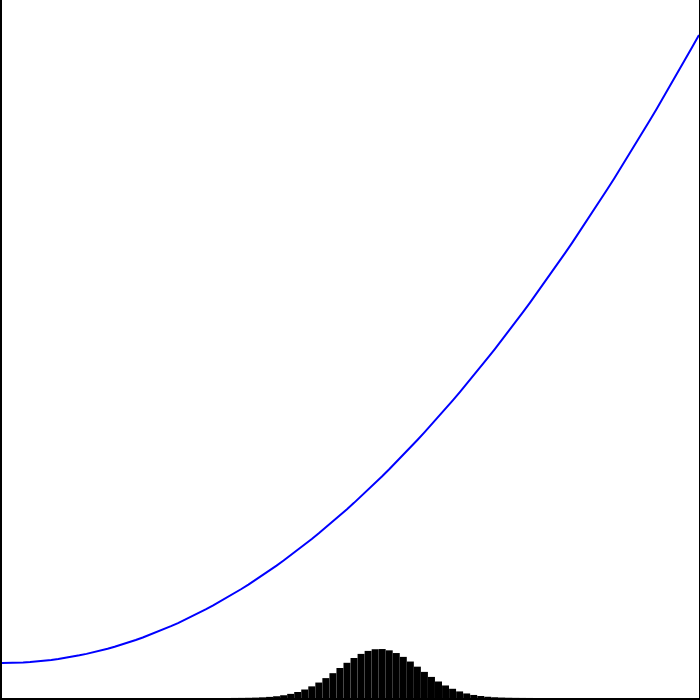} & \includegraphics[width=.2\textwidth]{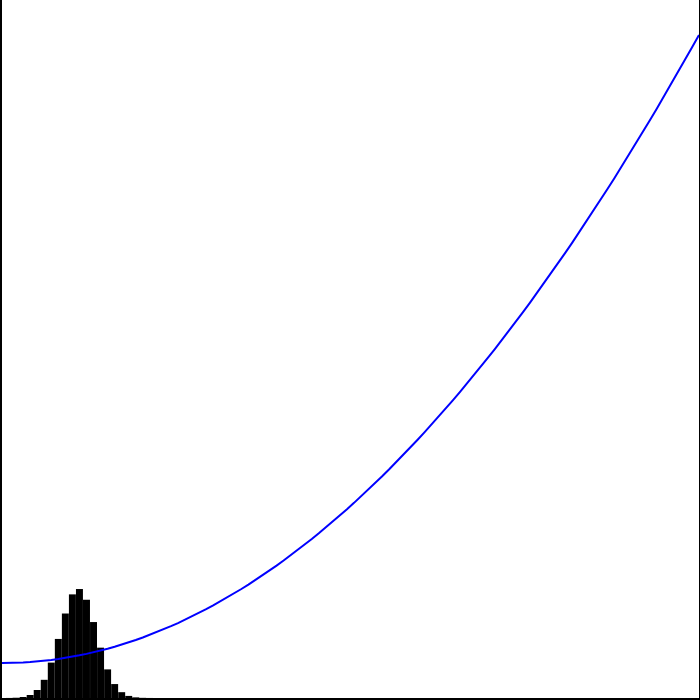}\\
&\(t=0.00\) & \(t=0.1\) & \(t=0.2\) & \(t=1\)
\end{tabular}
\caption{Sinkhorn potential flows (red for the Eulerian discretization at the top, orange for the Lagrangian one in the middle) for \(V(x) = x^2\) (blue plot) on \(\X = [0, 1]\) compared to the Wasserstein flow (black, at the bottom), all with the same Gaussian initial condition. For the Sinkhorn flows, the ticks on the x-axis are spaced by \(\sqrt{\varepsilon}\) to give a sense of the characteristic distance. In the Lagrangian discretization, we represent the empirical histogram of the particles. The Wasserstein flow is computed in closed form, then discretized to match the other figures.}
\label{fig:mass_collapse}
\end{figure}
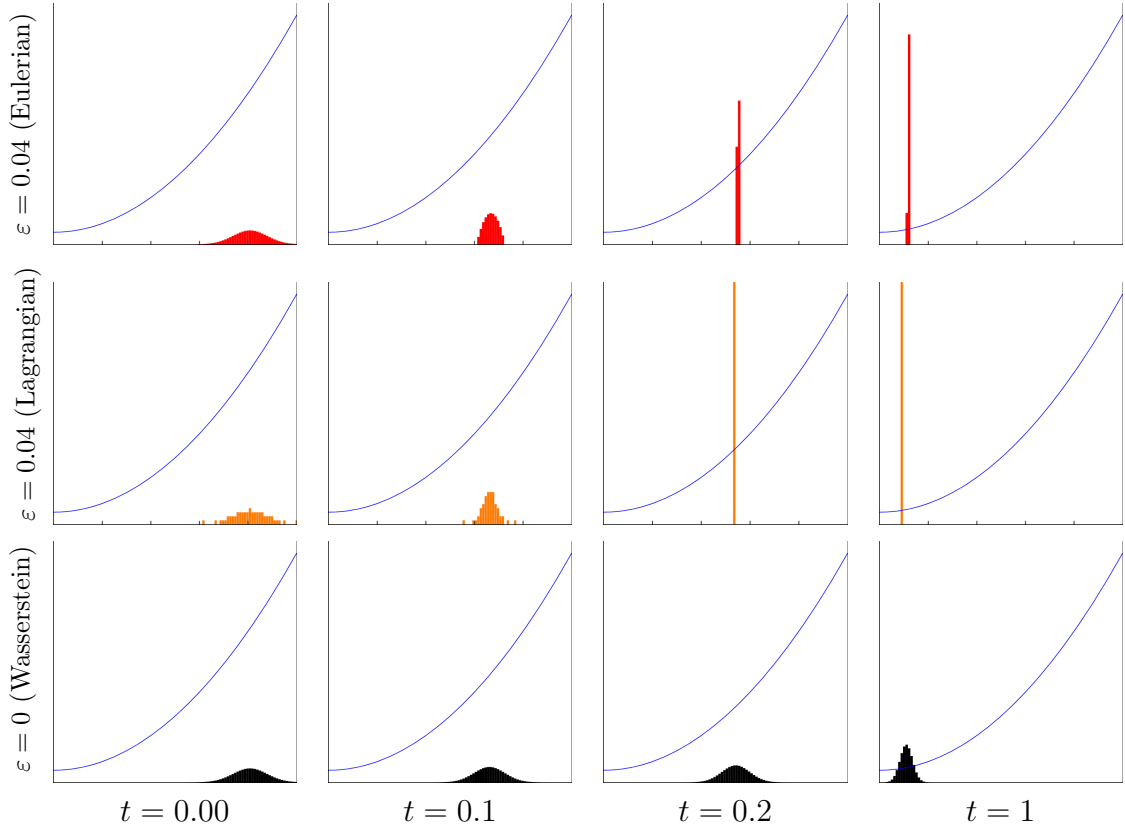

\begin{figure}[h]
\centering
\begin{tabular}{@{\hspace{-.4cm}}c@{\hspace{.02cm}}cccc}
\raisebox{.7cm}{
\begin{tikzpicture}
\draw (0, 0) node[rotate=90]{\(\varepsilon=0.04\)};
\end{tikzpicture}}&\includegraphics[width=.22\textwidth]{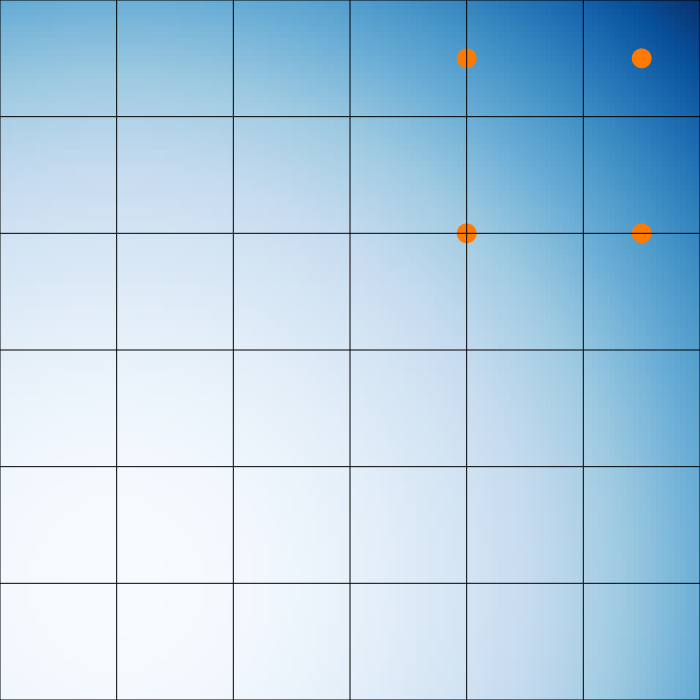} &  \includegraphics[width=.22\textwidth]{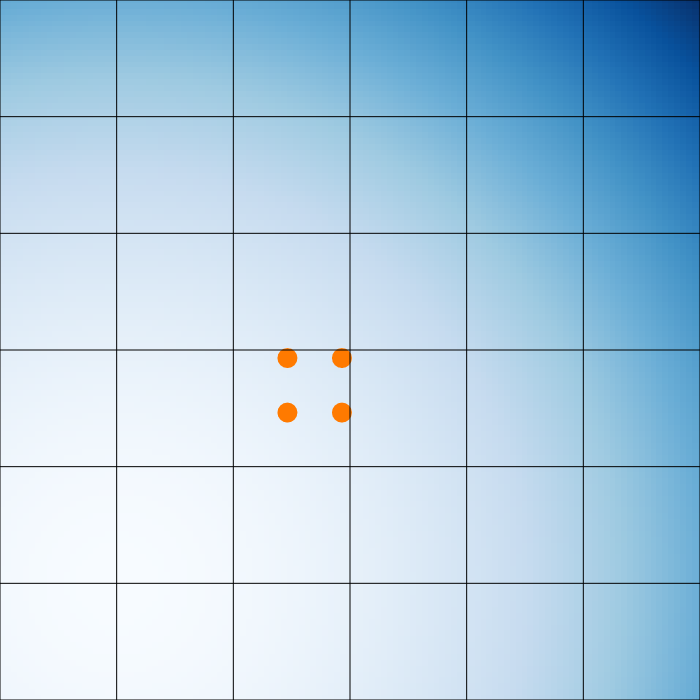} & \includegraphics[width=.22\textwidth]{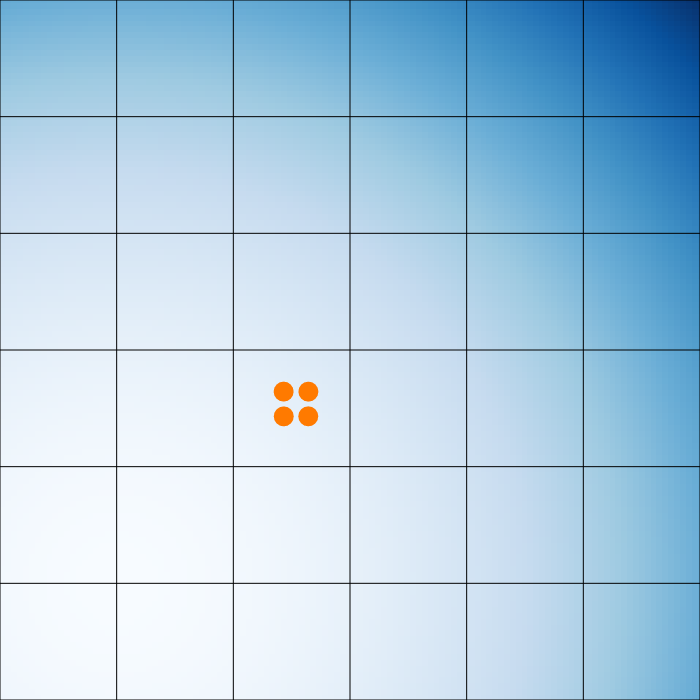} & \includegraphics[width=.22\textwidth]{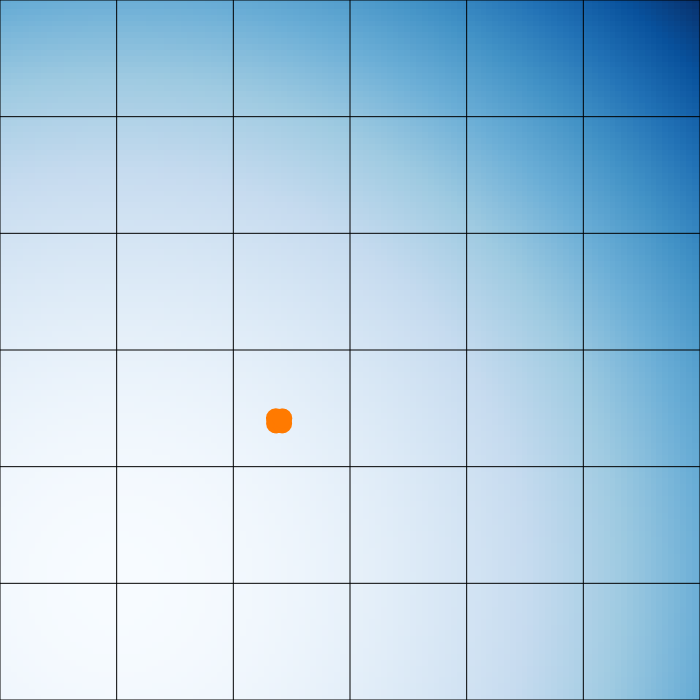}\\
\raisebox{.8cm}{
\begin{tikzpicture}
\draw (0, 0) node[rotate=90]{\(\varepsilon=0\)};
\end{tikzpicture}} &\includegraphics[width=.22\textwidth]{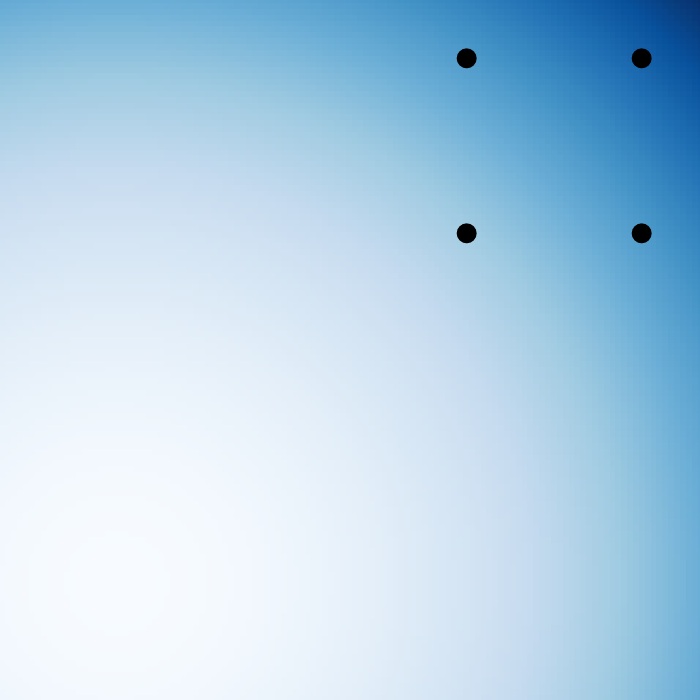} & \includegraphics[width=.22\textwidth]{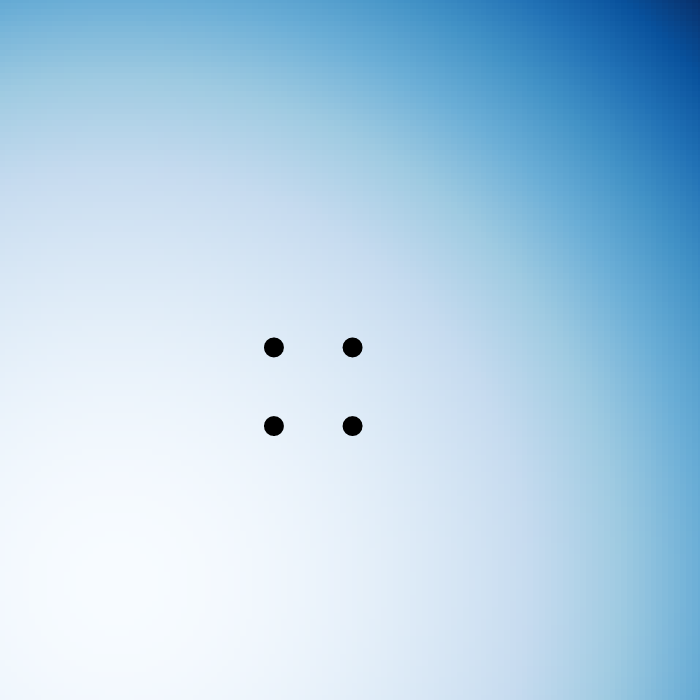} & \includegraphics[width=.22\textwidth]{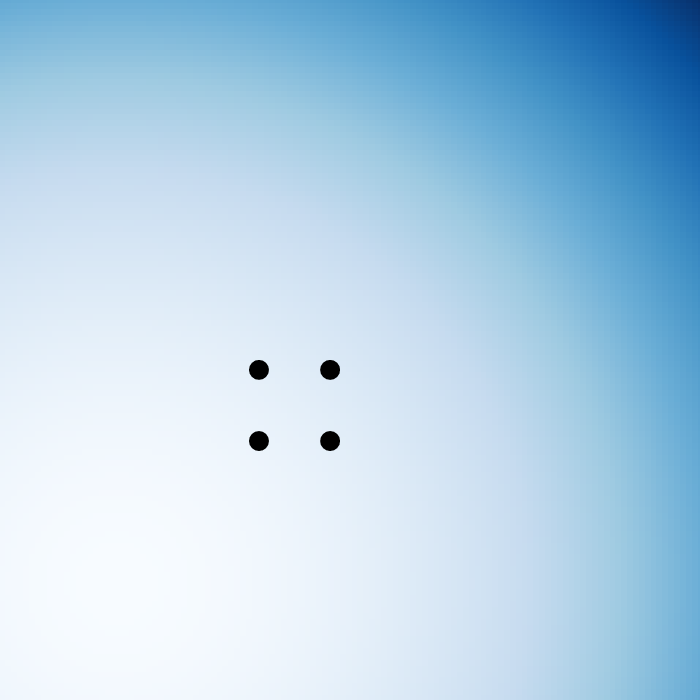} & \includegraphics[width=.22\textwidth]{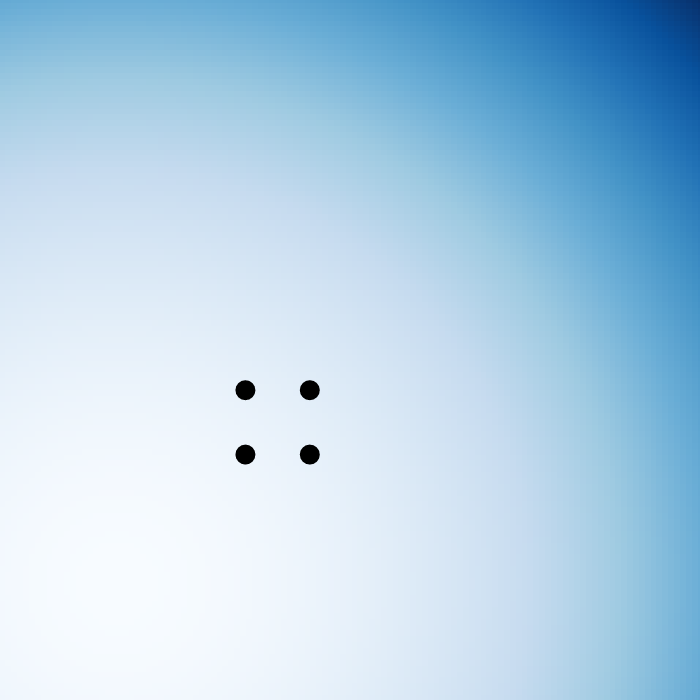}\\
&\(t=0.00\) & \(t=0.40\) & \(t=0.45\) & \(t=0.50\)
\end{tabular}
\caption{Lagrangian flows in 2D space for \(V = \norm{\cdot}^2\) (white to blue heat map), for the Sinkhorn divergence with \(\varepsilon =0.04\) (top) vs. the Wasserstein distance (bottom). For the latter, the grid lines are spaced by \(\sqrt{\varepsilon}\) to give a sense of the characteristic distance.}
\label{fig:particle_collapse}
\end{figure}

\subsection{Non-convex potential: vertical perturbations}

In a Eulerian discretization scheme, vertical perturbations are possible (and even the only possibility), while they are not in the Lagrangian case since the weights are fixed. As a result, teleportation through potential barriers and convergence towards the global minimum is observed for the former, while the particles get stuck in local minima for the latter (even with large \(\varepsilon\)). Both the Eulerian and the Lagrangian cases are displayed in \cref{fig:eul_lag_ncvx}. 
For the former, we also observe the behavior described intuitively as follows. The flow first behaves as in the convex case, until it reaches the potential barrier. Then a vertical perturbation appears, where mass ``tunnels'' a distance comparable to \(\sqrt{\varepsilon}\) through the potential barrier. This mass then moves horizontally towards the minimum, leaving some behind (this is where there is a split of mass, at \(t=1.25\) in \cref{fig:eul_lag_ncvx}).
The transferred mass then starts to accumulates towards the global minimum, and there remain three focal points of mass : near the global minimum, near the first apparition of a vertical perturbation, and at the local minimum. Our explanation for this is that it is costlier to make mass appear outside the support of the current measure, than to transfer mass within its support, as illustrated in \cref{fig:vertical_perturbation_cost} for the 2-point space. Therefore, the mass in the middle acts as a ``bridge'' between the local and global minimum.

\begin{remark}
In this example we see that the Lagrangian discretization does not adequately approximate solutions of~\eqref{eq:SJKO}, or of the Sinkhorn potential flow. This is because the flow teleports mass, something that a Lagrangian discretization cannot capture. Only the Eulerian discretization faithfully approximates the flow, indicating that its computation in high dimension is a challenging problem.
\end{remark}

\begin{figure}[h]
\centering
\begin{tabular}{cccc}
\includegraphics[width=.22\textwidth]{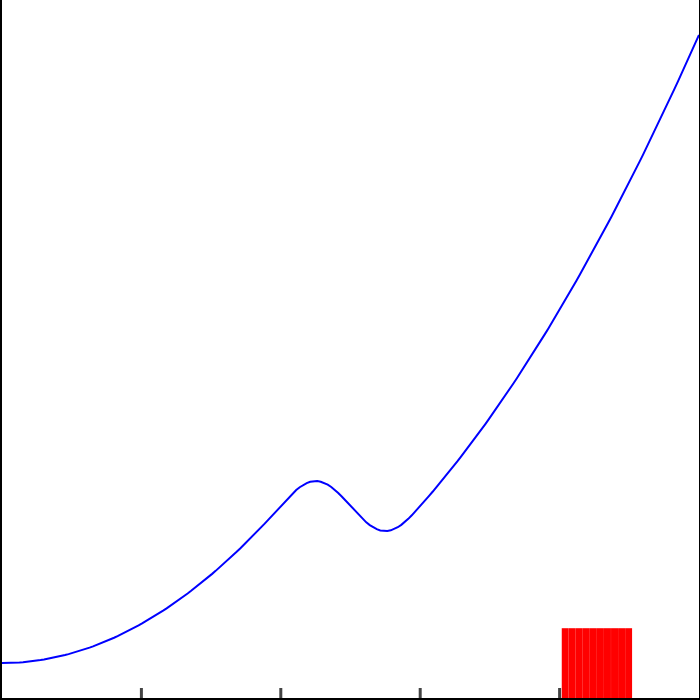} & \includegraphics[width=.22\textwidth]{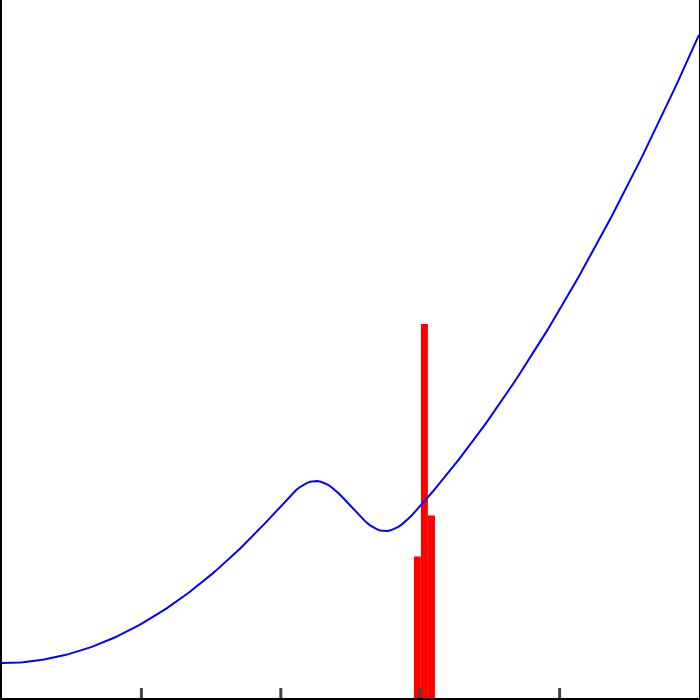} & \includegraphics[width=.22\textwidth]{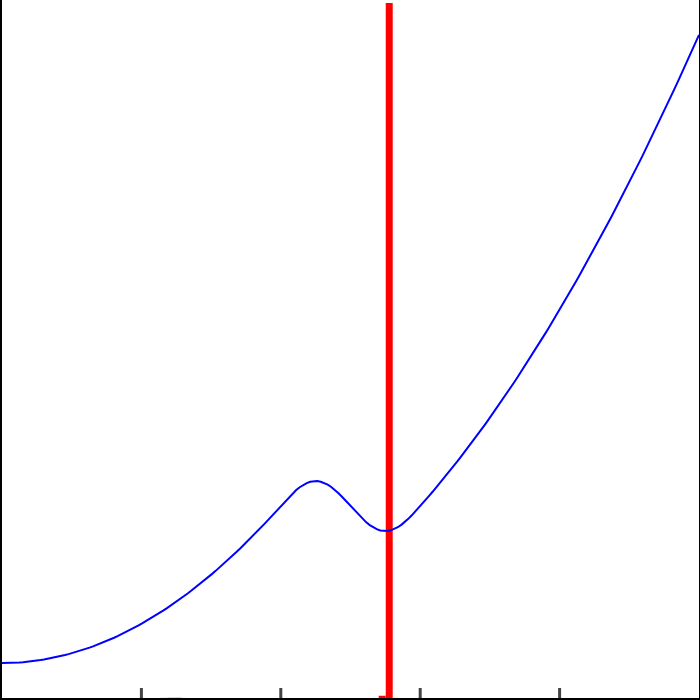} &
\includegraphics[width=.22\textwidth]{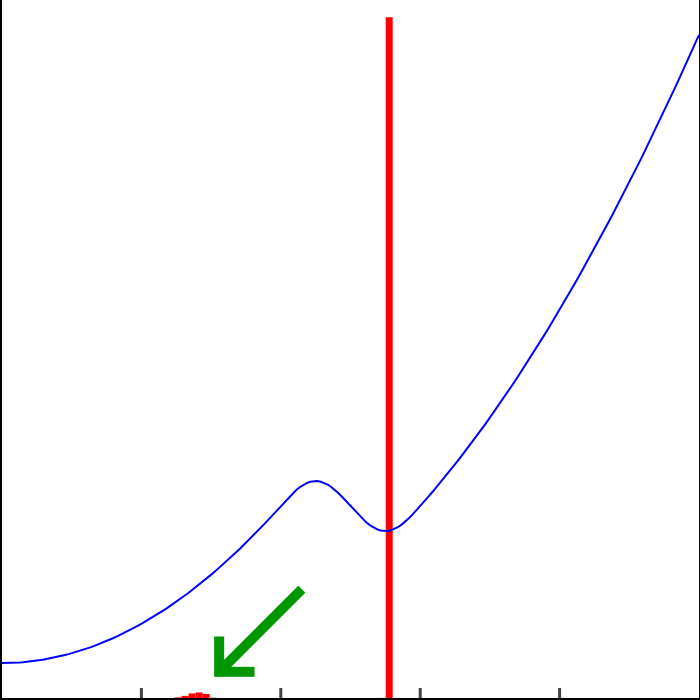}\\
\(t=0.00\) & \(t=0.25\) & \(t=0.40\) & \(t=0.60\)\\
\includegraphics[width=.22\textwidth]{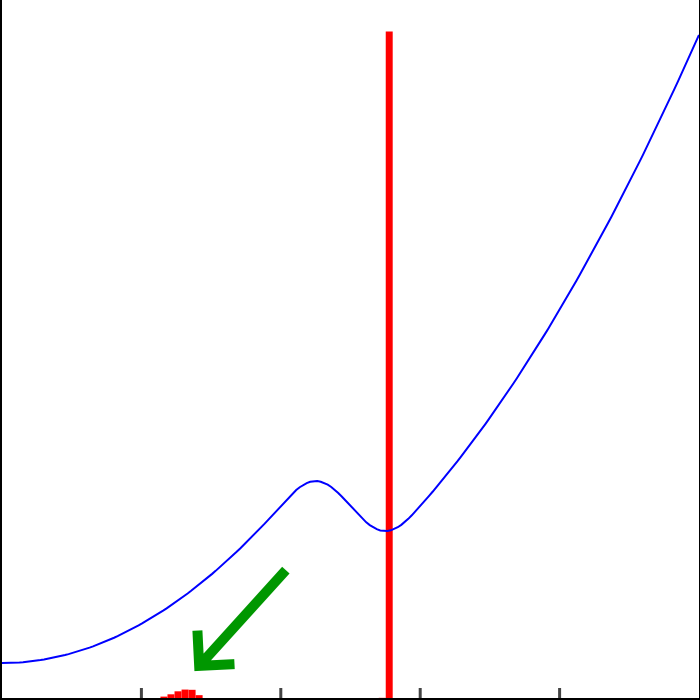} & \includegraphics[width=.22\textwidth]{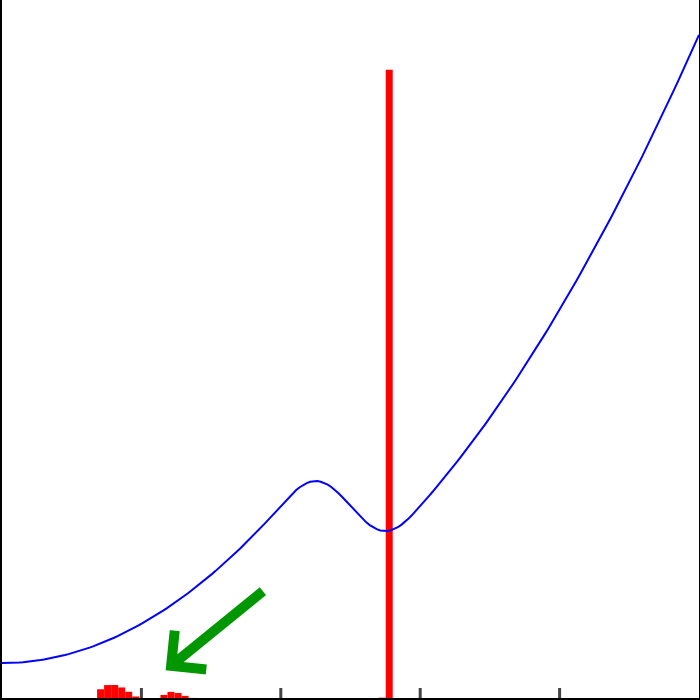} & \includegraphics[width=.22\textwidth]{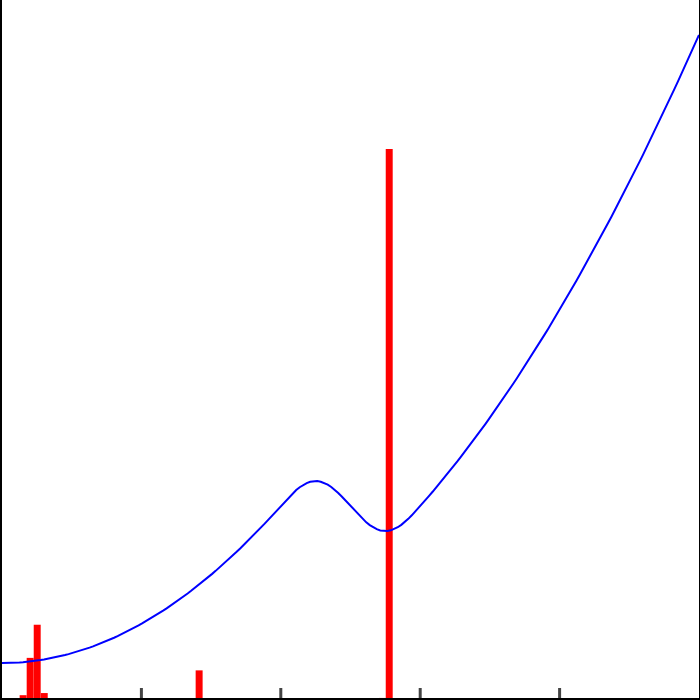}&
\includegraphics[width=.22\textwidth]{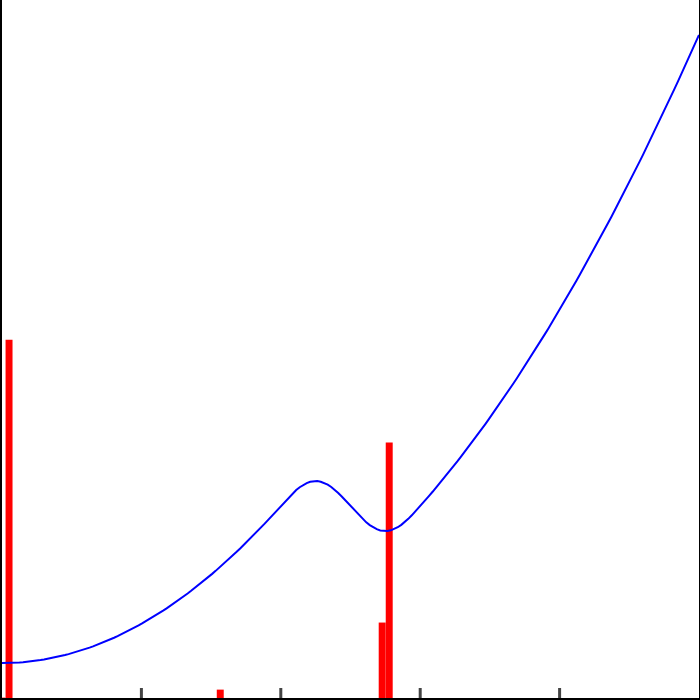}\\
\(t=0.75\) & \(t=1.25\) & \(t=2.50\) & \(t=7.35\)\\[.2cm]
\hline\\
\includegraphics[width=.22\textwidth]{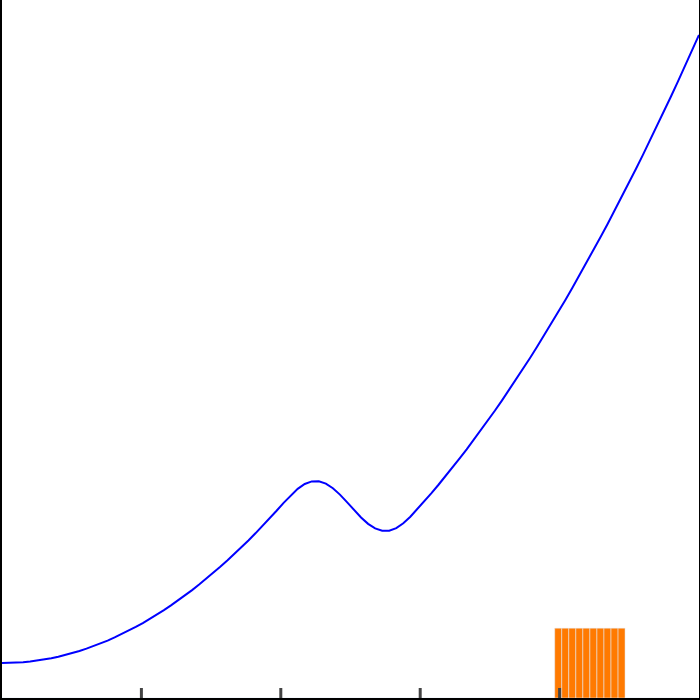} & \includegraphics[width=.22\textwidth]{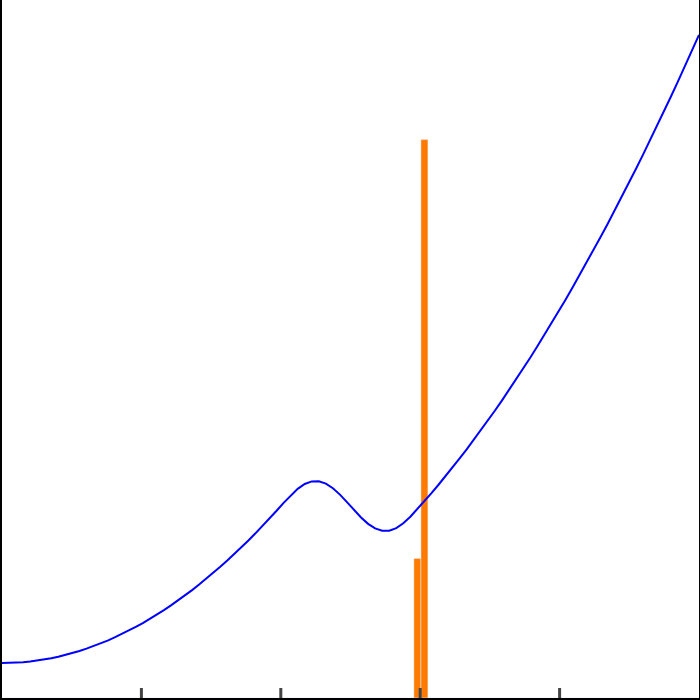} & \includegraphics[width=.22\textwidth]{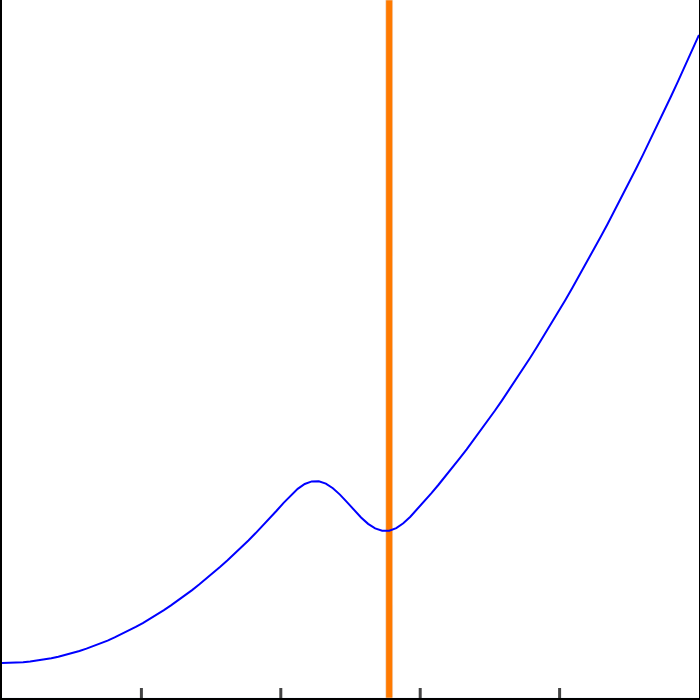} &
\includegraphics[width=.22\textwidth]{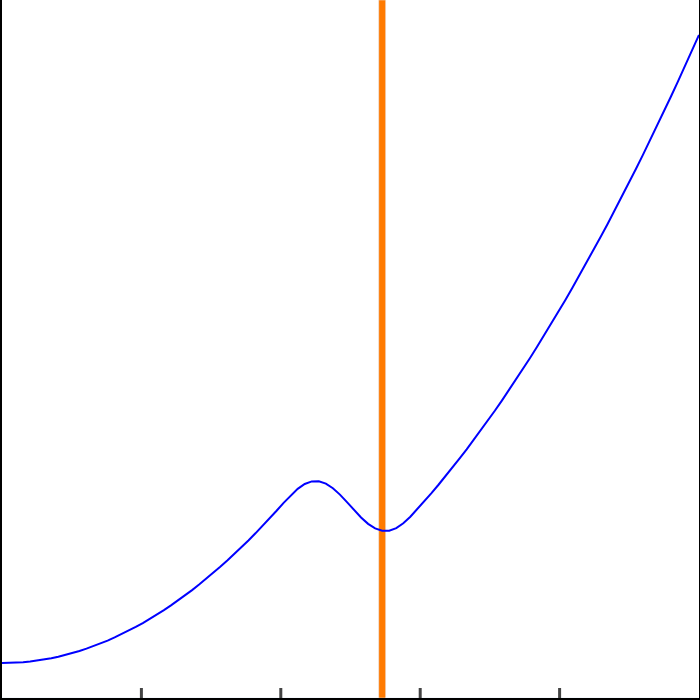}\\
\(t=0.00\) & \(t=0.25\) & \(t=1.25\) & \(t=7.35\)\\
\end{tabular}
\caption{Eulerian (red, top 2 rows) vs. Lagrangian (orange, bottom row) Sinkhorn potential flow for a non-convex potential (blue plot) with \(\varepsilon = 0.04\).}
\label{fig:eul_lag_ncvx}
\end{figure}

\begin{figure}
\centering
\includegraphics[width=.5\textwidth]{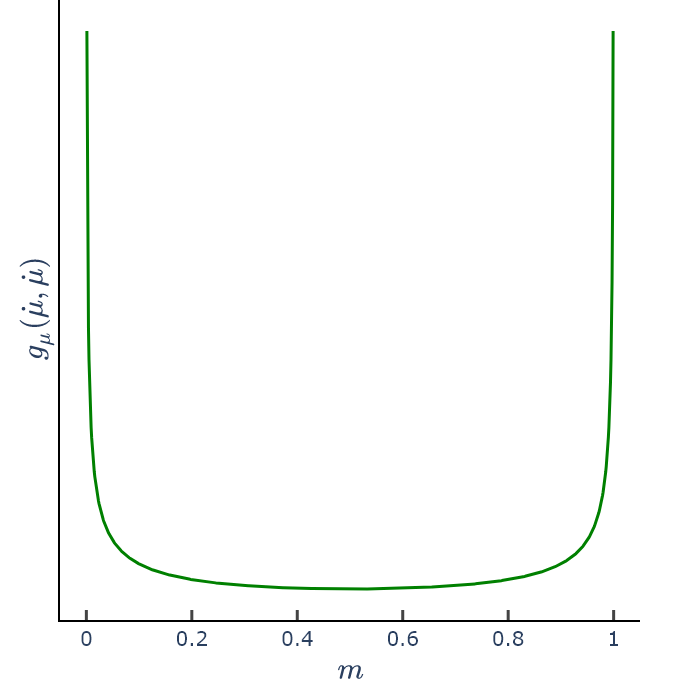}
\caption{Cost for the Sinhkorn metric tensor of a vertical perturbation \(\Dot \mu = (\delta_{x_2}-\delta_{x_1})\) of \(\mu = m\delta_{x_1} + (1-m)\delta_{x_2}\) as function of the mass \(m\). This is computed using the formula from \cite[Proposition 6.3]{RGSD}.}
\label{fig:vertical_perturbation_cost}
\end{figure}

\section{Conclusion and outlook}

In this work we started with a simple observation: if one wants to introduce entropic optimal transport in the scheme~\eqref{eq:JKO} without having to take the limit $\varepsilon \to 0$, then it is better to do it with the Sinkhorn divergence $S_\varepsilon$ rather than with $\mathrm{OT}_\varepsilon$. It gave us the definition of the scheme~\eqref{eq:SJKO}. 

At this point we followed where the mathematics lead us: understanding the formal limit of~\eqref{eq:SJKO} when $\tau \to 0$, i.e. the Sinkhorn potential flow~\eqref{eq:sinflow_mu}, required us to study the fine properties of this equation which is \emph{a priori} very unstable. We discovered in the end a very surprising structure ``linear skew-symmetric $+$ subdifferential of a convex function'' in the flat geometry of a Hilbert space, and a very intriguing behavior: the flow sometimes behaves like the Wasserstein one (\cref{prop:flow_dirac}), but sometimes teleports mass to escape local minima. This was observed in our numerical experiments, which show that only a Eulerian discretization, solving iteratively~\eqref{eq:SJKO}, can approximate the flow. 


\bigskip

Several questions are left open, beyond the generalization of \cref{thm:tauto0} to non-finite spaces $\X$. We mention them here, but we think they would require their own dedicated work.

\emph{Beyond potential energies.}
We only looked at $E(\mu) = \int V \, \dd \mu$ a potential energy. While the derivation of the formal equation of the flow in more general cases is not difficult (\cref{rm:other_E}), proving the well-posedness appears to be much more delicate: the structure ``linear skew-symmetric $+$ subdifferential of a convex function'' of the evolution~\eqref{eq:b_flat} in the $b$ variable looks very specific to potential energies. Any insight on the existence or long-term behavior of the flow in case of $E$ being the entropy or an internal energy is still missing. 

\emph{Limit $\varepsilon \to 0$.}
All our analysis is done for a fixed $\varepsilon$. In the limit $\varepsilon \to 0$, as $S_\varepsilon \to \mathrm{OT}_0$ we expect the flow to converge to the classical Wasserstein gradient flow, at least if $c$ is the quadratic cost. However proving rigorously such a result seems delicate: already the convergence $\varepsilon \to 0$ of the metric tensor $\mathbf{g}_\mu$ to the one of optimal transport is not rigorously done in~\cite[Section 4.4]{RGSD} because of technical difficulties.  

\emph{Applications.} 
Wasserstein gradient flows have numerous applications in pure and applied mathematics as recalled in the introduction. We expect our flow to be useful in some tasks of machine learning such as optimization or sampling: the Lagrangian discretization of the flow would lead to new interacting particle methods. 

\section*{Acknowledgments}

The research leading to this work started during a research stay of MH at the Department of Decision Sciences and BIDSA research center of Bocconi University, whose hospitality is warmly acknowledged. MH additionally acknowledges support by the French National Research Agency (ANR) under grant ANR-24-CE23-7711 ``Theory and Applications of the Geometry of Entropy Regularized Optimal Transport (TheATRE)'' during the final stages of this work. HL also acknowledges the support of the MUR-Prin 2022-202244A7YL ``Gradient Flows and Non-Smooth Geometric Structures with Applications to Optimization and Machine Learning'', funded by the European Union - Next Generation EU.

\appendix
\addtocontents{toc}{\vspace{\baselineskip}} 
\renewcommand{\thesection}{\Alph{section}}
\section{Background on Reproducing Kernel Hilbert Spaces}\label[appendix]{appendix:RKHS}

A function \(k\in \mcal C(\X\times \X)\) is a positive definite (PD) kernel if for all \(n \geq 1\), \((x_i)_{i=1}^n \subseteq \mcal X\), \((\lambda_i)_{i=1}^n\subset \R\), we have \(\sum_{i,j}\lambda_i\lambda_jk(x_i,x_j)\geq 0\). It is known \cite[Theorem 3]{berlinet2011rkhs} that for every such kernel there exists a unique RKHS \((\mcal H_k, \lrangle{\cdot, \cdot}_{\mcal H_k})\), i.e. a Hilbert space of continuous functions containing \(k(x, \cdot)\) for any \(x\in \X\) and verifying the reproducing kernel property \(\forall f\in \mcal H_k, \forall x\in \X, \lrangle{k(x, \cdot), f}_{\mcal H_k} = f(x)\). It is constructed as the completion of the pre-Hilbert space \(\vspan\{k(x, \cdot) \ : \ x\in\X\}\) with the inner product
\begin{equation*}
\lrangle{\sum_{i=1}^m\lambda_i k(x_i,\cdot), \sum_{j=1}^n\xi_jk(y_j,\cdot)}_{\mcal H_k} \coloneqq \sum_{i=1}^m\sum_{j=1}^n\lambda_i\xi_jk(x_i, y_j)
\end{equation*}
for \(m, n \geq 1\), \((\lambda_i)_{i=1}^m, (\xi_j)_{j=1}^n\subset \R\), \((x_i)_{i=1}^m, (y_j)_{j=1}^n \subseteq \mcal X\). 

The kernel \(k\) is further said to be universal if \(\mcal H_k\) is dense in \((\CX, \norm{\cdot}_{\infty})\), which is assumed in the sequel.

\begin{lemma}[{{\cite[Lemma B.2]{RGSD}}}]\label{lemma:infnorm<Hcnorm}
Weak convergence in \(\mcal H_k\) implies uniform convergence in $\CX$.
\end{lemma}

\begin{lemma}\label{lemma:Hkc0}
Denote \(H_k: \mcal H_k^* \to \mcal H_k\) the isometry mapping a linear functional to its representant given by the Riesz theorem. Recalling \(\MX\) is endowed with its weak-\txtstar{} topology as dual of \(\CX\), the following hold:
\begin{thmenum}
\item The restriction \(H_k:\MX \to \mcal H_k\) is weak-\txtstar{}-to-norm continuous. \label{lemma:Hkc0:Hk}
\item If $\mcal Q$ is a bounded subset of $\MX$, then the operator \(H_k\inv:H_k\sqbrac{\mcal Q} \to \mcal \MX\) is norm-to-weak-\txtstar{} continuous. \label{lemma:Hkc0:Hkinv}
\end{thmenum}
\end{lemma}

\begin{proof}
For \ref{lemma:Hkc0:Hk}, see \cite[Lemma B.2]{RGSD}. For \ref{lemma:Hkc0:Hkinv}, the Riesz-Fréchet theorem \cite[Theorem 5.5]{brezis} gives that \(H_k\inv: \mcal H_k \to \mcal H_k^*\) is an isometry and thus norm-to-norm continuous. Since \(k\) is universal, convergence in \(\mcal H_k^*\) implies weak-\txtstar{} convergence of measures when restricting to bounded subsets of \(\MX\) (proven on \(\mcal{P(X)}\) in \cite[Theorem 23]{sriperumbudur10}, the proof still works on bounded subsets of \(\MX\)). This yields the result.
\end{proof}

\section{Absolutely continuous curves in metric and Hilbert spaces}\label{appendix:H1}

Recall first that in a metric space $(\mcal Y, \mathsf{d}_{\mcal Y})$, if $I$ is an interval of $\R$, a continuous curve $(y_t)_{t \in I} \in \mcal C(I; \mcal Y)$ is said to be ($2$-)absolutely continuous \cite[Definition 1.1.1]{AGS}, written $y \in \mathrm{AC}^2(I; \mcal Y)$, if there exists a function $m \in L^2(I; \R)$ such that for all $t_0 \leq t_1$,
\begin{equation*}
\mathsf{d}_{\mcal Y}(y_{t_0},y_{t_1}) \leq \int_{t_0}^{t_1} m(t) \dd t. 
\end{equation*}
We say that the curve is locally absolutely continuous, written $y \in \mathrm{AC}^2_\text{loc}(I; \mcal Y)$, if $y \in \mathrm{AC}^2(J; \mcal Y)$ for any compact interval $J$ included in $I$.   

Next we focus on Hilbert-valued curves. Let \(\mcal H\) be a Hilbert space and $I$ an interval. A curve \((x_t)_t\in L^2(I;\mcal H)\) is said to be distributionally differentiable if there exists a curve \((v_t)_t\in L^2(I; \mcal H)\), called distributional derivative and denoted \(\Dot x\), such that for all $\phi \in \mcal C^\infty(I, \R)$ compactly supported in the interior of $I$, as Bochner integrals 
\begin{equation*}
\int_I \Dot\phi_tx_t\dd t = -\int_I \phi_tv_t\dd t.
\end{equation*}
The set of such curves is the first Sobolev space, denoted \(\mathscr H^1(I;\mcal H)\). We next state a few useful results.

\begin{theorem}[{{\cite[Propositions 3.7 and 3.8, Theorem 3.13]{kreuter15}}}]\label{thm:H1}
Let \((x_t)_t\in L^2(I; \mcal H)\). The following are equivalent:
\begin{thmenum}
\item \(x \in \mathscr H^1(I; \mcal H)\).
\item \(x\in\mathrm{AC}^2(I; \mcal H)\). \label{thm:H1:AC2}
\item \(x\) is a.e. differentiable, \(\Dot x\in L^2(I;\mcal H)\), and there exists \(t_0\in I\) such that \(x_t = x_{t_0} + \int_{t_0}^t\Dot x_s\dd s\) for all $t \in I$.\label{thm:H1:int}
\end{thmenum}
\end{theorem}

We recall a classical compactness result in $\mathscr H^1(I; \mcal H)$. 

\begin{theorem}
\label{thm:compactness_H1}
For a compact interval $I$ and a Hilbert space $\mcal H$, let $(x^n)_n$ be a sequence in $\mathscr H^1(I; \mcal H)$ such that:
\begin{enumerate}
\item There exists a compact $\mcal B \subset \mcal H$ with $x_t^n \in \mcal B$ for all $n$ and $t \in I$; 
\item We have the uniform estimate $\displaystyle{\sup_n \| \dot{x}^n \|_{L^2(I;\mcal H)} < + \infty}$.
\end{enumerate}
Then, up to extraction of a subsequence, the sequence $(x^n)_{n}$ converges uniformly in $\mcal C (I;\mcal H)$ to a limit $x \in \mathscr H^1(I, \mcal H)$, and the sequence $\dot{x}^n$ converges weakly in $L^2(I;\mcal H)$ to $\dot{x}$.  
\end{theorem}

\begin{proof}
Using \refthmitem{thm:H1}{int} and Cauchy-Schwarz (in \(L^2(I; \R)\)), we can obtain Morrey's inequality: for all \(t_0 \leq t_1\) with both in \(I\),
\begin{equation}
\label{eq:morrey}
\norm{x^n_{t_1} - x^n_{t_0}}_{\mcal H} \leq \int_{t_0}^{t_1}\norm{\Dot x^n_t}_{\mcal H}\dd t \leq \norm{\Dot x^n}_{L^2\brac{I; \mathcal H}} \sqrt{\abs{t_0-t_1}}.
\end{equation}
Thus the family $(x^n)_n$ in $\mcal C (I;\mcal H)$ is equi-continuous. As it is valued in the compact space $\mcal B$ we have the required assumptions for the Arzelà-Ascoli theorem \cite[Theorem 7.5.7]{dieudonneanalysis}
which yields that $x^n$ converges, up to extraction, uniformly in $\mcal C (I,\mcal H)$, to a limit \(x\).
Moreover, as $(\dot{x}^n)_n$ is uniformly bounded in the Hilbert space $L^2(I; \R)$, from the Banach-Alaoglu theorem \cite[Theorem 3.16]{brezis}, up to a further extraction, it converges weakly in \(L^2(I; \mcal H_c)\) to a limit denoted \(y\). As uniform convergence of $(x^n)_n$ implies distributional convergence, we conclude that $y = \dot{x}$.  
\end{proof}

\printbibliography

\end{document}